\documentclass[11pt,reqno]{amsart}

\usepackage{amsmath, amsthm, amssymb}
\usepackage{amsfonts}
\usepackage{mathtools}

\usepackage[ansinew]{inputenc}
\usepackage[dvips]{epsfig}
\usepackage{graphicx}
\usepackage[english]{babel}
\usepackage{enumitem}

\usepackage{stackengine}
\stackMath

\usepackage{aliascnt}

\theoremstyle{plain}
\newtheorem{theorem}{Theorem}[section]

\newaliascnt{lemma}{theorem}
\newtheorem{lemma}[lemma]{Lemma}
\aliascntresetthe{lemma}

\newaliascnt{proposition}{theorem}
\newtheorem{proposition}[proposition]{Proposition}
\aliascntresetthe{proposition}

\newaliascnt{corollary}{theorem}

\aliascntresetthe{corollary}

\theoremstyle{definition}
\newaliascnt{definition}{theorem}
\newtheorem{definition}[definition]{Definition}
\aliascntresetthe{definition}

\newaliascnt{remark}{theorem}

\aliascntresetthe{remark}

\newaliascnt{example}{theorem}

\aliascntresetthe{example}

\theoremstyle{definition}
\newtheorem*{remark*}{Remark}
\newtheorem*{assumption*}{Assumption}

\numberwithin{equation}{section}

\usepackage[backgroundcolor=white, bordercolor=blue,
linecolor=blue]{todonotes}

\usepackage{hyperref}
\usepackage{cleveref}

\usepackage{dsfont}
\usepackage{bbm}

\newcommand{\N}{\mathbb{N}}
\newcommand{\R}{\mathbb{R}}

\newcommand{\Z}{\mathbb{Z}}

\newcommand{\eps}{\epsilon}

\newcommand{\bbR}{\mathbb{R}}

\DeclareMathOperator{\dist}{dist}

\DeclareMathOperator{\supp}{supp}

\DeclareMathOperator*{\osc}{osc}

\DeclareMathOperator{\tail}{Tail}

\newcommand{\average}{{\mathchoice {\kern1ex\vcenter{\hrule height.4pt
width 6pt depth0pt} \kern-9.7pt} {\kern1ex\vcenter{\hrule
height.4pt width 4.3pt depth0pt} \kern-7pt} {} {} }}
\newcommand{\dashint}{\average\int}

\begin{document}
\allowdisplaybreaks
\title[Quantitative homogenization and uniform regularity]{Quantitative periodic homogenization and uniform regularity for nonlocal equations}

\author{Kyeongbae Kim}
 \address{Institute for Applied Mathematics, University of Bonn, Endenicher Allee 60, 53115, Bonn, Germany}
 \email{kkim@uni-bonn.de}
 \urladdr{https://sites.google.com/view/kyeongbaekim/}

\author{Simon Nowak}
\address{Fakult\"at f\"ur Mathematik, Universit\"at Bielefeld, Postfach 100131, D-33501 Bielefeld, Germany}
\email{simon.nowak@uni-bielefeld.de}

\author{Marvin Weidner}
 \address{Institute for Applied Mathematics, University of Bonn, Endenicher Allee 60, 53115, Bonn, Germany}
 \email{mweidner@uni-bonn.de}
 \urladdr{https://sites.google.com/view/marvinweidner/}



\allowdisplaybreaks

\begin{abstract}
In this work, we study quantitative homogenization for a class of nonlocal equations with periodically oscillating H\"older continuous coefficients. For Dirichlet problems in bounded $C^{1,\alpha}$ domains, we prove $L^2$ convergence rates that are expected to be optimal. In addition, we establish interior Lipschitz estimates and optimal boundary H\"older estimates for weak solutions, with constants independent of the period of oscillation.
\end{abstract}

\maketitle

\section{Introduction}

Quantitative homogenization for local second-order elliptic equations has been studied extensively over the past decades. In the periodic setting, the pioneering work of Avellaneda and Lin \cite{AL87} introduced a compactness method for proving regularity estimates with constants independent of the period of the coefficients. Subsequent work of Kenig, Lin and Shen \cite{KLS12,KLS13,KLS14} established convergence rates in bounded domains, uniform boundary estimates and asymptotic expansions of Green functions. We refer to the monograph of Shen \cite{Shen18} for a systematic account of the periodic theory. 

More recently, a quantitative theory has also been developed in the stochastic setting. We mention in particular the optimal variance and corrector estimates of Gloria, Neukamm and Otto \cite{GO11,GNO15}, and the corresponding large scale regularity theory of the mentioned authors \cite{GNO20}. A variational approach to quantitative homogenization and large scale regularity was independently developed by Armstrong and Smart \cite{AS16} and Armstrong, Kuusi and Mourrat \cite{AKM16,AKM17}. For detailed treatments of quantitative stochastic homogenization, we refer to the monographs \cite{AKM19,AK24}.

For nonlocal operators related to the fractional Laplacian, quantitative homogenization is far less understood already in the periodic setting, especially in bounded domains. A crucial difficulty is that even the homogenized solution generally has singular derivatives near the boundary. Indeed, the natural boundary behavior corresponds to $d_\Omega^s$, where $d_\Omega$ denotes the distance to the boundary. Despite this singular boundary behavior, the aim of this paper is to establish $L^2$ convergence rates for periodic nonlocal Dirichlet problems in bounded domains that match the known rates on the whole space, together with interior and boundary regularity estimates that are uniform in the period of oscillation.

\subsection{Quantitative homogenization of nonlocal Dirichlet problems}
To be more precise, for $s\in(0,1)$ and $\eps\in(0,1]$, we consider nonlocal operators formally given by
\begin{equation}\label{intro:operator}
L_\eps u(x)\coloneqq \mathrm{P.V.}\int_{\bbR^n}
\frac{u(x)-u(y)}{|x-y|^{n+2s}}K(x/\eps,y/\eps)\,dy,
\end{equation}
where $n>2s$. The coefficient $K:\bbR^n\times\bbR^n\to\bbR$ is symmetric, uniformly elliptic, and periodic in each variable, in the sense that
\begin{equation}\label{intro:kernel}
\begin{gathered}
K(\xi,\eta)=K(\eta,\xi),\qquad
\Lambda^{-1}\leq K(\xi,\eta)\leq\Lambda,\\
K(\xi+k,\eta+\ell)=K(\xi,\eta)
\qquad\text{for all }k,\ell\in\mathbb Z^n
\end{gathered}
\end{equation}
for some $\Lambda\geq1$. Thus we identify $K$ with a function on $\mathbb T^n\times\mathbb T^n$, where $\mathbb T^n=(-\frac12,\frac12]^n$. Moreover, we say that $K \in C^\alpha(\mathbb T^n)$ for some $\alpha \in (0,1)$ if 
\begin{align*}
    [K]_{C^\alpha(\mathbb T^n)} := \sup_{\xi,\eta,h \in \mathbb T^n} \left( \frac{|K(\xi+h,\eta+h) - K(\xi,\eta)|
    }{|h|^\alpha} \right) < \infty. 
\end{align*}

The homogenized operator is a constant multiple of the fractional Laplacian, namely
\begin{equation}\label{intro:homogenized}
\overline L u(x)\coloneqq\overline K\,\mathrm{P.V.}\int_{\bbR^n}
\frac{u(x)-u(y)}{|x-y|^{n+2s}}\,dy,
\qquad
\overline K\coloneqq\int_{\mathbb T^n}\int_{\mathbb T^n}K(\xi,\eta)\,d\eta\,d\xi.
\end{equation}
The identification of this homogenized operator is due to Kassmann, Piatnitski and Zhizhina \cite{KPZ19}. In contrast to their paper, the main goal of this work is to establish quantitative homogenization results with optimal convergence rates.
To do so, we study the Dirichlet problems
\begin{equation}\label{intro:dirichlet}
\left\{\begin{aligned}
L_\eps u_\eps&=h&&\quad\text{in }\Omega,\\
u_\eps&=g&&\quad\text{in }\bbR^n\setminus\Omega,
\end{aligned}\right.
\qquad
\left\{\begin{aligned}
\overline L\overline u&=h&&\quad\text{in }\Omega,\\
\overline u&=g&&\quad\text{in }\bbR^n\setminus\Omega,
\end{aligned}\right.
\end{equation}
where $\Omega\subset\bbR^n$ is a bounded domain. In contrast to local Dirichlet problems, the datum $g$ is prescribed on the whole complement of $\Omega$.

We are now in a position to state our first main result. For the precise notion of weak solutions, we refer to Section \ref{secprelim}.

\begin{theorem}[Quantitative homogenization]\label{thm:intro:homogenization}
Let $s\in(0,1)$, $\alpha\in(0,1)$ and $\varepsilon\in(0,1]$.
Assume that $\Omega$ is a bounded $C^{1,\alpha}$ domain and that
$K\in C^\alpha(\mathbb T^n)$ satisfies
\eqref{intro:kernel}. Let $u_\varepsilon$ and $\overline u$ be weak
solutions of \eqref{intro:dirichlet} with
$h\in C^{1,\alpha}(\overline\Omega)$ and
$g\in W^{1,\infty}(\mathbb R^n)$. Then we have
\begin{equation}\label{intro:rate}
\begin{aligned}
\|u_\varepsilon-\overline u\|_{L^2(\Omega)}
&\leq c\Bigl(
\|h\|_{C^{1,\alpha}(\overline\Omega)}
+\|g\|_{W^{1,\infty}(\mathbb R^n)}
\Bigr)
\begin{cases}
\varepsilon^{2s},
& s<\frac12,\\
\varepsilon(1+|\ln\varepsilon|)^{3/2},
& s=\frac12,\\
\varepsilon^{2-2s},
& s>\frac12,
\end{cases}
\end{aligned}
\end{equation}
for some constant
$c=c(n,s,\Lambda,\alpha,
[K]_{C^\alpha(\mathbb T^n)},\Omega)$.
\end{theorem}


The powers of $\eps$ in \eqref{intro:rate} coincide with those obtained on the whole space by Piatnitski, Sloushch, Suslina and Zhizhina \cite{PSSZ24}. Their proof relies on spectral methods and does not appear to readily extend to the case of bounded domains. For problems in bounded domains, Chen, Chen, Kumagai and Wang \cite{CCKW24} proved $L^2$ estimates with the same powers of $\eps$ as in \eqref{intro:rate}, but only under the additional assumption that $\overline u\in C_c^2(\Omega)$. This assumption is highly restrictive, as it excludes the natural boundary behavior of solutions to the nonlocal Dirichlet problem \eqref{intro:dirichlet}, which in general have only $C^s(\overline\Omega)$ regularity. Without this assumption, in \cite{CCKW24} only suboptimal rates in $L^1$ are established. Our  \autoref{thm:intro:homogenization} removes this restriction and shows that the natural boundary singularity of $\overline u$ does not cause any loss in the convergence rates, since we obtain the improved rates in $L^2$ as in the whole space, without imposing any additional regularity assumptions such as $\overline u\in C_c^2(\Omega)$. Moreover, we would like to point out that we can establish $L^1$ convergence rates of the form $\eps(1+|\ln\eps|)$ in case $s \in (0,\frac{1}{2}]$ and $h \equiv 0$; see \autoref{conv.sleq12}.

Moreover, in comparison to \cite{CCKW24} we weaken the smoothness assumptions on the right-hand side $h$, allow for nonzero exterior data $g$, and require merely H\"older continuity of the coefficient $K$, rather than $C^2$ regularity as in \cite{CCKW24}. The latter improvement is particularly relevant, as the assumption $K \in C^2$ allows the authors in \cite{CCKW24} to rewrite the equation in nondivergence form and to evaluate the nonlocal operator in a pointwise sense.  This approach is no longer available under our weaker assumptions. Instead, we must work directly on the level of weak solutions, which is considerably more involved, especially since we need to deal with the double integrals arising in the weak formulation of the problem. A key feature of our approach is the use of large-scale regularity estimates, which enable us to overcome these difficulties and establish optimal convergence rates (see more details below). In an upcoming (AI-assisted) preprint, we will establish \autoref{thm:intro:homogenization} without any regularity assumption on $K$ via a completely different method that does not use any correctors.

\subsection{Uniform regularity for nonlocal equations}
Our second main result concerns regularity estimates for solutions to \eqref{intro:dirichlet} that are \emph{uniform in the oscillation period $\eps$}. For second-order elliptic equations, such estimates go back to Avellaneda and Lin \cite{AL87}; see also \cite{KLS13,Shen18} for the periodic theory and \cite{AS16,AKM16,GNO20,AKM19,AK24} for large scale regularity in stochastic homogenization. The underlying principle is that, at scales larger than the oscillation period, solutions can be approximated by solutions of the homogenized equation, allowing for a transfer of regularity to such large scales. Combined with standard regularity estimates below this scale, this yields regularity bounds on all scales.

For nonlocal equations, an extensive regularity theory has been developed over the past decades. For equations with measurable coefficients, Harnack inequalities and H\"older estimates were established for example in the works \cite{BL02,Sil06,Kas09,DKP14,DKP16}. Higher interior regularity estimates across small scales under additional continuity assumptions on the coefficients were obtained for instance in \cite{Fal20,Now21,KNS22,FeR024S,DFMN26}.
At the boundary, optimal $C^s$ estimates for the fractional Laplacian and related estimates for translation-invariant operators of order $2s$ were established in \cite{ROS14,ROS17,RS16,RW24}. Moreover, higher order boundary regularity was obtained in \cite{AR20,Gru15}. For nonlocal equations with H\"older coefficients, optimal boundary regularity and sharp Green function estimates were recently established in \cite{KW25}. Boundary estimates under slightly weaker assumptions on the coefficients were obtained in \cite{BKK25}.

The aforementioned higher regularity estimates for nonlocal equations with variable coefficients depend on the modulus of continuity of the coefficients. Applying them to \eqref{intro:operator} therefore does not give bounds that are uniform in $\eps$, since the H\"older seminorm of $K(\cdot/\eps,\cdot/\eps)$ grows like $\eps^{-\delta}$. By invoking homogenization techniques, we prove that periodicity of $K$ allows us to obtain both interior Lipschitz estimates and boundary $C^s$ estimates with constants independent of $\eps$.

In order to formulate the estimates, we recall the nonlocal tail and the nonlocal excess given by
\begin{equation}\label{intro:tail}
\begin{aligned}
\tail(u;B_R(x_0))&\coloneqq R^{2s}\int_{\bbR^n\setminus B_R(x_0)}
\frac{|u(y)|}{|x_0-y|^{n+2s}}\,dy,\\
E(u;B_R(x_0))&\coloneqq
\dashint_{B_R(x_0)}|u|\,dx
+\tail(u;B_R(x_0)).
\end{aligned}
\end{equation}

\begin{samepage}
\begin{theorem}[Uniform interior and boundary regularity]\label{thm:intro:regularity}
Let $s\in(0,1)$, $\alpha \in(0,1)$ and $\eps\in(0,1]$. Suppose that $K\in C^\alpha(\mathbb T^n)$ satisfies \eqref{intro:kernel} and that $\Omega$ is a $C^{1,\alpha}$ domain. Let $x_0\in\overline\Omega$, $0<R\leq1$ and let $u_\eps\in H^s(B_R(x_0))\cap L^1_{2s}(\bbR^n)$ be a weak solution of
 \begin{equation*}
\left\{
\begin{alignedat}{3}
L_\eps u_\eps&=f&&\qquad \mbox{in  $\Omega\cap B_R(x_0)$}, \\
u_\eps&=0&&\qquad  \mbox{in $B_R(x_0)\setminus \Omega$},
\end{alignedat} \right.
\end{equation*}
where $f\in L^\infty(\Omega\cap B_R(x_0))$. Then
\begin{equation}\label{intro:boundary}
[u_\eps]_{C^s(B_{R/2}(x_0))}
\leq cR^{-s}\left(E(u_\eps;B_R(x_0))
+R^{2s}\|f\|_{L^\infty(\Omega\cap B_R(x_0))}\right),
\end{equation}
where $c$ depends only on $n,s,\Lambda,\alpha$, $[K]_{C^\alpha(\mathbb T^n)}$ and on $\Omega$, but not on $\eps$.

Moreover, if $s>\frac12$ and $B_R(x_0)\subset\Omega$, then
\begin{equation}\label{intro:gradient}
\|\nabla u_\eps\|_{L^\infty(B_{R/2}(x_0))}
\leq cR^{-1}\left(E(u_\eps;B_R(x_0))
+R^{2s}\|f\|_{L^\infty(B_R(x_0))}\right),
\end{equation}
where $c$ depends only on $n,s,\Lambda,\alpha$ and $[K]_{C^\alpha(\mathbb T^n)}$, but not on $\eps$.
\end{theorem}

\end{samepage}
To the best of our knowledge, \eqref{intro:boundary} and \eqref{intro:gradient} are the first boundary $C^s$ and interior Lipschitz estimates, respectively, that are uniform in the oscillation period for the class of nonlocal divergence form operators considered here. The boundary exponent $s$ cannot be improved, as is already seen from the boundary behavior of solutions to the Dirichlet problem for the fractional Laplacian. The restriction $s>\frac12$ in \eqref{intro:gradient} is also necessary already for the fractional Laplacian with bounded right-hand sides.

We remark that the estimates at scales larger than the period require no continuity of the coefficients. Indeed, for merely measurable periodic $K$, we prove interior oscillation decay with every exponent $\beta<\min\{2s,1\}$, including $\beta=1$ when $s>\frac12$, and boundary decay with every exponent $\beta<s$, down to scales comparable to $\eps$, see \autoref{lem.large.int} and \autoref{lem.large.bdry}. In other words, the H\"older assumption on $K$ in \autoref{thm:intro:regularity} is used primarily to include smaller scales.

Moreover, as an application of \autoref{thm:intro:regularity}, we obtain upper bounds for the Green function and the Poisson kernel associated to $L_\epsilon$ and $\Omega$ with constants independent of $\eps$; see \autoref{thm.green.C^s}. We also prove a Liouville-type theorem, classifying all global solutions to $L_\eps u_\eps = 0$ with suitable growth at infinity; see \autoref{thm.liou}. In this respect, our results provide nonlocal counterparts of key aspects of the classical theory by Avellaneda--Lin \cite{AL87}.

\subsection{Further comparison with previous results}
Qualitative homogenization for nonlocal equations has been studied in several settings. Arisawa \cite{Ari09} treated nonlocal integro-differential equations in nondivergence form, while Schwab \cite{Sch10,Sch13} considered periodic and stochastic homogenization for fully nonlinear integro-differential equations. In a related local setting, we refer to \cite{Eva92,HP08} for qualitative homogenization of second-order equations including degenerate cases. Homogenization of symmetric nonlocal operators in divergence form of the type \eqref{intro:homogenized} was established by Kassmann, Piatnitski and Zhizhina \cite{KPZ19}, and of stable-like operators in stationary ergodic media by Chen, Chen, Kumagai and Wang \cite{CCKW21}. The latter authors also obtained quantitative stochastic homogenization results for random conductance models with stable-like jumps \cite{CCKW25}.

Let us discuss in more detail the quantitative results closest to \autoref{thm:intro:homogenization}. For $C^2$ coefficients, bounded $C^{1,1}$ domains, suitably regular forcing and homogeneous exterior data, Chen, Chen, Kumagai and Wang \cite[Theorem~1.1]{CCKW24} proved $L^1$ convergence estimates with the weaker rates $\eps^s$ for $s<\frac12$ and $\eps^{1-s}$ for $s>\frac12$, and a logarithmic modification of $\eps^{1/2}$ at $s=\frac12$. The $L^2$ estimate under the compact support assumption on $\overline u$ discussed above is proved in \cite[Theorem~1.2]{CCKW24}.

The whole-space estimates of Piatnitski, Sloushch, Suslina and Zhizhina \cite{PSSZ24} give convergence of the resolvents in the $L^2\to L^2$ operator norm, with essentially the same rates that we obtain in \autoref{thm:intro:homogenization}. As mentioned, the methods applied in \cite{PSSZ24} seem to be global in nature.

Higher-order approximations on the whole space were recently constructed by Piatnitski, Sloushch, Suslina and Zhizhina \cite{PSSZ26} and Chen, Chen, Kumagai and Wang \cite{CCKW26}. These works improve the error by adding further corrector terms to the homogenized solution. In contrast, our error estimate \eqref{intro:rate} controls $u_\eps-\overline u$ directly in a bounded domain.

\subsection{Proof strategy}
\label{subsec:idea}

We briefly outline the main ideas behind the proofs of our main result \autoref{thm:intro:homogenization}. A key idea to obtain optimal convergence rates is to introduce two nonlocal correctors $\phi=(\phi_i)_{i=1}^n$ and $\psi$, defined as periodic weak solutions of
\begin{align*}
   L_1\phi_i(x)&=
\begin{cases}
\int_{\R^n}\frac{z_i}{|z|^{n+2s}}\bigl(\overline K - K(x,x-z)\bigr)\, dz,
& s>\frac12, \\
\int_{\eps^{-1}\mathbb T^n}\frac{z_i}{|z|^{n+2s}}\bigl(\overline K - K(x,x-z) \bigr)\, dz,
& s\leq\frac12,
\end{cases}\\
L_1\psi(x)&=-\overline K^{-1}\bigl(\overline K(x)-\overline K\bigr), \quad \text{where} \quad \overline{K}(x) = \int_{\mathbb{T}^n}K(x,x-z)\,dz.
\end{align*}
The first-order corrector $\phi$ compensates for the oscillations of the coefficients, while $\psi$ accounts for the long-range interactions of the nonlocal operator. Note that these objects were already introduced in \cite{CCKW24}. To prove optimal homogenization rates, a classical idea in homogenization theory is to establish suitable two-scale expansions of the solution $u_\eps$ to \eqref{intro:dirichlet} in terms of $\overline{u}$ and the corrector of the form
\begin{align*}
    u_\eps(x) = \overline{u}(x) + \eps \phi(x/\eps) \nabla \overline{u}(x) + \dots,
\end{align*}
which amounts to proving $L^2$ estimates of the error $\widetilde w_\eps := u_\eps - \big( \overline{u}+ \eps \phi(\cdot/\eps) \cdot\nabla \overline{u} \big)$. A key difficulty in the nonlocal setting comes from the lack of higher boundary regularity of solutions to the nonlocal Dirichlet problem. In fact, in general we only have that $\overline{u} \in C^s(\overline{\Omega})$, and therefore, it does not necessarily hold $\nabla \overline{u} \in H^s(\Omega)$. This means that $\widetilde w_\eps$ cannot be estimated via energy methods. To overcome this difficulty, we enlarge $\Omega$ to an $\eps$-tubular neighborhood $\Omega^\eps$ and consider the auxiliary problem
\begin{equation}\label{conv:eq.overbarueps}
\left\{
\begin{alignedat}{3}
\overline{L}\overline{u}_\eps&=h&&\qquad \mbox{in  $\Omega^{\eps}$}, \\
\overline{u}_\eps&=g&&\qquad  \mbox{in $\bbR^n\setminus \Omega^{\eps}$}.
\end{alignedat} \right.
\end{equation}
Since $\Omega$ lies at a distance comparable to $\eps$ from the boundary of $\Omega^\eps$, we have $\nabla\overline u_\eps\in H^s(\Omega)$. This motivates estimating the errors
\begin{equation*}
w_\eps\coloneqq u_\eps-\left(\overline u_\eps
+\eps\phi(\cdot/\eps)\cdot\nabla \overline u_\eps
+\eps^{2s}\psi(\cdot/\eps)h\right), \qquad b_\eps := \overline{u} - \overline{u}_\eps.
\end{equation*} 

The overall idea of replacing the homogenized solution by a solution in a slightly larger domain is inspired by \cite{KLS12},  who employed this construction in local elliptic homogenization to overcome the lack of $H^2$ regularity in Lipschitz domains. Note that in our nonlocal setting, the long-range interactions introduce additional difficulties in controlling the resulting  error terms. Another distinguishing feature of our approach is that we combine the corrector expansion with uniform Green function and Poisson kernel estimates of $L_\eps$, rather than the $L^2$ boundary-value estimates used in \cite{KLS12}, to estimate the errors $w_\eps$ and $b_\eps$.

We achieve this by first establishing large scale interior Lipschitz and boundary $C^\beta$ estimates for every $\beta < s$ by a compactness argument, inspired by Avellaneda--Lin \cite{AL87}; see \autoref{lem.large.int} and \autoref{lem.large.bdry}. Thereby, we prove a preliminary version of \autoref{thm:intro:regularity}. These estimates yield almost sharp bounds for the Green function and Poisson kernel associated with $L_\eps$, with constants independent of $\eps$; see Section \ref{sec:green}. Using these bounds to estimate the errors $w_\eps$ and $b_\eps$, ultimately yields the optimal convergence rates in \autoref{thm:intro:homogenization}, at least for $s \not= \frac12$, i.e. we obtain for any $\gamma \in (0,1)$,
 \begin{align}
 \label{eq:prelim-conv-intro}
     \|u_\eps-\overline{u}\|_{L^2(\Omega)}\lesssim 
        \eps^{2s}\mbox{\large$\chi$}_{\{s<\frac12\}}+\eps^{\gamma}\mbox{\large$\chi$}_{\{s=\frac12\}}+\eps^{2-2s}\mbox{\large$\chi$}_{\{s>1/2\}}.
 \end{align}
We refer the reader to \autoref{lem.conv} and the subsequent  discussion, where we explain in more details how $w_\eps$ and $b_\eps$ are estimated.

To obtain the optimal boundary $C^s$ regularity in \autoref{thm:intro:regularity}, we make use of the convergence rates in \eqref{eq:prelim-conv-intro} and combine them with sharp boundary regularity estimates of the homogenized equation. These estimates in turn yield sharp bounds for the Green function and Poisson kernel with the optimal boundary exponent $s$; see \autoref{thm.green.C^s}. Finally, by repeating the aforementioned procedure to obtain suitable two-scale expansions of $u_\eps$, and using this time the sharp bounds, we obtain the optimal convergence rate $\eps(1+|\ln\eps|)^{3/2}$ also in case $s=\frac12$, completing the proof of \autoref{thm:intro:homogenization}.

\subsection{Acknowledgments}
Kyeongbae Kim and Marvin Weidner were supported by the Deutsche Forschungsgemeinschaft (DFG,
German Research Foundation) under Germany's Excellence Strategy - EXC-2047/1 - 390685813 and through the CRC 1720 ``Analysis of criticality: from complex phenomena to models and estimates'', 53930965.
Simon Nowak was supported by the Deutsche
Forschungsgemeinschaft (DFG, German Research Foundation) - SFB 1283/2 2021 - 317210226.

\subsection{AI statement} AI tools were used exclusively in the final stages of the preparation of this manuscript to check the paper for errors and inconsistencies.

\section{Preliminaries} \label{secprelim}
Throughout the paper, we denote by $c$ a universal constant which is bigger or equal than 1. Moreover, we use the parenthesis to show the dependence of the constant $c$ such as $c=c(n,s,\Lambda)$ if the constant $c$ depends only on $n,s,\Lambda$.

First we introduce some notation that is used frequently in the paper.
\begin{itemize}
    \item Let $\mathbb{T}^n\coloneqq (-\frac12,\frac12]^n$ and  $a\mathbb{T}^n(z)\coloneqq z+(-\frac{a}2,\frac{a}2]^n$ for any $a>0$ and $z\in\bbR^n$.
    \item For any $x\in\bbR^n$, $\|x\|_\infty\coloneqq \max\limits_{i}\{|x_i|\,:\,x=(x_i)\}$. 
    \item For any set $D$, we say that $u\in L^1_{2s}(D)$ if $\|u\|_{L^1_{2s}(D)}\coloneqq\int_{D}\frac{|u(y)|}{(1+|y|)^{n+2s}}\,dy<\infty$.
    \item For any $u\in L^1_{2s}(\bbR^n)$, we write
    \begin{equation}\label{defn.excess}
    \begin{aligned}
    \mathrm{Tail}(u;B_r(x_0))&\coloneqq r^{2s}\int_{\bbR^n\setminus B_r(x_0)}\frac{|u(y)|}{|x_0-y|^{n+2s}}\,dy,\\
        E(u;B_r(x_0))&\coloneqq \dashint_{B_r(x_0)}|u|\,dx+\mathrm{Tail}(u;B_r(x_0)).
    \end{aligned}
    \end{equation}
    \item For any bounded domain $\Omega$, we define $d_\Omega(x)\coloneqq \mathrm{dist}(x,\partial\Omega)$. In addition, for $\rho>0$, we write 
\begin{align}\label{defn.B_2Omega} 
    B_\rho(\Omega)\coloneqq\{x\in\bbR^n:\dist(x,\Omega)<\rho\}.
\end{align}
\end{itemize}

In order to control the growth at infinity, for a measurable set $A\subset\bbR^n$ we write
\begin{equation}\label{intro:tailspace}
\|v\|_{L^1_{2s}(A)}\coloneqq\int_A\frac{|v(y)|}{(1+|y|)^{n+2s}}\,dy,
\end{equation}
and denote by $L^1_{2s}(A)$ the corresponding weighted $L^1$ space. 

Let us make precise the notion of weak solutions that we are going to use. For $h\in L^2(\Omega)$, a function $u_\eps\in L^1_{2s}(\bbR^n)$ is a weak solution of the first problem in \eqref{intro:dirichlet} if $u_\eps\in H^s(U)$ for some bounded open neighborhood $U$ of $\overline\Omega$, if $u_\eps=g$ almost everywhere in $\bbR^n\setminus\Omega$, and if
\[
\frac12\iint_{\bbR^n\times\bbR^n}
\frac{(u_\eps(x)-u_\eps(y))(\varphi(x)-\varphi(y))}{|x-y|^{n+2s}}
K(x/\eps,y/\eps)\,dx\,dy
=\int_\Omega h\varphi\,dx
\]
for every $\varphi\in H^s(\bbR^n)$ that vanishes almost everywhere in $\bbR^n\setminus\Omega$. Here $H^s(U)=W^{s,2}(U)$ denotes the usual fractional Sobolev space. The definition for $\overline u$ is the same, with $K(x/\eps,y/\eps)$ replaced by $\overline K$.

\subsection{Periodic settings}
Here we recall the fractional Sobolev space in $\mathbb{T}^n$ and provide the existence and regularity of periodic solutions to nonlocal problems, and prove qualitative homogenization results.

For any $s\in(0,1)$, we recall the fractional Sobolev space $H^{s}({\mathbb{T}^n})$ given by 
\begin{align*}
    H^{s}({\mathbb{T}^n})\coloneqq\left\{f:\mathbb{T}^n\to\bbR\,:\,\int_{\mathbb{T}^n}\int_{\mathbb{T}^n}\frac{|f(x+z)-f(x)|^2}{|z|^{n+2s}}\,dx\,dz+\|f\|^2_{L^2(\mathbb{T}^n)}<\infty\right\}.
\end{align*}
In addition, we write its mean zero subspace $\dot{H}^s(\mathbb{T}^n)$ by 
\begin{align*}
    \dot{H}^{s}(\mathbb{T}^n)\coloneqq \left\{u\in H^{s}({\mathbb{T}^n})\,:\, \int_{\mathbb{T}^n}u=0\right\}.
\end{align*}
Now we establish the existence of periodic solutions to nonlocal equations with fractional divergence-type data and their regularity. 
\begin{lemma}\label{lem.bdd.periodic}
    Let $\eps\in(0,1]$, $i\in \{1,2,\ldots ,n\}$ and $K_0=K_0(x,y)$ be a periodic function with $K_0(x,y)=K_0(y,x)$ and
    $|K_0|\leq\Lambda$ for some $\Lambda\geq1$. There is a weak solution $u_i\in \dot{H}^{s}(\mathbb{T}^n)$ to
    \begin{align}\label{eq.periodic}
        Lu_i=\int_{\eps^{-1}\mathbb{T}^n}\frac{z_i}{|z|^{n+2s}}K_0(x,x-z)\,dz+\mbox{\large$\chi$}_{\{s>1/2\}}\int_{\bbR^n\setminus \eps^{-1}\mathbb{T}^n}\frac{z_i}{|z|^{n+2s}}K_0(x,x-z)\,dz,
    \end{align}
    in the sense that for any $\psi\in H^{s}(\mathbb{R}^n)$ with $\supp \psi\Subset \bbR^n$,
    \begin{equation}\label{bdd.periodic:weak.for}
    \begin{aligned}
        &\int_{\bbR^n}\int_{\bbR^n}\frac{(u_i(x)-u_i(x-z))(\psi(x)-\psi(x-z))}{|z|^{n+2s}}K(x,x-z)\,dx\,dz\\
        &=\int_{\bbR^n}\int_{\eps^{-1}\mathbb{T}^n}\frac{z_i(\psi(x)-\psi(x-z))}{|z|^{n+2s}}K_0(x,x-z)\,dz\,dx\\
        &\quad+\mbox{\large$\chi$}_{\{s>1/2\}}\int_{\bbR^n}\int_{\bbR^n\setminus\eps^{-1}\mathbb{T}^n}\frac{z_i(\psi(x)-\psi(x-z))}{|z|^{n+2s}}K_0(x,x-z)\,dz\,dx.
    \end{aligned}
    \end{equation}
    In addition, there is a small constant $\beta=\beta(n,s,\Lambda)$ such that
    \begin{equation*}
        \|u_i\|_{C^\beta(\mathbb{T}^n)}\leq c(n,s,\Lambda).
    \end{equation*}
    If $K \in C^\alpha(\mathbb T^n)$ for some $\alpha>0$, then for any $\sigma\in(0,\min\{2s,1\})$, we have
    \begin{align*}
        \|u_i\|_{C^{\sigma}(\mathbb{T}^n)}\leq c,
    \end{align*}
    where $c=c(n,s,\Lambda,\sigma,\alpha,[K]_{C^\alpha(\mathbb{T}^n)})$.
\end{lemma}
\begin{proof}
First, we prove the existence of the solution using a variational argument. Define 
\begin{align*}
    E_i(w)&\coloneqq \frac12\int_{\mathbb{T}^n}\int_{\mathbb{R}^n}\frac{|w(x)-w(x-z)|^2}{|z|^{n+2s}}K(x,x-z)\,dz\,dx\\
    &\quad-\int_{\mathbb{T}^n}\int_{\eps^{-1}\mathbb{T}^n}\frac{z_i(w(x)-w(x-z))}{|z|^{n+2s}}K_0(x,x-z)\,dz\,\,dx\\
    &\quad-\mbox{\large$\chi$}_{\{s>1/2\}}\int_{\mathbb{T}^n}\int_{\bbR^n\setminus\eps^{-1}\mathbb{T}^n}\frac{z_i(w(x)-w(x-z))}{|z|^{n+2s}}K_0(x,x-z)\,dz\,dx
\end{align*}
for any $w\in \dot{H}^{s}(\mathbb{T}^n)$. First, we see
\begin{align*}
    \int_{\mathbb{T}^n}\int_{\mathbb{R}^n}\frac{|w(x)-w(x-z)|^2}{|z|^{n+2s}}K(x,x-z)\,dz\,dx&=    \int_{\mathbb{T}^n}\int_{\mathbb{T}^n}\frac{|w(x)-w(x-z)|^2}{|z|^{n+2s}}K(x,x-z)\,dz\,dx\\
    &\quad+    \int_{\mathbb{T}^n}\int_{\mathbb{R}^n\setminus \mathbb{T}^n}\frac{|w(x)-w(x-z)|^2}{|z|^{n+2s}}K(x,x-z)\,dz\,dx\\
    &\leq c\left( [w]^2_{H^{s}(\mathbb{T}^n)}+\int_{\bbR^n\setminus\mathbb{T}^n}\frac{\|w\|^2_{L^2(\mathbb{T}^n)}}{|z|^{n+2s}}\,dz\right)\\
    &\leq c \|w\|^2_{H^{s}(\mathbb{T}^n)}
\end{align*}
for some constant $c=c(n,s)$, where we have used the fact that
\begin{align}\label{cond.periodic}
    \int_{\mathbb{T}^n}g(x+z)\,dx=\int_{\mathbb{T}^n}g(x)\,dx=\int_{\mathbb{T}^n(z)}g(x)\,dx
\end{align}
for any $z\in\bbR^n$ and periodic function $g$. Similarly, using H\"older's inequality and \eqref{cond.periodic}, we observe
\begin{align*}
    &\left|\int_{\mathbb{T}^n}\int_{\eps^{-1}\mathbb{T}^n}\frac{z_i(w(x)-w(x-z))}{|z|^{n+2s}}K_0(x,x-z)\,dz\,\,dx\right|\\
    &\quad+\left|\mbox{\large$\chi$}_{\{s>1/2\}}\int_{\mathbb{T}^n}\int_{\bbR^n\setminus\eps^{-1}\mathbb{T}^n}\frac{z_i(w(x)-w(x-z))}{|z|^{n+2s}}K_0(x,x-z)\,dz\,dx\right|\leq c\|w\|_{H^{s}(\mathbb{T}^n)}.
\end{align*}
We have proved $E_i(w)\leq c(1+\|w\|^2_{H^{s}(\mathbb{T}^n)})$ for any $w\in \dot{H}^{s}(\mathbb{T}^n)$. Moreover, we can prove coercivity of the functional $E_i(\cdot)$. Thus, using standard variational methods, we obtain a weak solution $u_i$ such that for any $\psi\in \dot{H}^{s}(\mathbb{T}^n)$, 
\begin{align*}
    &\int_{\mathbb{T}^n}\int_{\mathbb{R}^n}\frac{(u_i(x)-u_i(x-z))(\psi(x)-\psi(x-z))}{|z|^{n+2s}}K(x,x-z)\,dz\,dx\\
    &=\int_{\mathbb{T}^n}\int_{\eps^{-1}\mathbb{T}^n}\frac{z_i( \psi(x)-\psi(x-z))}{|z|^{n+2s}}K_0(x,x-z)\,dz\,dx\\
    &\quad+\mbox{\large$\chi$}_{\{s>1/2\}}\int_{\mathbb{T}^n}\int_{\bbR^n\setminus\eps^{-1}\mathbb{T}^n}\frac{z_i( \psi(x)-\psi(x-z))}{|z|^{n+2s}}K_0(x,x-z)\,dz \,dx.
\end{align*}
Now we want to prove that $u_i$ indeed satisfies \eqref{bdd.periodic:weak.for}. Fix $\psi\in H^{s}(\bbR^n)$ with $\psi\equiv 0$ on $\bbR^n\setminus D$ for some $D\Subset \bbR^n$.  Then for 
\begin{align*}
    \widetilde{\psi}(x)\coloneqq\sum_{y\in\mathbb{Z}^n}\psi(x+y)-(\psi(\cdot+y))_{\mathbb{T}^n}\in \dot{H}^{s}(\mathbb{T}^n),
\end{align*}
we have 
\begin{align*}
    &\int_{\mathbb{T}^n}\int_{\mathbb{R}^n}\frac{(u_i(x)-u_i(x-z))(\widetilde{\psi}(x)-\widetilde{\psi}(x-z))}{|z|^{n+2s}}K(x,x-z)\,dz\,dx\\
    &=\int_{\mathbb{T}^n}\int_{\eps^{-1}\mathbb{T}^n}\frac{z_i(\widetilde{\psi}(x)-\widetilde{\psi}(x-z))}{|z|^{n+2s}}K_0(x,x-z)\,dz\,dx\\
    &\quad+\mbox{\large$\chi$}_{\{s>1/2\}}\int_{\mathbb{T}^n}\int_{\bbR^n\setminus\eps^{-1}\mathbb{T}^n}\frac{z_i(\widetilde{\psi}(x)-\widetilde{\psi}(x-z))}{|z|^{n+2s}}K_0(x,x-z)\,dz\,dx.
\end{align*}
Using this and the periodicity of $u_i$, $K$, $K_0$, we deduce that $u_i$ is the desired weak solution to \eqref{eq.periodic}. 

Now we are going to prove the regularity of the solution $v_i$ by 
\begin{align}\label{bdd.periodic:eq.v_i}
    &Lv_i=\int_{4\mathbb{T}^n}\frac{z_i}{|z|^{n+2s}}K_0(x,x-z)\,dz+\mbox{\large$\chi$}_{\{s>1/2\}}\int_{\bbR^n\setminus 4\mathbb{T}^n}\frac{z_i}{|z|^{n+2s}}K_0(x,x-z)\,dz.
\end{align}
Note that when $s>\frac12$, we immediately derive the desired regularity estimates for $u_i$ from the regularity of $v_i$. When $s\leq\frac12$, we consider another solution $w_i\coloneqq u_i-v_i$ such that
\begin{equation}\label{bdd.periodic:eq.w_i}
    Lw_i(x)=\int_{\eps^{-1}\mathbb{T}^n\setminus 4\mathbb{T}^n}\frac{z_i}{|z|^{n+2s}}K_0(x,x-z)\,dz\eqqcolon f(x).
\end{equation}
Indeed, if we naively estimate $u_i$ as in the estimate of $v_i$ when $s\leq\frac12$, we get $\|u_i\|_{C^\sigma}\lesssim c(\eps^{-1+2s}+|\ln\eps|_{\mbox{\large$\chi$}_{\{s=1/2\}}})$. Instead, we use the periodicity of function $K_0$ to get some cancellation.

\begin{itemize}[leftmargin=*]
    \item We prove the regularity of $v_i$ via the perturbation argument. Fix $\sigma\in(0,\min\{2s,1\})$, $r\leq1$. Let $v$ be a weak solution to 
 \begin{equation*}
\left\{
\begin{alignedat}{3}
L v&=0&&\qquad \mbox{in  $B_r$}, \\
v&=v_i&&\qquad  \mbox{in $\bbR^n\setminus B_r$}.
\end{alignedat} \right.
\end{equation*}
Then, by the standard energy estimates, we have
\begin{align*}
    &r^{-n}\iint_{\bbR^{2d}}\frac{|(v_i-v)(x)-(v_i-v)(x-z)|^2}{|z|^{n+2s}}\,dz\,dx\\
    &\leq cr^{-n}\int_{\bbR^n}\int_{4\mathbb{T}^n}\frac{|z_i||(v_i-v)(x)-(v_i-v)(x-z)|}{|z|^{n+2s}}\,dz\,dx\\
    &\quad+\mbox{\large$\chi$}_{\{s>1/2\}}cr^{-n}\int_{\bbR^n}\int_{\bbR^n\setminus 4\mathbb{T}^n}\frac{|z_i||(v_i-v)(x)-(v_i-v)(x-z)|}{|z|^{n+2s}}\,dz\,dx\eqqcolon J_1+J_2,
\end{align*}
where $c=c(n,s,\Lambda)$. To estimate $J_1$, we rewrite it as
\begin{align*}
    J_1=\int_{\bbR^n}\int_{B_{2r}}(\cdot)+\int_{\bbR^n}\int_{4\mathbb{T}^n\setminus B_{2r}}(\cdot)\eqqcolon J_{1,1}+J_{1,2}.
\end{align*}
By the fact that $(u_i-v)\equiv 0$ in $\bbR^n\setminus B_r$ and the Cauchy inequality, we have
\begin{align*}
    J_{1,1}&\leq c\int_{B_{4r}}\dashint_{B_{2r}}\frac{|z_i||(v_i-v)(x)-(v_i-v)(x-z)|}{|z|^{n+2s}}\,dz\,dx\\
    &\leq \frac{r^{-n}}8[u_i-v]^2_{H^{s}(\bbR^n)}+c\int_{B_{4r}}\dashint_{B_{2r}}|z|^{2-(n+2s)}\,dz\,dx\\
    &\leq \frac{r^{-n}}8[v_i-v]^2_{H^{s}(\bbR^n)}+cr^{2-2s}
\end{align*}
for some constant $c=c(n,s,\Lambda)$. Similarly, we estimate $J_{1,2}$ as 
\begin{align*}
    J_{1,2}&\leq c\dashint_{B_r}|v_i-v|\,dx\int_{4\mathbb{T}^n\setminus B_{2r}}|z|^{1-(n+2s)}\,dz\\
    &\leq \frac{r^{-n}}8[v_i-v]^2_{H^{s}(\bbR^n)}+cr^{2s}\left(\int_{4\mathbb{T}^n\setminus B_{2r}}|z|^{1-(n+2s)}\,dz\right)^2\\
    &\leq \frac{r^{-n}}8[v_i-v]^2_{H^{s}(\bbR^n)}+c\mu(r),
\end{align*}
for some constant $c=c(n,s,\Lambda,\sigma)$, where
\begin{align*}
    \mu(r)\coloneqq \begin{cases}
        r^{2-2s}&\quad\text{if }s>\frac12,\\
        r(1+|\ln r|)^2&\quad\text{if }s=\frac12,\\
        r^{2s}&\quad\text{if }s<\frac12.
    \end{cases}
\end{align*} Now, using $(u_i-v)\equiv 0$ in $\bbR^n\setminus B_r$, we estimate $J_2$ as 
\begin{align*}
    J_2\leq c\dashint_{B_r}|v_i-v|\,dx\int_{\bbR^n\setminus 4\mathbb{T}^n}|z|^{1-(n+2s)}\,dz\leq \left(\frac{r^{-n}}8[v_i-v]^2_{H^{s}}+cr^{2s}\right).
\end{align*}
Combining all estimates $J_1$ and $J_2$, and using the Sobolev inequality, we have for $\sigma<\min\{1,2s\}$,
\begin{align}\label{bdd.periodic:comp.est}
    \dashint_{B_{r}}|v_i-v|^2\,dx\leq cr^{2\sigma},
\end{align}
for some constant $c=c(n,s,\sigma,\Lambda)$, where we have used the fact that $|\ln r|\leq c(q)r^{-q}$ for any $q>0$.
Now using the regularity of $v$ by \autoref{lem.reg.int}, for any $\rho\in(0,1]$ and $\sigma_0\coloneqq (\sigma+\min\{1,2s\})/2$, we get
\begin{align}\label{bdd.periodic:hol.v}
    E(v-(v)_{B_{\rho r}};B_{\rho r})\leq c\rho^{\sigma_0}E(v-(v)_{B_{ r}};B_r)
\end{align}
for some constant $c=c(n,s,\alpha,[K]_{C^\alpha(\mathbb{T}^n)},\sigma)$.
Therefore, for any $\rho\in(0,2^{-4}]$, we derive
\begin{align*}
    E(v_i-(v_i)_{B_{\rho r}};B_{\rho r})&\leq E(v-(v)_{B_{\rho r}};B_{\rho r})+E(v_i-v;B_{\rho r})\\
    &\leq c\left(\rho^{\sigma_0}E(v-(v)_{B_r};B_r)+\rho^{-n}E(v_i-v;B_r)\right)\\
    &\leq c\left(\rho^{\sigma_0} E(v_i-(v_i)_{B_r};B_r)+\rho^{-2n}r^{\sigma}\right),
\end{align*}
where we have used \eqref{bdd.periodic:hol.v} and \eqref{bdd.periodic:comp.est}. As in \cite[Theorem 4.4]{DN25}, we have for any $r\leq1$,
\begin{align}\label{bdd.periodic:est.hol.u}
    \dashint_{B_{r}}|v_i-(v_i)_{B_r}|\leq cr^{\sigma},
\end{align}
where $c=c(n,s,\alpha,[K]_{C^\alpha(\mathbb{T}^n)},\sigma)$. Since this is true for the function $v_{i_0}(x)\coloneqq v_i(x+x_0)$ for any $x_0\in \mathbb{T}^n$, we have the desired H\"older estimate 
\begin{equation}\label{bdd.periodic:reg.v_i}
    \|v_i\|_{C^\sigma(\mathbb{T}^n)}\leq c(n,s,\Lambda,[K]_{C^\alpha(\mathbb{T}^n)},\alpha,\sigma)
\end{equation}
for any $\sigma\in(0,\min\{2s,1\})$. Note that when $K$ is merely measurable, then we obtain the same estimate \eqref{bdd.periodic:reg.v_i} for a small constant $\sigma=\sigma(n,s,\Lambda)$ and $c=c(n,s,\Lambda)$.
\item We prove the regularity of $w_i$ given in \eqref{bdd.periodic:eq.w_i} by obtaining the boundedness of the right-hand side $f$.
\noindent
First, we observe
\begin{align*}
    f(x)&=\int_{\eps^{-1}\mathbb{T}^n\setminus 4\mathbb{T}^n}\frac{z_i}{|z|^{n+2s}}\left(K_0\left(x,x-{z}\right)-\overline{K}_0(x)\right)\,dz\\
    &=\sum_{p\in \Z^n}\int_{\mathbb{T}^n(p)\cap (\eps^{-1}\mathbb{T}^n\setminus 4\mathbb{T}^n)}\frac{z_i}{|z|^{n+2s}}\left(K_0\left(x,x-{z}\right)-\overline{K}_0(x)\right)\,dz\\
    &=\sum_{\substack{p\in \Z^n,\mathbb{T}^n(p)\setminus (\eps^{-1}\mathbb{T}^n)\neq\emptyset\\
    \mathbb{T}^n(p)\cap (\eps^{-1}\mathbb{T}^n)\neq\emptyset}}\int_{\mathbb{T}^n(p)\cap (\eps^{-1}\mathbb{T}^n)}\frac{z_i}{|z|^{n+2s}}\left(K_0\left(x,x-{z}\right)-\overline{K}_0(x)\right)\,dz\\
    &\quad+\sum_{\substack{p\in \Z^n,\mathbb{T}^n(p)\setminus (4\mathbb{T}^n)\neq\emptyset\\
    \mathbb{T}^n(p)\cap (4\mathbb{T}^n)\neq\emptyset}}\int_{\mathbb{T}^n(p)\cap (\eps^{-1}\mathbb{T}^n)}\frac{z_i}{|z|^{n+2s}}\left(K_0\left(x,x-{z}\right)-\overline{K}_0(x)\right)\,dz\\
    &\quad+\sum_{\substack{p\in \Z^n,\\
    \mathbb{T}^n(p)\subset (\eps^{-1}\mathbb{T}^n\setminus 4\mathbb{T}^n)}}\int_{\mathbb{T}^n(p)\cap (\eps^{-1}\mathbb{T}^n\setminus 4\mathbb{T}^n)}\frac{z_i}{|z|^{n+2s}}\left(K_0\left(x,x-{z}\right)-\overline{K}_0(x)\right)\,dz\eqqcolon \sum_{i=1}^3A_i,
\end{align*}
where $\overline{K}_0(x)\coloneqq\int_{\mathbb{T}^n}K_0(x,x-z)\,dz$, we have used the change of variables and the fact that
\begin{align*}
    \int_{\eps^{-1}\mathbb{T}^n\setminus 4\mathbb{T}^n}\frac{z_i}{|z|^{n+2s}}\,dz=0
\end{align*}
for the first equality. Note that for the set $S_\eps$  defined by
\begin{align*}
    S_\eps&\coloneqq \{p\in \Z^n\,:\, \mathbb{T}^n(p)\setminus (\eps^{-1}\mathbb{T}^n)\neq \emptyset,\quad \mathbb{T}^n(p)\cap (\eps^{-1}\mathbb{T}^n)\neq\emptyset\},
\end{align*}
we have that if $p\in S_\eps$, then
\begin{align*}
\eps^{-n+1}/c\leq |S_\eps|\leq c \eps^{-n+1}    \quad\text{and}\quad \eps^{-1}/c\leq |z|\leq c\eps^{-1}\quad\text{for any }z\in \mathbb{T}^n(p).
\end{align*}
Thus, using this, we estimate $A_1+A_2$ as
\begin{align*}
    A_1+A_2\leq c\left(\sum_{p\in S_\eps}\int_{\mathbb{T}^n(p)}\eps^{2s-1+n}\,dz+\sum_{p\in S_4}\int_{\mathbb{T}^n(p)}\,dz\right)\leq c.
\end{align*}
Now for $A_3$, we use the periodicity condition of $K$ to see that
\begin{align*}
    A_3&=\sum_{\substack{p\in S}}\int_{\mathbb{T}^n(p)}\left(\frac{z_i}{|z|^{n+2s}}-\frac{p_i}{|p|^{n+2s}}\right)K_0\left(x,x-{z}\right)\,dz-\sum_{\substack{p\in S}}\int_{\mathbb{T}^n(p)}\left(\frac{z_i}{|z|^{n+2s}}-\frac{p_i}{|p|^{n+2s}}\right)\overline{K}_0(x)\,dz,
\end{align*}
where $S\coloneqq\{p\in\Z^n\,:\, \mathbb{T}^n(p)\subset \eps^{-1}\mathbb{T}^n\setminus 4\mathbb{T}^n\}$. Note that $\|p\|_\infty\geq2$ as $p\in S$. Thus we observe $|z|\geq1$ and $|tz+(1-t)p|\geq \frac{|z|}{16}$ for any $z\in \mathbb{T}^n(p)$ and $t\in[0,1]$, when $p\in S$.
Next, using this observation, we further estimate $A_{3}$ as
\begin{align*}
    A_{3}&\leq c\sum_{p\in S}\int_{\mathbb{T}^n(p)}\frac{|z_i-p_i|}{|z|^{n+2s}}\,dz+\int_{\mathbb{T}^n(p)}|p_i|\left(\frac1{|z|^{n+2s}}-\frac{1}{|p|^{n+2s}}\,dz\right)\\
    &\leq c\left(\int_{|z|>1}\frac{\,dz}{|z|^{n+2s}}+\sum_{p\in S}\int_{\mathbb{T}^n(p)}\int_{0}^1|z-p||z|\left(\frac{\,dt\,dz}{|tz+(1-t)p|^{n+2s+1}}\right)\right)\\
    &\leq c\int_{|z|>1}\frac{\,dz}{|z|^{n+2s}}
\end{align*}
for some constant $c=c(n,s)$. Therefore, combining all estimates $A_1,A_2$ and $A_3$ yields $|f(x)|\leq c$, which implies for any $\sigma<\min\{2s,1\}$, 
\begin{equation}\label{bdd.periodic:reg.w_i}
    \|w_i\|_{C^\sigma(\mathbb{T}^n)}\leq c
\end{equation}
for some constant $c=c(n,s,\Lambda,\sigma,\alpha,[K]_{C^\alpha(\mathbb{T}^n)})$ by \autoref{lem.reg.int}. Moreover, when $A$ is merely measurable, by the standard nonlocal De Giorgi-Nash-Moser theory (see for instance \cite{DKP16}), we have $\|w_i\|_{C^\sigma(\mathbb{T}^n)}\leq c(n,s,\Lambda)$ for some small constant $\sigma=\sigma(n,s,\Lambda)$.
\end{itemize} 
Combining estimates \eqref{bdd.periodic:reg.v_i} and \eqref{bdd.periodic:reg.w_i} yields the desired result. This completes the proof.
\end{proof}
The next lemma is also about the existence and regularity of solutions to nonlocal equations when the right-hand side is given by non-divergence-type data. 
\begin{lemma}\label{lem.bdd.periodic2}
    Let $f$ be a bounded periodic function with $\int_{\mathbb{T}^n}f=0$ and $|f|\leq 1$. There is a weak solution $u_0\in \dot{H}^{s}(\mathbb{T}^n)$ to
    \begin{align*}
        L u_0=f,
    \end{align*}
    in the sense that for any $\psi\in H^{s}(\bbR^n)$ with $\supp \psi\Subset \bbR^n$,
    \begin{equation}\label{bdd.periodic2:weak.for}
    \begin{aligned}
        & \frac{1}{2} \int_{\bbR^n}\int_{\bbR^n}\frac{(u_i(x)-u_i(x-z))(\psi(x)-\psi(x-z))}{|z|^{n+2s}}K(x,x-z)\,dx\,dz=\int_{\bbR^n}f(x)\psi(x)\,dx.
    \end{aligned}
    \end{equation}
    Moreover, if  $K \in C^\alpha(\mathbb T^n)$ for some $\alpha>0$, then for any $\sigma\in(0,\min\{2s,1\})$,
    \begin{align*}
        \|u_i\|_{C^{\sigma}(\mathbb{T}^n)}\leq c,
    \end{align*}
    where $c=c(n,s,\Lambda,\alpha,[K]_{C^\alpha(\mathbb{T}^n)},\sigma)$. 
\end{lemma}
\begin{proof}
    First, for the existence of the solution, 
    we consider the functional $E(\cdot)$ by
    \begin{align*}
        E(w)\coloneqq \frac12\int_{\mathbb{T}^n}\int_{\mathbb{R}^n}\frac{|w(x)-w(x-z)|^2}{|z|^{n+2s}}K(x,x-z)\,dz\,dx-\int_{\mathbb{T}^n}fw\,dx.
    \end{align*}
    Then we note for any $w\in \dot{H}^{s}(\mathbb{T}^n)$,
    \begin{align*}
        \left|\int_{\mathbb{T}^n}fw\,dx\right|\leq \|f\|_{L^2(\mathbb{T}^n)}\|w\|_{\dot{H}^{s}(\bbR^n)}.
    \end{align*}
    Therefore, using this, the fact that $\int_{\mathbb{T}^n}f=0$ and following the same argument as in \autoref{lem.bdd.periodic}, we get the existence of the solution $u_0$ to \eqref{bdd.periodic2:weak.for}. Then its regularity follows from \autoref{lem.reg.int}. This completes the proof.
\end{proof}

Now, we present a qualitative homogenization result concerning the convergence of a sequence of homogenized operators and the corresponding solutions. Such a result is used for a compactness argument in the next section.
\begin{lemma}\label{lem.conv.qual}
    Let $(K_l)_l$ be a sequence of periodic kernel coefficients with $\Lambda^{-1}\leq K_l\leq \Lambda$ for some $\Lambda\geq1$. Let $(\eps_l)_l$ be a sequence such that $\lim\limits_{l\to\infty}\eps_l= 0$ and let  $u_{l}$ be a weak solution to 
     \begin{equation*}
\left\{
\begin{alignedat}{3}
L_{\eps_l}u_{l}&=f_l&&\qquad \mbox{in  $\Omega_l\cap B_1$}, \\
u_l&=0&&\qquad  \mbox{in $B_1\setminus \Omega_l$},
\end{alignedat} \right.
\end{equation*}
where $\Omega_l$ is a bounded domain and
    \begin{align*}
        L_{\eps_l}u_l(x)\coloneqq \int_{\bbR^n}\frac{u_l(x)-u_l(x-z)}{|z|^{n+2s}}K_l(x/\eps_l,(x-z)/\eps_l)\,dy
    \end{align*}
    with
    \begin{align}\label{conv:ass}
        E(u_l;B_{1})+\|f_l\|_{L^\infty(B_{1})}\leq 1.
    \end{align}
    Then we have the following:
    \begin{itemize}
        \item Assume 
        \begin{align}\label{conv.qual:domain}
       B_1\cap\{x_n>0\}\subset \Omega_l\cap B_1\subset B_1\cap\{x_n>-\eps_l\}.
   \end{align} Fix a cut off function $\psi\in C_c^\infty(B_{\frac34})$ with $\psi\equiv1 $ on $B_{\frac12}$. Then $u_{l}\psi\to u $ in $L^2({\bbR^n})$ up to a subsequence and $u$ is a weak solution to
      \begin{equation*}
\left\{
\begin{alignedat}{3}
\overline{L}u&=g&&\qquad \mbox{in  $B_{\frac12}\cap\{x_n>0\}$}, \\
u&=0&&\qquad  \mbox{in $B_{\frac12}\cap\{x_n\leq0\}$}
\end{alignedat} \right.
\end{equation*}
for some $g\in L^\infty(B_1)$ with
    \begin{align*}
        \overline{L}u\coloneqq\int_{\bbR^n}\frac{u(x)-u(x-z)}{|z|^{n+2s}}\overline{K}\,dz
    \end{align*}
    and  
    \begin{align}\label{conv.qual:goal1}
        E(u;B_{1})+\|g\|_{L^\infty(B_1)}\leq c,
    \end{align}
   where $\overline{K}\in[\Lambda^{-1},\Lambda]$ and $c=c(n,s,\Lambda)$.
   \item Let $\Omega_l\cap B_1=B_1$. Fix a cut off function $\psi\in C_c^\infty(B_{\frac34})$ with $\psi\equiv1 $ on $B_{\frac12}$. Then $u_{l}\psi\to u $ in $L^2({\bbR^n})$ up to a subsequence and $u$ is a weak solution to
      \begin{equation*}
\overline{L}u=g\quad\text{in }B_{\frac12}
\end{equation*}
where $g\in L^\infty(B_1)$ and  
    \begin{align}\label{conv.qual:goal2}
        E(u;B_{1})+\|g\|_{L^\infty(B_1)}\leq c,
    \end{align}
   for some constant $c=c(n,s,\Lambda)$.
    \end{itemize}
\end{lemma}
Note that assumption \eqref{conv:ass} ensures uniform bounds of $u_l$ in $H^s(B_{1})$, by standard energy estimates. However, this does not imply strong convergence of the sequence in $L^p(\bbR^n\setminus B_1)$ for any $p\geq1$. We overcome this issue by using a localization argument; see \eqref{conv.qual:goal2}. 

\begin{proof}
We define $\overline{K}_l\coloneqq \int_{\mathbb{T}^n}\int_{\mathbb{T}^n}K_l(x,y)\,dx\,dy$. Note that $\Lambda^{-1} \le \overline{K}_l \le \Lambda$, and therefore, up to a subsequence, there is $\overline{K} \in [\Lambda^{-1},\Lambda]$ such that $\lim\limits_{l\to\infty}\overline{K}_l= \overline{K}$. Next, by \cite[Proposition 2.2.2]{Shen18}, as $\Lambda^{-1}\leq K_l\leq\Lambda$, we observe for any $N>1$ and $F=F(x,y)\in L^1(B_N\times B_N)$,
    \begin{align}\label{conv:limit1}
        \lim_{l\to\infty}\int_{B_N}\int_{B_N}F(x,y) K_l(x/\eps_l,y/\eps_l)\,dx\,dy=\int_{B_N}\int_{B_N}F(x,y) \overline{K}\,dx\,dy.
    \end{align}

\begin{itemize}[leftmargin=*]
    \item By \cite[Lemma 2.3]{BKK25}, a localized function $v_l\coloneqq u_l\psi$ is a weak solution to 
    \begin{align*}
        L_{\eps_l}v_l=g_l\quad\text{in }\Omega_l \cap B_{\frac12}
    \end{align*}
    with $v_l\equiv 0$ in $B_{\frac12}\setminus \Omega_l$ and $\|g_l\|_{L^\infty(B_1)}\leq c$ by \eqref{conv:ass}. In addition, by standard energy estimates of $u_l$ and by using \eqref{conv:ass} again, we have 
    \begin{align}\label{conv.qual:est.vlgl}
        \|v_l\|_{H^{s}(\bbR^n)}+\|g_l\|_{L^\infty(B_1)}\leq c(\|u_l\|_{H^{s}(B_{3/4})}+\|g_l\|_{L^\infty(B_1)} )\leq c
    \end{align}
    for some constant $c=c(n,s,\Lambda)$. Thus,
    up to a subsequence, we get
    \begin{align*}
      \|u\|_{H^s(\bbR^n)}\leq c,\quad v_l\to u\quad\text{in }H^{\beta}(\bbR^n),\quad g_l\rightharpoonup g\quad\text{in }L^2(B_{\frac34})
    \end{align*}
  for any $\beta<s$ and
    \begin{align*}
        E(u;B_1)+\|g\|_{L^\infty(B_1)}\leq c
    \end{align*}
    for some constant $c=c(n,s,\Lambda)$. Thus, we have \eqref{conv.qual:goal1}.  In addition, we obtain $u\equiv 0$ on $(\bbR^n\setminus B_{\frac34})\cup (B_{\frac12}\cap \{x_n\leq0\})$, as $v_l\equiv 0$ on $B_1\setminus\Omega_l\supset B_{\frac12}\cap \{x_n\leq-\eps_l\}$ and the fact that $\eps_l\to0$ as $l\to\infty$. Now, it remains to prove $\overline{L}u=g$ in $ B_{\frac12}\cap \{x_n>0\}$. To this end, we fix $\varphi\in C_c^\infty(B_{\frac12}\cap\{x_n>0\})$ and  show that for any $\delta>0$, there is a large integer $l_0$ such that if $l\geq l_0$, then
    \begin{align*}
        I\coloneqq\left|\iint_{\bbR^{2d}}\frac{(v_l(x)-v_l(y))(\varphi(x)-\varphi(y))}{|x-y|^{n+2s}}K_l(x/\eps_l,y/\eps_l)-\frac{(u(x)-u(y))(\varphi(x)-\varphi(y))}{|x-y|^{n+2s}}\overline{K}\,dx\,dy\right|\leq\delta.
    \end{align*}
    First, we observe
    \begin{align*}
        \iint_{\bbR^{2d}}\frac{|u(x)-u(y)||\varphi(x)-\varphi(y)|}{|x-y|^{n+2s}}\,dx\,dy<\infty.
    \end{align*} Thus, there is a small constant $\sigma>0$ such that
    \begin{align*}
       \iint_{G_\sigma}\frac{|u(x)-u(y)||\varphi(x)-\varphi(y)|}{|x-y|^{n+2s}}\,dx\,dy\leq \frac\delta{3\Lambda},
    \end{align*}
    where $G_\sigma\coloneqq \{(x,y)\in\bbR^n\,:\,|x|+|y|\geq \sigma^{-1}\}$. Next, using \eqref{conv:limit1} and the fact that $u_l\to u$ in $H^{\beta}(\bbR^n)$ for any  $\beta\in(0,s)$, there is a sufficiently large number $l_1$ such that for any $l\geq l_1$,
    \begin{align*}
        &\left|\iint_{\bbR^{2d}\setminus G_\sigma}\frac{(u(x)-u(y))(\varphi(x)-\varphi(y))}{|x-y|^{n+2s}}(K_l(x/\eps_l,y/\eps_l)-\overline{K})\right|\leq \frac\delta3,\\
        &\left|\iint_{\bbR^{2d}}\frac{\left[(v_l(x)-v_l(y))-(u(x)-u(y))\right](\varphi(x)-\varphi(y))}{|x-y|^{n+2s}}\right|\leq \|v_l-u\|_{H^{\beta}(\bbR^n)}\|\varphi\|_{C^1(\bbR^n)}\leq\frac\delta{3\Lambda}.
    \end{align*}
    Therefore, we derive
    \begin{align*}
        I&=\left|\iint_{\bbR^{2d}}\frac{\left[(v_l(x)-v_l(y))-(u(x)-u(y))\right](\varphi(x)-\varphi(y))}{|x-y|^{n+2s}}K_l(x/\eps_l,y/\eps_l)\right|\\
        &\quad+\left|\iint_{\bbR^{2d}\setminus G_\sigma}\frac{(u(x)-u(y))(\varphi(x)-\varphi(y))}{|x-y|^{n+2s}}(K_l(x/\eps-l,y/\eps_l)-\overline{K})\right|\\
        &\quad+\left|\iint_{ G_\sigma}\frac{(u(x)-u(y))(\varphi(x)-\varphi(y))}{|x-y|^{n+2s}}(K_l(x/\eps-l,y/\eps_l)-\overline{K})\right|\leq\delta.
    \end{align*}
    Now using this and the fact that $\varphi\equiv 0$ on $B_1\setminus \Omega_l$ by \eqref{conv.qual:domain} and $g_l\rightharpoonup f$ in $L^2(B_{\frac34})$, we get for any $\delta>0$ that there is a large integer $l_0$ such that if $l\geq l_0$, then
    \begin{align*}
        |\langle\overline{L}u,\varphi\rangle-\langle g,\varphi\rangle |\leq I+|\langle g-g_l,\varphi\rangle|\leq 2\delta.
    \end{align*}
    Since $u\in H^{s}(\bbR^n)$ and by density, we have for any $\varphi\in H^{s}_0(B_{\frac12}\cap\{x_n>0\})$, $\langle\overline{L}u,\varphi\rangle=\langle g,\varphi\rangle$, which gives the desired result.
    \item For the second one, by following the same lines as in the first claim together with the choice of $\varphi\in C_c^\infty(B_{\frac12})$ instead of $\varphi\in C_c^\infty(B_{\frac12}\cap\{x_n>0\})$, we get the desired result.
\end{itemize}
\end{proof}
Lastly, we note that by following the above argument, we can get the same result when the kernel coefficient $K$ satisfies for $a\in S^{n-1}$, $K(x+ap,y+aq)=K(x,y)$ for all $p,q\in\mathbb{Z}^n$.
\subsection{Regularity for nonlocal equations} In this subsection, we provide several known regularity results for nonlocal equations. The first and second lemmas serve to prove uniform regularity estimates for homogenized problems. The third lemma is used to establish sharp convergence rates for such problems.

The first one is about interior regularity results.
\begin{lemma}\label{lem.reg.int}
    Let $\alpha\in(0,1)$ and $K\in C^\alpha(\mathbb{T}^n)$. Let $u$ be a weak solution to
    \begin{align*}
        Lu=f\quad\text{in }B_R(x_0),
    \end{align*}
    where $f\in L^\infty(B_R(x_0))$. Then for any $\sigma\in(0,\min\{1,2s\})$ and $\rho\in(0,1]$, we have 
    \begin{align}\label{reg.int:goal1}
        E(u-(u)_{B_{\rho R}(x_0)};B_{\rho R}(x_0))\leq c\rho^{\sigma}(E(u;B_R(x_0))+\|R^{2s}f\|_{L^\infty(B_R(x_0))})
    \end{align}
    for some constant $c=c(n,s,\Lambda,\alpha,[K]_{C^\alpha(\mathbb{T}^n)},\sigma)$. If $s>\frac12$, then  for any $\sigma\in(0,\min\{2s-1,\alpha\})$, we have
    \begin{align*}
        E(u-(u)_{B_{\rho R}(x_0)}-(\nabla u)_{B_{\rho R}(x_0)}\cdot (x-x_0);B_{\rho R}(x_0))\leq c\rho^{1+\sigma}(E(u;B_R(x_0))+\|R^{2s}f\|_{L^\infty(B_R(x_0))})
    \end{align*}
    for some constant $c=c(n,s,\Lambda,\alpha,[K]_{C^\alpha(\mathbb{T}^n)},\sigma)$. In addition, when $s>\frac12$, we have 
    \begin{align}\label{reg.int:goal3}
        E(u-(u)_{B_{\rho R}(x_0)};B_{\rho R}(x_0))\leq c\rho(E(u;B_R(x_0))+\|R^{2s}f\|_{L^\infty(B_R(x_0))})
    \end{align}
     for some constant $c=c(n,s,\Lambda,\alpha,[K]_{C^\alpha(\mathbb{T}^n)})$.
\end{lemma}
\begin{proof}
    From \cite[Proposition 2.12]{KNS22}, for any $\sigma\in(0,\min\{1,2s\})$ and $r\in(0,1]$, we have oscillation estimates 
    \begin{align*}
        \osc_{B_{r R}(x_0)}u\leq cr^{\sigma}(E(u;B_R(x_0))+\|R^{2s} f\|_{L^\infty(B_R(x_0))}),
    \end{align*}
    where $c=c(n,s,\Lambda,\alpha,[K]_{C^\alpha(\mathbb{T}^n)},\sigma)$. Note that the excess functional $E(\cdot)$ involves a tail term. Thus, combining this and the tail estimates given in \cite[Lemma 2.4]{KNS22} yields the desired estimate \eqref{reg.int:goal1}. When $s>\frac12$,  using the $C^{1,\sigma}$ estimate given in \cite[Theorem 1.12]{KNS22} for any $\sigma\in(0,\min\{2s-1,\alpha\})$ and the higher order version of the tail estimate given in \cite[Lemma 2.10]{KNS22}, we show the second claim. Together with the fact that
    \begin{align*}
        E((\nabla u)_{B_{\rho R}(x_0)}\cdot (x-x_0);B_{\rho R}(x_0)\leq c\rho R\|\nabla u\|_{L^\infty(B_{R/2}(x_0))}\leq c\rho( E(u;B_R(x_0))+\|R^{2s}f\|_{L^\infty(B_R(x_0))}),
    \end{align*}
    the second claim gives \eqref{reg.int:goal3}. This completes the proof.
\end{proof}
Next, we provide boundary regularity in $C^{1,\alpha}$ domains. We refer to Appendix \ref{sec.appendix} for the definition of a $C^{1,\alpha}$ domain and related notions such as the $C^{1,\alpha}$ radius and $C^{1,\alpha}$ norm.
\begin{lemma}\label{lem.reg.bdd1}
     Let $r_0,M>0$ and let $\partial\Omega\in C^{1,\alpha}$ with $C^{1,\alpha}$ radius $r_0$ and $C^{1,\alpha}$ norm bounded by $M$. Let $\alpha\in(0,1)$ and $K\in C^\alpha(\mathbb{T}^n)$. Let $u$ be a weak solution to
     \begin{equation*}
\left\{
\begin{alignedat}{3}
Lu&=f&&\qquad \mbox{in  $\Omega\cap B_R(x_0)$}, \\
u&=0&&\qquad  \mbox{in $ B_R(x_0)\setminus \Omega$}
\end{alignedat} \right.
\end{equation*}
for some $f\in L^\infty(B_R(x_0))$. Then we have for any $\rho\in(0,1]$,
\begin{align*}
    E(u-(u)_{B_{\rho R}(x_0)};B_{\rho R}(x_0))\leq c\rho^s(E(u;B_R(x_0))+\|R^{2s}f\|_{L^\infty(B_R(x_0))})
\end{align*}
for some constant $c=c(n,s,\Lambda,\alpha,[K]_{C^\alpha(\mathbb{T}^n)},r_0,M)$. 
\end{lemma}
\begin{proof}
    By \cite[Theorem 1.6]{KW25}, we have 
    \begin{align*}
        \osc_{B_{r R}(x_0)} u\leq cr^s(E(u;B_R(x_0))+\|R^{2s}f\|_{L^\infty(B_R(x_0))})
    \end{align*}
    for some constant $c=c(n,s,\Lambda,\alpha,[K]_{C^\alpha(\mathbb{T}^n)},\Omega)$. Although the theorem states that the constant depends on $\Omega$, we can easily obtain a more precise dependence of the constants by carefully inspecting the proof. Indeed, the constant $c$ depends on $n,s,\Lambda,\alpha,[K]_{C^\alpha(\mathbb{T}^n)},r_0,M$. Then combining this and the tail estimate in \cite[Lemma 2.4]{KNS22} yields the desired estimate. This completes the proof.
\end{proof}
The third result is about global regularity of the Dirichlet problem for the fractional Laplacian.

\begin{lemma}\label{lem.reg.bdd2}
  Let $r_0,M,N>0$ and let $\partial\Omega\in C^{1,\alpha}$ with $C^{1,\alpha}$ radius $r_0$, $C^{1,\alpha}$ norm bounded by $M$, $\mathrm{diam}(\Omega)\leq N$ and $0\in \Omega$. Let $u$ be a weak solution to 
     \begin{equation*}
\left\{
\begin{alignedat}{3}
(-\Delta)^{s}u&=f&&\qquad \mbox{in  $\Omega$}, \\
u&=g&&\qquad  \mbox{in $ \R^n \setminus \Omega$},
\end{alignedat} \right.
\end{equation*}
where $f\in C^{1}(\Omega)$ and $g\in C^{1}(B_2(\Omega))\cap L^1_{2s}(\bbR^n)$. Then we have the following:
\begin{itemize}
    \item Let $s\neq\frac12$. Then for any $k\in\{0,1\}$ and $\mu\in(0,\min\{2s,1\}]$ with $k+\mu\geq s$,
    Then we have
\begin{align*}
    \frac{|\nabla^ku(x)-\nabla^ku(y)|}{|x-y|^\mu}\leq c\left(d_\Omega^{s-(k+\mu)}(x)+d_\Omega^{s-(k+\mu)}(y)\right)(\|f\|_{C^{1}(\Omega)}+\|g\|_{C^{1}(B_2(\Omega)\setminus \Omega)}+\|g\|_{L^1_{2s}(\bbR^n\setminus \Omega)}),
\end{align*}
where $c=c(n,s,\Lambda,\alpha,r_0,M,N,\mu)$.
\item Let $s< \frac12$. Let $f\in C^{1,\delta}(\Omega)$ for some $\delta>0$. If $2s+\delta<1$, then we have
\begin{align*}
    \frac{|\nabla u(x)-\nabla u(y)|}{|x-y|^{2s+\delta}}\leq c\left(d_\Omega^{-1-s-\delta}(x)+d_\Omega^{-1-s-\delta}(y)\right)(\|f\|_{C^{1+\delta}(\Omega)}+\|g\|_{C^{1}(B_2(\Omega)\setminus \Omega)}+\|g\|_{L^1_{2s}(\bbR^n\setminus \Omega)}),
\end{align*}
where $c=c(n,s,\Lambda,\alpha,r_0,M,N,\delta)$. 
\item Let $s=\frac12$. Let $f\in C^{1,\delta}(\Omega)$ for some $\delta>0$. Then we have 
\begin{align*}
    &\frac{|\nabla u(x)-\nabla u(y)|}{|x-y|}\leq c\left(d_\Omega^{-\frac32}(x)+d_\Omega^{-\frac32}(y)\right)(\|f\|_{C^{1+\delta}(\Omega)}+\|g\|_{C^{1}(B_2(\Omega)\setminus \Omega)}+\|g\|_{L^1_{2s}(\bbR^n\setminus \Omega)}),\\
    & \frac{|\nabla^2 u(x)-\nabla^2 u(y)|}{|x-y|^\delta}\leq c\left(d_\Omega^{-\frac32-\delta}(x)+d_\Omega^{-\frac32-\delta}(y)\right)(\|f\|_{C^{1+\delta}(\Omega)}+\|g\|_{C^{1}(B_2(\Omega)\setminus \Omega)}+\|g\|_{L^1_{2s}(\bbR^n\setminus \Omega)}),
\end{align*}
where $c=c(n,\Lambda,\alpha,r_0,M,N,\delta)$. 
\end{itemize}
\end{lemma}
\begin{proof}
For any $B_R(x_0)\subset \Omega$, by the standard regularity theory, if $2s+\delta\notin \N$, then 
\begin{align*}
    [u]_{C^{2s+\delta}(B_{R/2}(x_0))}&\leq \frac{c}{R^{2s+\delta}}(E(u-(u)_{B_R(x_0)};B_R(x_0))+[R^{2s+\delta}f]_{C^\delta(B_R(x_0))})
\end{align*}
for some constant $c=c(n,s,\Lambda,\delta)$. In addition, from \cite[Proposition 2.6.16]{FRRO24} together with a suitable cutoff argument, we have 
\begin{align*}
    \|u\|_{C^s(B_2(\Omega))}\leq c(\|g\|_{C^{1}(B_2(\Omega)\setminus \Omega)}+\|g\|_{L^1_{2s}(\bbR^n\setminus \Omega)}+\|f\|_{L^\infty(\Omega)})
\end{align*}
for some constant $c=c(n,s,\Lambda,\alpha,r_0,M,N)$. Therefore, combining the above two estimates together with a standard covering argument yields the desired estimates.
\end{proof}

\section{Large scale regularity estimates}
In this section, we prove regularity estimates for solutions $u_\eps$ to nonlocal equations involving the homogenized operator $L_\eps $, with constants independent of the small parameter $\eps$.

We present the main result of this section.
\begin{theorem}\label{thm.uni.est}
    Let $\alpha\in(0,1)$, $\eps\in(0,1)$, $K\in C^\alpha(\mathbb{T}^n)$, $f\in L^\infty(\Omega\cap B_R(x_0))$ and $\partial\Omega\in C^{1,\alpha}$ with $C^{1,\alpha}$ radius $r_0$ and $C^{1,\alpha}$ norm bounded by $M$ for some constants $r_0,M>0$. Let $u_\eps$ be a weak solution to 
    \begin{equation*}
\left\{
\begin{alignedat}{3}
L_\eps u_\eps&=f&&\qquad \mbox{in  $\Omega\cap B_R(x_0)$}, \\
u_\eps&=0&&\qquad  \mbox{in $B_R(x_0)\setminus \Omega$}
\end{alignedat} \right.
\end{equation*}
for some $R>0$ and $x_0\in\bbR^n$. Then for any $\beta\in(0,s)$, we have
\begin{align}\label{uni.est:main.est1}
    [u_\eps]_{C^{\beta}(B_{R/2}(x_0))}\leq cR^{-\beta}(E(u_\eps;B_R(x_0))+R^{2s}\|f\|_{L^\infty(\Omega\cap B_R(x_0))}),
\end{align}
where $c=c(n,s,\Lambda,\alpha,[K]_{C^{\delta}(\mathbb{T}^n)},\beta,r_0,M)$. If $B_R(x_0)\subset \Omega$, then for any $\beta\in(0,\min\{2s,1\})$,
\begin{align*}
    [u_\eps]_{C^\beta(B_{R/2}(x_0))}\leq cR^{-\beta}(E(u_\eps;B_R(x_0))+R^{2s}\|f\|_{L^\infty( B_R(x_0))})
\end{align*}
for some constant $c=c(n,s,\Lambda,\alpha,[K]_{C^{\alpha}(\mathbb{T}^n)},\beta)$. Furthermore, when $B_R(x_0)\subset \Omega$ and $s>\frac12$, we get
\begin{align*}
    \|\nabla u_\eps\|_{L^\infty(B_{R/2}(x_0))}\leq cR^{-1}(E(u_\eps;B_R(x_0))+R^{2s}\|f\|_{L^\infty(B_R(x_0))})
\end{align*}
for some constant $c=c(n,s,\Lambda,\alpha,[K]_{C^{\alpha}(\mathbb{T}^n)})$.
\end{theorem}
To obtain this, we first establish regularity estimates down to the microscopic scale $\eps$ without assuming any smoothness of the associated kernel coefficient $K$ based on a compactness argument. More precisely, we prove uniform decay estimates at scales $r\geq \eps$. Next, assuming that the kernel coefficient $K$ is H\"older continuous, we use the previously mentioned regularity results for nonlocal problems to obtain the corresponding uniform estimates at scales $r\leq \eps$. Combining these estimates yields the desired result. 

We divide this section into three subsections, devoted to interior estimates at scales $r\geq \eps$, boundary estimates on scales $r\geq \eps$, and uniform estimates on all scales.

\subsection{Interior estimates}

The main result of this subsection is the following interior regularity result at large scales $r \ge \eps$. Since we are working only on large scales, throughout this subsection, we do not make any regularity assumptions on $K$.

\begin{proposition}\label{lem.large.int}
    Let $\eps\in(0,1)$ and let $v_\eps$ be a weak solution to 
    \begin{align*}
        L_\eps u_\eps=f\quad\text{in }B_1.
    \end{align*}
    For any $r\in[\eps,1]$ and $\beta<\min\{2s,1\}$, we have 
    \begin{align}\label{large.int:main.est}
        E(u_\eps-(u_\eps)_{B_r};B_r)\leq cr^{\beta}(E(u_\eps;B_1)+\|f\|_{L^\infty(B_1)})
    \end{align}
    for some constant $c=c(n,s,\Lambda,\beta)$. Furthermore, if $s>\frac12$, we have \eqref{large.int:main.est} with $\beta=1$.
\end{proposition}
The proof proceeds by establishing a one step improvement of oscillation via a compactness argument and then iterating this estimate to deduce a suitable decay estimate.

We first give the lemma on the one step improvement of oscillation below the gradient level.
\begin{lemma}\label{lem.exp.sol}
    Let $\eps\in(0,1)$ and let $u_\eps$ be a weak solution to 
    \begin{align*}
        L_\eps u_\eps=f\quad\text{in }B_1.
    \end{align*} Fix $\beta\in(0,\min\{2s,1\})$. Then there are $\eps_0, \theta \in (0,1)$ depending only on $n,s,\Lambda,\beta$ such that for any $\eps\leq \eps_0$,
    \begin{align*}
        E(u_\eps-(u_\eps)_{B_\theta};B_\theta)\leq \theta^{\beta}(E(u_\eps;B_1)+\|f\|_{L^\infty(B_1)})
    \end{align*}
    for some constant $c=c(n,s,\Lambda,\beta)$.
\end{lemma}
\begin{proof}
    Fix $\beta\in(0,\min\{2s,1\})$ and a cutoff function $\psi\in C_c^\infty(B_{\frac34})$ with $\psi\equiv 1$ on $B_{\frac12}$. First, we make an appropriate choice for $\theta=\theta(n,s,\Lambda,\beta)$. To do so, note that there is a constant $c_1=c_1(n,s)$ such that for any $u\in L^1_{2s}(\bbR^n)$ and $\theta\leq\frac12$,
\begin{equation}\label{exp.bdry:defn.c_1}
    \begin{aligned}
     |E(u\psi-(u\psi)_{B_\theta};B_{\theta})-E(u-(u)_{B_\theta};B_\theta)|&=\theta^{2s}\left|\int_{\bbR^n\setminus B_{\frac12}}\frac{|(u\psi)(y)-(u)_{B_\theta}|}{|y|^{n+2s}}-\frac{|u(y)-(u)_{B_\theta}|}{|y|^{n+2s}}\,dy\right|\\
     &\leq \theta^{2s}\int_{\bbR^n\setminus B_{\frac12}}\frac{|(u\psi)(y)-u(y)|}{|y|^{n+2s}}\,dy\\
     &\leq c_1\theta^{2s}E(u;B_1).
\end{aligned}
\end{equation}
Next, let us fix $c_0=c_0(n,s,\Lambda)$ determined in \eqref{conv.qual:goal2}. For any weak solution $u\in L^2(B_1)\cap L^1_{2s}(\bbR^n)$ to
\begin{align}\label{exp.sol:eq}
\begin{cases}
    \overline{L}u=g\quad\text{in }B_{\frac12},\\
     E(u;B_1)+\|g\|_{L^\infty(B_1)}\leq c_0,
\end{cases}
\end{align}
using \autoref{lem.reg.int}, we have for any $\rho\leq\frac14$,
\begin{align*}
    E(u-(u)_{B_\rho};B_\rho)\leq c\rho^{\beta_0}(E(u;B_{1/2})+\|g\|_{L^\infty(B_{1/2})}) \leq c\rho^{\beta_0}(E(u;B_1)+\|g\|_{L^\infty(B_1)})\leq c\rho^{\beta_0},
\end{align*}
where $\beta_0\coloneqq \frac12(\beta+\min\{1,2s\})<\min\{1,2s\}$ and $c=c(n,s,\Lambda,\beta)$. Now we fix $\theta=\theta(n,s,\Lambda,\beta)$ such that for any weak solution $u$ to \eqref{exp.sol:eq},
\begin{align}\label{exp.sol:cond.theta}
    E(u-(u)_{B_\theta};B_\theta)\leq \frac{\theta^{\beta}}{2}\quad\text{and}\quad c_1\theta^{2s}\leq \frac{\theta^{\beta}}{4}.
\end{align}
With this choice of $\theta=\theta(n,s,\Lambda,\beta)$, we now prove the main result by contradiction. Assume that there are sequences $(\eps_l)_l$, $(K_l)_l$, $(u_{\eps_l})_l$ and $(f_l)_l$ such that $u_{\eps_l}$ is a weak solution to
 \begin{equation*}
\int_{\bbR^n}\frac{u_{\eps_l}(x)-u_{\eps_l}(x-z)}{|z|^{n+2s}}K_l(x/\eps_l,(x-z)/\eps_l)\,dy=f_l(x)\quad\text{in }B_1
\end{equation*}
and $\lim\limits_{l\to\infty}\eps_l=0$, but
\begin{align*}
    E(u_{\eps_l}-(u_{\eps_l})_{B_\theta};B_\theta)>\theta^{\beta}(E(u_{\eps_l};B_1)+\|f_l\|_{L^\infty(B_1)}).
\end{align*}
By the normalization, we may assume $E(u_{\eps_l};B_1)+\|f\|_{L^\infty(B_1)}=1$, which gives 
\begin{align}\label{exp.sol.contr}
    E(u_{\eps_l}-(u_{\eps_l})_{B_\theta};B_\theta)>\theta^{\beta}.
\end{align}
Now applying the second claim in \autoref{lem.conv.qual}, there is a solution $u$ to 
\begin{align*}
    \overline{L}u=g\quad\text{in }B_{\frac12}
\end{align*}
such that $u_{\eps_l}\psi\to u$ in $L^2(\bbR^n)$ up to a subsequence and
\begin{align*}
    E(u;B_1)+\|g\|_{L^\infty(B_1)}\leq c_0.
\end{align*}
Then by \eqref{exp.sol:cond.theta}, we have 
\begin{align*}
    E(u-(u)_{B_\theta};B_\theta)\leq \frac{\theta^\beta}2.
\end{align*}
On the other hand, using \eqref{exp.bdry:defn.c_1}, \eqref{exp.sol:cond.theta}, \eqref{exp.sol.contr} and the assumption that $E(u_{\eps_l};B_1)\leq 1$  we observe 
\begin{align*}
    E(u_{\eps_l}\psi-(u_{\eps_l}\psi)_{B_\theta};B_\theta)\geq E(u_{\eps_l}-(u_{\eps_l})_{B_\theta};B_\theta)-c_1\theta^{2s}E(u_{\eps_l};B_1)\geq \frac{3}4\theta^{\beta}.
\end{align*}
However, using the convergence $u_{\eps_l}\psi\to u$ in $L^2(\bbR^n)$, we have
\begin{align*}
    \frac{\theta^\beta}2\geq E(u-(u)_{B_\theta};B_\theta)=\lim_{l\to\infty} E(u_{\eps_l}\psi-(u_{\eps_l}\psi)_{B_\theta};B_\theta)>\frac34\theta^\beta,
\end{align*}
which is a contradiction. This completes the proof.
\end{proof}
Next we prove a higher order version of \autoref{lem.exp.sol}.
\begin{lemma}\label{lem.exp.lip}
 Let $s>\frac12$, $\eps\in(0,1)$ and let $u_\eps$ be a weak solution to 
    \begin{align*}
        L_\eps u_\eps=f\quad\text{in }B_1.
    \end{align*}
    Let $\phi=(\phi_i)_i$ be a periodic solution to 
    \begin{align}\label{exp.lip:per.eq}
        L_1\phi(x)=-\int_{\bbR^n}\frac{z}{|z|^{n+2s}}K(x,x-z)\,dz.
    \end{align}
    Fix $\gamma\in(1,2s)$.
    Then there are $\eps_0 ,\theta \in (0,1)$ depending only on $n,s,\Lambda,\gamma$ such that for any $\eps\leq \eps_0$,
    \begin{align*}
        \inf_{a\in\bbR,\,b\in\bbR^n}{E}(u_\eps-a-b\cdot(x+\eps\phi(x/\eps));B_\theta)\leq \theta^{\gamma}({E}(u_\eps;B_1)+\|f\|_{L^\infty(B_1)}).
    \end{align*}
\end{lemma}
First, we give some comments on the lemma.
\begin{itemize}
    \item The existence of the periodic function $\phi_i$ follows from \autoref{lem.bdd.periodic}. In addition, we have $\|\phi_i\|_{L^\infty(\bbR^n)}\leq c(n,s,\Lambda)$.
    \item By \autoref{lem.funct} below, there exist $a_0\in\bbR$ and $b_0\in\bbR^n$ such that
    \begin{align*}
         \inf_{a\in\bbR,\,b\in\bbR^n}{E}(u_\eps-a-b\cdot(x+\eps\phi(x/\eps));B_\theta)= {E}(u_\eps-a_0-b_0\cdot(x+\eps\phi(x/\eps);B_\theta).
    \end{align*}
    \item By the construction of $\phi_i$ given in \eqref{exp.lip:per.eq}, we observe 
    \begin{equation*}
    L_\eps(u_\eps-a-b\cdot(x+\eps\phi(x/\eps))=0,
\end{equation*}
which is essential to close the iteration scheme in the proof of \autoref{lem.large.int}.
\end{itemize}

\begin{proof}[Proof of \autoref{lem.exp.lip}]
We proceed as in the proof of \autoref{lem.exp.sol}. First, we choose an appropriate $\theta=\theta(n,s,\Lambda,\gamma)$.
To do so, fix a cutoff function $\psi\in C_c^\infty(B_{\frac34})$ with $\psi\equiv 1$ on $B_{\frac12}$ and recall the constant $c_1=c_1(n,s)$ determined in \eqref{exp.bdry:defn.c_1}.

Next, for any weak solution $u$ to
    \begin{align*}
    \begin{cases}
        &\overline{L}u=g\quad\text{in }B_{\frac12},\\
        &E(u;B_1)+\|g\|_{L^\infty(B_1)}\leq c_0,
    \end{cases}
    \end{align*}
    where the constant $c_0=c_0(n,s,\Lambda)$ is determined in \eqref{conv.qual:goal2}, by \autoref{lem.reg.int}, we then obtain for any $\rho\in(0,\frac14]$,
    \begin{align*}
         E(u-(u)_{B_\rho}-(\nabla u)_{B_\rho}\cdot  x;B_\rho)\leq c\rho^{\frac{\gamma+2s}2}(E(u;B_1)+\|g\|_{L^\infty(B_1)})\leq c\rho^{\frac{\gamma+2s}2}
    \end{align*}
    for some constant $c=c(n,s,\Lambda,\gamma)$. Thus, we now fix a small $\theta=\theta(n,s,\Lambda,\gamma)$ such that
   \begin{align}\label{exp.lip:cond.theta}
         E(u-(u)_{B_\theta}-(\nabla u)_{B_\theta}\cdot  x;B_\theta)\leq \frac{\theta^{\gamma}}2\quad\text{and}\quad c_1\theta^{2s}\leq \frac{\theta^{\gamma}}4.
    \end{align}
    Taking $\theta$ as above, we prove the main result by contradiction. Assume that there are sequences $(\eps_l)_l$, $(K_l)_l$, $(u_{\eps_l})_l$, $(\phi_l)_l$ and $(f_l)_l$ such that $\lim\limits_{l\to\infty}\eps_l=0$ and,  $u_{\eps_l}$ and $\phi_l$ are weak solutions to
 \begin{align*}
\int_{\bbR^n}\frac{u_{\eps_l}(x)-u_{\eps_l}(x-z)}{|z|^{n+2s}}K_l(x/\eps_l,(x-z)/\eps_l)\,dy=f_l(x)\quad\text{in }B_1
\end{align*}
and
\begin{align*}
    \int_{\bbR^n}\frac{\phi_{l}(x)-\phi_{l}(x-z)}{|z|^{n+2s}}K_l(x,x-z)\,dy=-\int_{\bbR^n}\frac{z}{|z|^{n+2s}}K_l(x,x-z)\,dy,
\end{align*}
respectively, but
\begin{align*}
    \inf_{a\in\bbR,\,b\in\bbR^n}E(u_{\eps_l}-a-b\cdot(x+\eps_l\phi_{l}(x/\eps_l));B_\theta)>\theta^{\gamma}(E(u_{\eps_l};B_1)+\|f\|_{L^\infty(B_1)}).
\end{align*}
After normalization, we may assume $E(u_{\eps_l};B_1)+\|f_l\|_{L^\infty(B_1)}=1$. Then we have
\begin{align}\label{exp.lip:contr}
     \inf_{a\in\bbR,\,b\in\bbR^n}E(u_{\eps_l}-a-b\cdot(x+\eps_l\phi_{l}(x/\eps_l));B_\theta)> \theta^{\gamma}.
\end{align}
On the other hand, let us fix  $\widetilde{a}_l\in\bbR$ and $\widetilde{b}_l\in\bbR^n$ such that
\begin{align*}
    E(u_{l}-\widetilde{a}_l-\widetilde{b}_l\cdot(x+\eps_l\phi_{}(\cdot/\eps_l));B_\theta)=\inf_{a\in\bbR,\,b\in\bbR^n}E(u_{l}-a-b\cdot(x+\eps_l\phi_{l}(\cdot/\eps_l));B_\theta),
\end{align*}
where $u_l\coloneqq u_{\eps_l}\psi$. In addition, we have
\begin{align}\label{exp.lip:phibdd}
    E(u_l;B_\theta)\leq c(n,s,\theta) E(u_l;B_1)\leq c(n,s,\Lambda,\gamma)\quad\text{and}\quad \|\phi_l\|_{L^\infty}\leq c(n,s),
\end{align}
where we have used \autoref{lem.bdd.periodic} for the boundedness of $\phi_l$. By \eqref{funct:bdd.a_0b_0}, we obtain 
\begin{align*}
    |\widetilde{a}_l|+|\widetilde{b}_l|\leq c(n,s,\Lambda,\gamma)
\end{align*}
for some constant $c$ independent of $l$. This implies $\widetilde{a}_l\to a_\infty$, $\widetilde{b}_l\to b_\infty$ up to a subsequence.

Next, by the second claim in \autoref{lem.conv.qual}, there is a solution $u$ to
\begin{align*}
    \overline{L}u=g\quad\text{in }B_{\frac12},
\end{align*}
where $u_{l}=u_{\eps_l}\psi\to u$ in $L^2(\bbR^n)$ and 
\begin{align*}
    E(u;B_1)+\|g\|_{L^\infty(B_1)}\leq c_0.
\end{align*}
Thus, by the second claim in \autoref{lem.funct} together with the fact that $\widetilde{a}_l\to a_\infty$, $\widetilde{b}_l\to b_\infty$, $u_l\to u$ in $L^2(\bbR^n)$ and $\|\eps_l\phi(\cdot/\eps_l)\|_{L^\infty}\leq c\eps_l$ for some constant $c=c(n,s,\Lambda,\gamma)$ independent of $l$ by \eqref{exp.lip:phibdd}, we have 
\begin{align*}
    \inf_{a\in\bbR,\,b\in\bbR^n}E(u-a-b\cdot x;B_\theta)=\lim_{l\to\infty}E(u_{l}-\widetilde{a}_l-\widetilde{b}_l\cdot(x+\eps_l\phi_{}(\cdot/\eps_l));B_\theta)=E(u-a_\infty-b_\infty\cdot x;B_\theta).
\end{align*}
Then, using \eqref{exp.lip:cond.theta}, we get 
\begin{align}\label{exp.lip:limit.contr}
    E(u-a_\infty-b_\infty\cdot x;B_\theta)\leq E(u-(u)_{B_\theta}-(\nabla u)_{B_\theta}\cdot x;B_\theta)\leq \frac{\theta^{\gamma}}{2}.
\end{align}
However, we have 
\begin{align*}
    E(u_l-\widetilde{a}_l-\widetilde{b}_l\cdot(x+\eps_l\phi_{\eps}(\cdot/\eps_l));B_\theta)\geq E(u_{\eps_l}-\widetilde{a}_l-\widetilde{b}_l\cdot(x+\eps_l\phi_{}(\cdot/\eps_l));B_\theta)- E(u_{\eps_l}-u_{l};B_\theta)\geq\frac34\theta^{\gamma},
\end{align*}
where we have used \eqref{exp.lip:contr} and the fact that
\begin{align*}
   E(u_{\eps_l}-u_{l};B_\theta)= E(u_{\eps_l}-u_{\eps_l}\psi;B_\theta)\leq c_1\theta^{2s}E(u_{\eps_l};B_1)\leq c_1\theta^{2s}\leq \frac{\theta^{\gamma}}{4},
\end{align*}
by \eqref{exp.bdry:defn.c_1} and \eqref{exp.lip:cond.theta}. 
Therefore, we have 
\begin{align*}
 E(u-a_\infty-b_\infty\cdot x;B_\theta)=\lim_{l\to\infty}    E(u_l-\widetilde{a}_l-\widetilde{b}_l\cdot(x+\eps_l\phi_{\eps_l}(x));B_\theta)> \frac34\theta^{\gamma},
\end{align*}
which contradicts \eqref{exp.lip:limit.contr}. This completes the proof.
\end{proof}
Now we are ready to prove the excess decay estimates at scales $r\geq \eps$ stated in \autoref{lem.large.int} by an iterative application of \autoref{lem.exp.sol} and \autoref{lem.exp.lip}.
\begin{proof}[Proof of \autoref{lem.large.int}]
We split the proof into two parts.
\begin{itemize}
    \item First, we prove \eqref{large.int:main.est}. To this end, fix $\beta\in(0,\min\{2s,1\})$, then there are $\eps_0=\eps_0(n,s,\Lambda,\beta)$  and $\theta=\theta(n,s,\Lambda,\beta)$ determined in \autoref{lem.exp.sol}. If $\eps\geq \eps_0\theta$, then we get the desired estimate by making the constant $c$ sufficiently large as $r\geq \eps_0\theta$ and the constants $\eps_0$ and $\theta$ depend only on $n,s,\Lambda,\beta$. Thus we may assume that there is a positive integer $k$ such that $\frac{\eps_0}2\leq \eps/\theta^k< \eps_0$. For any $i\leq k$, the function $u_{\eps,i}(x)\coloneqq u_{\eps}(\theta^ix)-(u_{\eps}(\theta^ix))_{B_1}$ solves
   \begin{equation*}
L_{\eps/\theta^i} u_{\eps,i}=f_i\quad\text{in }B_1,
\end{equation*}
where $f_{i}(x)\coloneqq\theta^{2si}f(\theta^{i}x)$. As $\eps/\theta^i\leq \eps_0$, by \autoref{lem.exp.sol}, we have 
\begin{align*}
    E(u_{\eps,i};B_\theta)\leq \theta^\beta(E(u_{\eps,i}-(u_{\eps,i})_{B_1};B_1)+\|f_i\|_{L^\infty(B_1)}),
\end{align*}
which, after scaling back, yields
\begin{align*}
    E(u_\eps-(u_\eps)_{B_{\theta^{i+1}}};B_{\theta^{i+1}})\leq \theta^\beta (E(u_\eps-(u_\eps)_{B_{\theta^{i}}};B_{\theta^{i}})+\theta^{2si}\|f\|_{L^\infty(B_1)}).
\end{align*}
Thus, after the iteration, we have for any $j\leq k$,
\begin{equation}\label{large.int:ineq1}
\begin{aligned}
    E(u_\eps-(u_\eps)_{B_{\theta^{j}}};B_{\theta^{j}})&\leq \theta^{j\beta} (E(u_\eps-(u_\eps)_{B_{1}};B_{1})+\sum_{i=0}^j\theta^{(2s-\beta)i}\|f\|_{L^\infty(B_1)})\\
    &\leq c\theta^{j\beta}(E(u_\eps-(u_\eps)_{B_{1}};B_{1})+\|f\|_{L^\infty(B_1)})
\end{aligned}
\end{equation}
for some constant $c=c(n,s,\Lambda,\beta)$ independent of $j$. From this, we deduce for any $r\geq \frac{\eps}{\eps_0}$,
\begin{align*}
    E(u_\eps-(u_\eps)_{B_{r}};B_{r})\leq cr^{\beta}(E(u_\eps-(u_\eps)_{B_{1}};B_{1})+\|f\|_{L^\infty(B_1)})
\end{align*}
for some constant $c=c(n,s,\Lambda,\beta)$. Moreover when $r\in [\eps,\frac{\eps}{\eps_0}]$, as $\eps_0$ depends only on $n,s,\Lambda,\beta$, we have 
\begin{align*}
    E(u_\eps-(u_\eps)_{B_{r}};B_{r})\leq c E(u_\eps-(u_\eps)_{B_{\eps/\eps_0}};B_{\eps/\eps_0})\leq cr^{\beta}(E(u_\eps-(u_\eps)_{B_{1}};B_{1})+\|f\|_{L^\infty(B_1)})
\end{align*}
for some constant $c=c(n,s,\Lambda,\beta)$. Thus, for any $r\geq \eps$, we obtain \eqref{large.int:main.est}.
\item Let $s>\frac12$. We are now going to prove \eqref{large.int:main.est} with $\beta=1$. Fix $\gamma\in(1,2s)$ and constants $\eps_0=\eps_0(n,s,\Lambda,\gamma)$ and $\theta=\theta(n,s,\Lambda,\gamma)$ given in \autoref{lem.exp.lip}. As in the previous case, we may assume $\eps\leq \eps_0$  so that there is a positive integer $k$ such that $\frac{\eps_0}2\leq \eps/\theta^k< \eps_0$. Recall the periodic solution $\phi=(\phi_i)_i$ given in \eqref{exp.lip:per.eq} and note that $L_\eps(x+\eps \phi(\cdot/\eps))=0$. Inductively, we show that for any $i\leq k$, there are $a_i\in\bbR$ and $b_i\in \bbR^n$ such that
\begin{equation}\label{large.int:ind.lip}
\begin{aligned}
    &E\left(u_\eps-\sum_{j=1}^i\theta^{\gamma(j-1)}a_i-\sum_{j=1}^i\theta^{(\gamma-1)(j-1)}b_j\cdot (x+\eps \phi(\cdot/\eps))\,;\,B_{\theta^i}\right)\\
    &\leq \theta^{i\gamma}\left(E(u_\eps;B_1)+\sum_{j=1}^i \theta^{(2s-\gamma)(i-j)}\|f\|_{L^\infty(B_1)}\right)
\end{aligned}
\end{equation}
and
\begin{align}\label{large.int:ind2.lip}
    |a_i|+|b_i|\leq c\left(E(u_\eps;B_1)+\sum_{j=1}^i\theta^{(s-\frac12)(i-j)}\|f\|_{L^\infty(B_1)}\right)
\end{align}
for some constant $c=c(n,s,\Lambda,\gamma)$. Since $\|\phi_i\|_{L^\infty}\leq c(n,s,\Lambda)$ by \autoref{lem.bdd.periodic}, using \autoref{lem.exp.lip} and \eqref{funct:bdd.a_0b_0} together with the fact that $\theta$ depends only on $n,s,\Lambda,\gamma$, there are $a_1\in \bbR$ and $b_1\in\bbR^n$ such that \eqref{large.int:ind.lip} and \eqref{large.int:ind2.lip} with $i=1$ holds. Suppose \eqref{large.int:ind.lip} and \eqref{large.int:ind2.lip} are true for any $i=1,2,\ldots l$.  Then we define
\begin{align*}
    u_{\eps,l}(x)\coloneqq \theta^{-l\gamma}\left(u_\eps(\theta^l x)-\sum_{j=1}^l\theta^{\gamma(i-1)}a_i-\sum_{j=1}^l\theta^{(\gamma-1)(j-1)}b_j\cdot (\theta^{l}x+\eps \phi(\theta^l x/ \eps))\right)
\end{align*}
to see that 
\begin{equation}\label{large.int:ind.eq}
\begin{aligned}
    &L_{\eps/\theta^l}u_{\eps,l}=f_l\quad\text{in }B_1,\\
    &E(u_{\eps,l};B_1)\leq  E(u_\eps;B_1)+\sum_{j=1}^l\theta^{(2s-\gamma)(l-j)}\|f\|_{L^\infty(B_1)},
\end{aligned}
\end{equation}
where $f_l(x)\coloneqq \theta^{(2s-\gamma)l}f(\theta^{l}x)$ and the last inequality follows from the inductive assumption.  By applying \autoref{lem.exp.lip}, there are  $a_{l+1}\in\bbR$ and $b_{l+1}\in\bbR^{n}$ such that
\begin{align}\label{large.int:ineq2}
    E(u_{\eps,l}-a_{l+1}-b_{l+1}\cdot(x+(\eps/\theta^l)\phi(\theta^lx/\eps));B_\theta)\leq \theta^{\gamma}(E(u_{\eps,l};B_1)+\|f_{l}\|_{L^\infty(B_1)}).
\end{align}
Moreover, using \eqref{funct:bdd.a_0b_0}, $\|\phi_i\|_{L^\infty}\leq c$, the fact that $\theta$ depends only on $n,s,\Lambda$ and the inequality given in \eqref{large.int:ind.eq}, we have 
\begin{align*}
    |a_{l+1}|+|b_{l+1}|\leq c(n,s,\Lambda,\gamma)E(u_{\eps_l};B_1)\leq c(n,s,\Lambda,\gamma)\left(E(u_\eps;B_1)+\sum_{j=1}^l \theta^{(2s-\gamma)(l-j)}\|f\|_{L^\infty(B_1)}\right),
\end{align*} which shows \eqref{large.int:ind2.lip} with $i=l+1$.
On the other hand, scaling \eqref{large.int:ineq2} back, we get \eqref{large.int:ind.lip} with $i=l+1$. Thus, we inductively prove \eqref{large.int:ind.lip} and \eqref{large.int:ind2.lip} for every $i\leq k$.

\noindent
Using \eqref{large.int:ind.lip} , we deduce for $A_i\coloneqq\sum\limits_{j=1}^i\theta^{\gamma(i-1)}a_i$ and $B_j\coloneqq \sum\limits_{j=1}^i\theta^{(\gamma-1)(j-1)}b_j$,
\begin{align}\label{large.int:exc.gra}
    E(u_\eps-A_i-B_i\cdot (x+\eps\phi(\cdot/\eps));B_{\theta^i})&\leq c\theta^{i\gamma }\left(E(u_\eps;B_1)+\sum_{j=1}^i \theta^{(\gamma-2s)(i-j)}\|f\|_{L^\infty(B_1)}\right).
\end{align}
Note that we will use such an expansion estimate in proving the Liouville theorem given in \autoref{thm.liou} below. In this proof, we choose explicitly the constant $\gamma=\frac{1+2s}{2}\in(1,2s)$. We then use the fact that
\begin{align*}
    E(x;B_{\theta^i})\leq c(n,s)\theta^{i},\quad E(1;B_{\theta^i})\leq c(n,s),\quad \|\phi\|_{L^\infty}\leq c(n,s,\Lambda)
\end{align*}
and \eqref{large.int:ind2.lip} to further estimate
\begin{align*}
    E(u_\eps-A_i;B_{\theta^i})&\leq c(\theta^i+\eps)\left(E(u_\eps;B_1)+\|f\|_{L^\infty(B_1)}\right)
\end{align*}
for some constant $c=c(n,s,\Lambda)$ independent of $i$. Therefore, for every $i\leq k$, we derive
\begin{align}\label{large.int:ineq3}
    \dashint_{B_{\theta^i}}|u_\eps-(u_\eps)_{B_{\theta^i}}|\,dx\leq  2E(u_\eps-A_i;B_{\theta^i})&\leq c(\theta^i+\eps)\left(E(u_\eps;B_1)+\|f\|_{L^\infty(B_1)}\right),
\end{align}
where we have used the fact that for any $r>0$ and $a\in\bbR$, $\int_{B_r}|u-(u)_{B_r}|\,dx\leq 2\int_{B_r}|u-a|\,dx$.
As in the tail estimate given in \cite[Lemma 2.4]{KNS22} together with the fact that the constant $\theta$ depends only on $n,s,\Lambda$, we deduce 
\begin{align*}
    E(u_\eps-(u_\eps)_{B_{\theta^i}};B_{\theta^i})\leq c\left(\sum_{j=0}^{i}\theta^{2s(i-j)} \dashint_{B_{\theta^j}}|u_\eps-(u_\eps)_{B_{\theta^j}}|\,dx+\theta^{-2si}E(u_\eps;B_1)\right)
\end{align*}
for some constant $c=c(n,s,\Lambda)$. Combining this and \eqref{large.int:ineq3} yields that if $\frac{\eps_0}2\leq \eps/\theta^k\leq \eps_0$ for some positive integer $k$, then for any $i\leq k$,
\begin{align*}
    E(u_\eps-(u_\eps)_{B_{\theta^i}};B_{\theta^i})\leq c\theta^i(E(u_\eps;B_1)+\|f\|_{L^\infty(B_1)}).
\end{align*}
By the same argument used in the last part of the previous case, we derive \eqref{large.int:main.est} with $\beta=1$ and $r\in[\eps,1]$. 
\end{itemize}
This completes the proof.
\end{proof}

We end this subsection with the following Liouville theorem.
\begin{theorem}\label{thm.liou}
    Let $\gamma\in(0,2s)$ and $\eps\in(0,1]$. Let $u_\eps$ be a weak solution to 
    \begin{equation}\label{liou:eq}
\left\{
\begin{alignedat}{3}
L_\eps u_\eps&=0&&\qquad \mbox{in  $\bbR^n$}, \\
\|u_\eps\|_{L^\infty(B_R)}&\leq cR^{\gamma}&&\qquad  \mbox{for any $R\geq1$}
\end{alignedat} \right.
\end{equation}
for some constant $c>0$. Then there exist $a\in\bbR$ and $b\in\bbR^n$ such that 
\begin{align*}
    u_\eps(x)=a+b\cdot(x+\eps\phi(x/\eps)),
\end{align*}
where $\phi=(\phi_i)$ is a periodic solution to \eqref{exp.lip:per.eq}. 
\end{theorem}
\begin{proof}
   For $\gamma_0\coloneqq \frac{\gamma+2s}{2}\in(\gamma,2s)$, by \autoref{lem.exp.sol} with $\beta=\gamma_0$ if $s\leq\frac12$ and \autoref{lem.exp.lip} with $\gamma=\gamma_0$ if $s>\frac12$, there are small constants $\eps_0=\eps_0(n,s,\Lambda,\gamma)$ and $\theta=\theta(n,s,\Lambda,\gamma)$. Next, for any large constant $R>1$, there is a positive integer $k$ such that 
   \begin{align}\label{liou:ineq1}
       \frac{\eps_0}2\leq\eps/(R\theta^k)< \eps_0.
   \end{align} 
   Now, let $u_{\eps,R}(x)\coloneqq u_{\eps}(Rx)$ to see that 
   \begin{equation*}
\left\{
\begin{alignedat}{3}
L_{\eps/R} u_{\eps,R}&=0&&\qquad \mbox{in  $\bbR^n$}, \\
\|u_{\eps,R}\|_{L^\infty(B_M)}&\leq c(M R)^{\gamma}&&\qquad  \mbox{for any $M\geq1$}.
\end{alignedat} \right.
\end{equation*}
As in the proof of \eqref{large.int:ineq1} and \eqref{large.int:exc.gra} together with \eqref{liou:ineq1}, for any $i\leq k$, we get $A_i\in\bbR$ and $B_i\in\bbR^n$ such that
\begin{align*}
    E(u_{\eps,R}-A_i-B_i\cdot(x+\eps R^{-1}\phi(\cdot/(\eps R^{-1}));B_{\theta^{i}})\leq c\theta^{i\gamma_0}E(u_{\eps,R};B_1)
\end{align*}
for some constant $c=c(n,s,\Lambda,\gamma)$. Note that $s\leq\frac12$, we have $B_i=0$. After scaling it back, we derive
\begin{align}\label{liou:ineq2}
    E(u_{\eps}-A_i-B_iR^{-1}\cdot(x+\eps\phi(\cdot/\eps));B_{R\theta^i})\leq c\theta^{i\gamma_0}R E(u_{\eps,R};B_1),
\end{align}
where $c=c(n,s,\Lambda,\gamma)$. We now assume $R^{-\delta}<\frac{\eps_0}2$ for $\delta\coloneqq\frac{\gamma_0-\gamma}{2\gamma_0}$, by taking $R>1$ sufficiently large. Then as $R^{1-\delta}<\theta^{-k}$ by \eqref{liou:ineq1}, there is a positive integer $j\leq k-1$ such that
\begin{align*}
    \theta^{-j}\leq R^{1-\delta}< \theta^{-j-1}.
\end{align*}
Using this and \eqref{liou:ineq2} with $i=j$ together with the fact that
\begin{align*}
    E(u_{\eps,R};B_1)\leq c\left( R^{\gamma}+\int_{\bbR^n\setminus B_1}\frac{|u(Rx)|}{|y|^{n+2s}}\,dy\right)\leq cR^{\gamma}
\end{align*}
by the growth assumption in \eqref{liou:eq}, we obtain
\begin{align*}
     E(u_{\eps}-a_R-b_R\cdot(x+\eps\phi(\cdot/\eps));B_{R^\delta})\leq cR^{(\delta-1)\gamma_0+\gamma}\leq cR^{-\delta_0}
\end{align*}
for some constant $c=c(n,s,\Lambda,\gamma)$, where $a_R\in\bbR$, $b_R\in \bbR^n$ and $\delta_0\coloneqq \frac{\gamma_0-\gamma}{2}>0$. In particular, $b_R=0$ when $s\leq\frac12$. By the fact that
\begin{align*}
    L_\eps(u_\eps-a_R-b_R\cdot(x+\eps\phi(\cdot/\eps)))=0\quad\text{in }\bbR^n,
\end{align*}
we use the local boundedness estimate to get that for large $R\geq1$,
\begin{align}\label{liou:ineq3}
    \|u_\eps-a_R-b_R\cdot(x+\eps\phi(\cdot/\eps))\|_{L^\infty(B_{R^\delta/2})}\leq  cE(u_{\eps}-a_R-b_R\cdot(x+\eps\phi(\cdot/\eps));B_{R^\delta})\leq cR^{-\delta_0},
\end{align}
where $c=c(n,s,\Lambda,\gamma)$. Since this is true for any large $R>1$, we are going to prove that there are sequences $(R_i)_i$, $(a_{i})_i$ and $(b_i)_i$ such that
\begin{align}\label{liou:seq1}
    \|u_\eps-a_i-b_i\cdot(x+\eps\phi(\cdot/\eps))\|_{L^\infty(B_{R_i^\delta/2})}\leq cR_i^{-\delta_0}.
\end{align}
with
\begin{align}\label{liou:seq2}
    \lim_{i\to\infty}R_i=\infty,\quad \lim_{i\to\infty}a_i=a_\infty,\quad \lim_{i\to\infty}\|(b_i-b_\infty)\cdot(x+\eps\phi(\cdot/\eps))\|_{L^\infty(B_{R_i^{\delta}/2})}=0
\end{align}
for some $a_\infty\in\bbR$ and $b_\infty\in\bbR^n$. To this end, we first observe that
as $\|\phi\|_{L^\infty}\leq c(n,s,\Lambda)$, for any large  $R>1$,
\begin{align*}
     \frac1c( |a|+R|b|)\leq \|a+b\cdot(x+\eps\phi(\cdot/\eps))\|_{L^\infty(B_R)}\leq c( |a|+R|b|)
\end{align*}
for some constant $c=c(n,s,\Lambda)$. Thus, using this and \eqref{liou:ineq3}, we have for large $R>1$ and any $N\in\N$,
\begin{align*}
    |a_R-a_{2^NR}|+R^{\delta}|b_R-b_{2^NR}|&\leq\sum_{i=1}^{N}|a_{2^iR}-a_{2^{i-1}R}|+(2^{i-1}R)^{\delta}|b_{2^iR}-b_{2^{i-1}R}|\\
    &\leq c\sum_{i=1}^N \|a_{2^iR}-a_{2^{i-1}R}+b_{2^iR}-b_{2^{i-1}R}\cdot(x+\eps\phi(\cdot/\eps))\|_{L^\infty(B_{(2^{i-1}R)^\delta/2})}\\
    &\leq c\sum_{i=1}^N c(2^iR)^{-\delta_0}\leq cR^{-\delta_0},
\end{align*}
where $c=c(n,s,\Lambda,\gamma)$. For fixed large $r>1$, we choose sequences 
\begin{equation}\label{liou:seq}
    R_i\coloneqq 2^ir,\quad a_i\coloneqq a_{2^ir},\quad b_i\coloneqq b_{2^ir}
\end{equation}
to see that
\begin{align*}
    |a_i-a_{i+N}|+R_i^{\delta}|b_i-b_{i+N}|\leq cR_i^{-\delta_0},
\end{align*}
which implies that there are $a_\infty\in \bbR$ and $b_{\infty}\in\bbR^n$ such that 
\begin{align*}
    \lim_{i\to\infty}a_i=a_\infty,\quad \lim_{i\to\infty}b_i=b_\infty\quad\text{and}\quad \lim_{i\to\infty} R_i^{\delta}|b_i-b_\infty|=0.
\end{align*} 
Moreover, we have
\begin{align*}
    \lim_{i\to\infty}\|(b_i-b_\infty)\cdot(x+\eps\phi(\cdot/\eps))\|_{L^\infty(B_{R_i^{\delta}/2})}\leq c\lim_{i\to\infty} R_i^{\delta}|b_i-b_\infty|=0.
\end{align*}
Thus, we have proved that sequences in \eqref{liou:seq} satisfy \eqref{liou:seq1} and \eqref{liou:seq2}. Finally, using such sequences, we obtain
\begin{align*}
    \|u_\eps-a_\infty-b_\infty\cdot(x+\eps\phi(\cdot/\eps))\|_{L^\infty(B_{R_i^\delta/2})}&\leq \|u_\eps-a_i-b_i\cdot(x+\eps\phi(\cdot/\eps))\|_{L^\infty(B_{R_i^\delta/2})} \\
    &\quad+|a_i-a_\infty|+\|(b_i-b_\infty)\cdot(x+\eps\phi(\cdot/\eps))\|_{L^\infty(B_{R_i^{\delta}/2})},
\end{align*}
which tends to 0 as $i\to\infty$. Therefore, we prove $u_\eps=a_\infty+b_\infty\cdot(x+\eps\phi(\cdot/\eps))$ for some $a_\infty\in\bbR$ and $b_\infty\in\bbR^n$ with $b_\infty=0$ if $s\leq\frac12$. This completes the proof.
\end{proof}

\subsection{Boundary estimates}
The main result of this subsection is the following large scale boundary regularity result of order $\beta < s$. Throughout this subsection we do not make any regularity assumption on $K$. 

\begin{proposition}\label{lem.large.bdry}
    Let $\alpha\in(0,1)$, $\eps\in(0,1)$, $\partial\Omega\in C^{1,\alpha}$ and let $u_\eps$ be a weak solution to 
    \begin{equation*}
\left\{
\begin{alignedat}{3}
L_\eps u_\eps&=f&&\qquad \mbox{in  $\Omega\cap B_1$}, \\
u_\eps&=0&&\qquad  \mbox{in $B_1\setminus \Omega$}
\end{alignedat} \right.
\end{equation*}
with $0\in \partial\Omega$, $C^{1,\alpha}$ radius 1 and $C^{1,\alpha}$ norm bounded by 1.
Fix $\beta\in(0,s)$. Then for any $r\in[\eps,1]$,
\begin{align}\label{large.bdry:main.est}
    E(u_\eps-(u_{\eps})_{B_r};B_r)\leq cr^{\beta}(E(u_\eps;B_1)+\|f\|_{L^\infty(B_1)}),
\end{align}
where $c=c(n,s,\Lambda,\alpha,\beta)$.
\end{proposition}
The proof is similar to the interior case in the sense that we first prove a one step improvement and then derive an iterative scheme.

The following lemma is on proving the one step improvement at the boundary.
\begin{lemma}\label{lem.exp.bdry}
    Let $\eps\in(0,1)$ and let $u_\eps$ be a weak solution to 
    \begin{equation*}
\left\{
\begin{alignedat}{3}
L_\eps u_\eps&=f&&\qquad \mbox{in  $\Omega\cap B_1$}, \\
u_\eps&=0&&\qquad  \mbox{in $B_1\setminus \Omega$},
\end{alignedat} \right.
\end{equation*}
where 
\begin{align*}
       B_1\cap\{x_n>0\}\subset \Omega\cap B_1\subset B_1\cap\{x_n>-\eps\}.   
\end{align*} 
Fix $\beta\in(0,s)$. Then there are constants $\eps_0,\theta\in(0,1)$ depending only on $n,s,\Lambda,\beta$ such that for any $\eps\leq \eps_0$,
\begin{align*}
    E(u_\eps;B_\theta)\leq \theta^{\beta}(E(u_\eps;B_1)+\|f\|_{L^\infty(B_1)}),
\end{align*}
where $c=c(n,s,\Lambda,\beta)$.
\end{lemma}
\begin{proof}
    Fix $\beta\in(0,s)$ and a cutoff function $\psi\in C^\infty_c(B_{\frac34})$ with $\psi\equiv 1$ on $B_{\frac12}$. Recall the constants $c_0=c_0(n,s,\Lambda)$ and $c_1=c_1(n,s)$ determined in the first claim \eqref{conv.qual:goal1} and \eqref{exp.bdry:defn.c_1}, respectively.
    Now, we choose $\theta=\theta(n,s,\Lambda,\beta)$. First, by \autoref{lem.reg.bdd1}, for any weak solution $u$ of
   \begin{equation*}
\left\{
\begin{alignedat}{3}
\overline{L} u&=g&&\qquad \mbox{in  $B_{\frac12}\cap \{x_n>0\}$}, \\
u&=0&&\qquad  \mbox{in $B_{\frac12}\setminus\{x_n>0\}$},\\
E(u;B_1)+\|g\|_{L^\infty(B_1)}&\leq c_0&&
\end{alignedat} \right.
\end{equation*}
    with 
    \begin{align}\label{defn.overlineL}
    \overline{L}u(x)\coloneqq\int_{\bbR^n}\frac{u(x)-u(x-z)}{|z|^{n+2s}}\overline{K}\,dz
\end{align}
    for some constant $\overline{K}\in[\Lambda^{-1},\Lambda]$,
    we have for any $\rho\in(0,\frac14]$,
\begin{align*}
    E(u-(u)_{B_\rho};B_\rho)\leq c\rho^{\beta'}(E({u};B_{1})+\|g\|_{L^\infty(B_{1})})\leq c\rho^{\beta'},
\end{align*}
where $\beta'=\frac12(s+\beta)\in(\beta,s)$ and $c=c(n,s,\Lambda,\beta)$, as $c_0$ depends only on $n,s,\Lambda$. On the other hand, since $u\equiv 0$ in $B_{\frac12}\cap\{x_n\leq0\}$, we have
\begin{align*}
        \dashint_{B_\rho}|u(x)|\,dx\leq c(n)\dashint_{B_\rho}|u(x)-(u)_{B_\rho}|\,dx.
    \end{align*}
Therefore, we get 
\begin{align}\label{exp.bdry:ineq.v}
    E(u;B_\rho)\leq cE(u-(u)_{B_\rho};B_\rho)\leq c\rho^{\beta'},
\end{align}
where $c=c(n,s,\Lambda,\beta)$. Thus, we fix $\theta=\theta(n,s,\Lambda,\beta)$ such that
\begin{align}\label{exp.bdry:thetachoice}
    E(u;B_\theta)\leq \frac{\theta^{\beta}}{2}\quad\text{and}\quad c_1\theta^{2s}\leq\frac{\theta^\beta}{4}.
\end{align}

With this choice of $\theta$, we now prove the desired result by contradiction. Assume that there are sequences $(\eps_l)_l$, $(K_l)_l$, $(v_{\eps_l})_l$, $(\Omega_l)_l$ and $(f_l)_l$ such that $u_{\eps_l}$ is a weak solution to
 \begin{equation*}
\left\{
\begin{alignedat}{3}
\int_{\bbR^n}\frac{u_{\eps_l}(x)-u_{\eps_l}(x-z)}{|z|^{n+2s}}K_l(x/\eps_l,(x-z)/\eps_l)\,dy&=f_l&&\qquad \mbox{in  $\Omega_l\cap B_1$}, \\
u_{\eps_l}&=0&&\qquad  \mbox{in $B_1\setminus \Omega_l$},
\end{alignedat} \right.
\end{equation*}
where $\lim\limits_{l\to\infty}\eps_l=0$,
\begin{align*}
       B_1\cap\{x_n>0\}\subset \Omega_l\cap B_1\subset B_1\cap\{x_n>-\eps_l\},
   \end{align*}
but
\begin{align*}
    E(u_{\eps_l};B_\theta)>\theta^{\beta}(E(u_{\eps_l};B_1)+\|f_l\|_{L^\infty(B_1)}).
\end{align*}
After a normalization, we may assume $E(u_{\eps_l};B_1)+\|f_l\|_{L^\infty(B_1)}= 1$ and 
\begin{align}\label{exp.bdry:cont}
     E(u_{\eps_l};B_{\theta})>\theta^{\beta}.
\end{align}
Then by the first claim in \autoref{lem.conv.qual}, there is a weak solution $u$ to
\begin{equation*}
\left\{
\begin{alignedat}{3}
\overline{L} u&=g&&\qquad \mbox{in  $B_{\frac12}\cap \{x_n>0\}$}, \\
u&=0&&\qquad  \mbox{in $B_{\frac12}\setminus\{x_n>0\}$},
\end{alignedat} \right.
\end{equation*}
where $u_{\eps_l}\psi\to u$ in $L^2({\bbR^n})$,
\begin{align*}
    \overline{L}u(x)\coloneqq\int_{\bbR^n}\frac{u(x)-u(x-z)}{|z|^{n+2s}}\overline{K}\,dz
\end{align*}
and
\begin{align*}
    E(u;B_1)+\|g\|_{L^\infty(B_1)}\leq c_0,
\end{align*}
where $c_0$ is determined in \eqref{conv.qual:goal1}. Thus, by \eqref{exp.bdry:ineq.v} and \eqref{exp.bdry:thetachoice}, we have
\begin{align*}
    E(u;B_\theta)\leq \frac{\theta^{\beta}}2.
\end{align*}
On the other hand, using \eqref{exp.bdry:defn.c_1}, the fact that $E(u_{\eps_l};B_1)\leq 1$ and \eqref{exp.bdry:thetachoice}, we have
\begin{align*}
     |E(u_{\eps_l}\psi;B_{\theta})-E(u_{\eps_l};B_\theta)|&\leq c_1\theta^{2s}E(v_{\eps_l};B_1)\leq \frac{\theta^{\beta}}4
\end{align*}
Together with this, \eqref{exp.bdry:cont} gives
\begin{align*}
    E(u_{\eps_l}\psi;B_\theta)> \frac{3\theta^{\beta}}4.
\end{align*}
However, this contradicts 
\begin{align*}
     \frac{3\theta^{\beta}}{4}<\lim_{l\to\infty}E(u_{\eps_l}\psi;B_\theta)=E(u;B_\theta)\leq \frac{\theta^\beta}{2},
\end{align*}
which completes the proof.
\end{proof}

Now, using the smoothness of the domain $\Omega$ and iterating \autoref{lem.exp.bdry}, we prove \autoref{lem.large.bdry}

\begin{proof}[Proof of \autoref{lem.large.bdry}]
First, note that there is a constant $C_2=C_2(n,s)\geq1$ such that for any $a\in(0,\frac12]$, $\rho>0$ and $u\in L^1_{2s}(\bbR^n)$,
\begin{align}\label{large.bdry:const.C_2}
    E(u;B_{\frac{\rho}2})\leq C_2 E(u;B_{\rho}(a\rho))\quad\text{and}\quad E(u;B_{r}(a\rho))\leq C_2 E(u;B_{2\rho}).
\end{align}
Next, fix $\beta\in(0,s)$ and write $\beta'\coloneqq \frac{\beta+s}{2}<s$. By \autoref{lem.exp.bdry} with $\beta$ replaced by $\beta'$, we obtain $\theta=\theta(n,s,\Lambda,\beta)$ and $\eps_0=\eps_0(n,s,\Lambda,\beta)$. We may assume 
\begin{align}\label{large.bdry:theta}
    \theta^{\beta'-\beta}\leq (4(1+C_2))^{-2}
\end{align}
by choosing $\theta$ small enough in the step \eqref{exp.bdry:thetachoice}, as $C_2$ depends only on $n,s$.

By the smoothness assumption of the domain and the fact that $\eps_0=\eps_0(n,s,\Lambda,\beta)$, there is a constant $r_0=r_0(n,s,\Lambda,\alpha,\beta)$ such that for any $\rho\leq r_0$,
\begin{align*}
    B_\rho\cap \{x_n>\rho\eps_1\}\subset \mathcal{R}({\Omega}\cap B_\rho)\subset B_\rho\cap \{x_n>-\rho\eps_1\},
\end{align*}
where $\eps_1\coloneqq\frac1{8}\min\{\eps_0,\theta\}$ and $\mathcal{R}$ is the composition of a translation and a rotation. Therefore, we may assume for any $\rho\leq r_0$,
\begin{align}\label{large.int:bdry}
     B_\rho\cap \{x_n>\rho\eps_1\}\subset {\Omega}\cap B_\rho\subset B_\rho\cap \{x_n>-\rho\eps_1\}.
\end{align}
With these choices, we are going to prove \eqref{large.bdry:main.est}. First, if $\eps\geq \frac{\eps_1\theta}8$, then $r\geq\frac{\eps_1\theta}8$ so that
\eqref{large.bdry:main.est} follows by taking the constant $c$ sufficiently large, as $\eps_1$ and $\theta$ depend only on $n,s,\Lambda,\beta$. Thus we may assume $\eps\leq \frac{\eps_1\theta}8$. 

Assume  $\frac{\eps_1}2< {2\eps}\left(\frac{4}\theta\right)^{k}\leq \eps_1$ for some positive integer $k$. We want prove for any $i\in[0,k]$,
\begin{align}\label{large.bdry:main0}
     E(u_\eps;B_{(\theta/4)^i})\leq c(\theta/4)^{i\beta}(E(u_\eps;B_1)+\|f\|_{L^\infty(B_1)})
\end{align}
for some constant $c=c(n,s,\Lambda,\alpha,\beta)$ independent of $i$. 

To this end, fix $i_0\in\N$ such that 
\begin{equation}\label{large.bdry:i0defn}
    (\theta/4)^{i_0}\leq r_0\leq (\theta/4)^{i_0-1}.
\end{equation}
As $\theta$ and $r_0$ depends only on $n,s,\Lambda,\alpha,\beta$, for any $i\leq i_0+1$, we have
\begin{align}\label{large.bdry:bdry0}
    E(u_\eps;B_{(\theta/4)^i})\leq c(\theta/4)^{i\beta}(E(u_\eps;B_1)+\|f\|_{L^\infty(B_1)})
\end{align}
for some constant $c=c(n,s,\Lambda,\alpha,\beta)$.

Thus, we may assume $i_0+1\leq k$. Let $i\in[i_0+1,k]$, $x_1\coloneqq (0,\ldots,0,\eps_1)$ and $R_i\coloneqq (\frac{\theta}{4})^i\frac12$. Then the function $\widetilde{u}_{\eps,i}(x)\coloneqq u_{\eps}(x+R_ix_1)$ solves
   \begin{equation*}
\left\{
\begin{alignedat}{3}
\widetilde{L}_{\eps} \widetilde{u}_{\eps,i}&=\widetilde{f}_i&&\qquad \mbox{in  $\widetilde{\Omega}_i\cap B_{R_i}$}, \\
\widetilde{u}_{\eps,i}&=0&&\qquad  \mbox{in $B_{R_i}\setminus \widetilde{\Omega}_i$},
\end{alignedat} \right.
\end{equation*}
    where $\widetilde{f}_i(x)\coloneqq R_i^{2s}f(x+R_ix_1)$, $\widetilde{\Omega}_i=\{ x-R_ix_1\,:\, x\in\Omega\}$ and $\widetilde{L}$ is defined as in \eqref{intro:operator} with $K(x,y)$ replaced by $\widetilde{K}(x,y)=K(x+R_ix_1,y+R_i x_1)$.
    In addition, by \eqref{large.int:bdry} and the fact that $R_i\leq r_0$, we have
    \begin{align*}
         B_{R_i}\cap \{x_n>0\}\subset {\Omega}\cap B_{R_i}\subset B_{R_i}\cap \{x_n>-2R_i\eps_1\}.
    \end{align*}
    Thus, for $u_{\eps,i}(x)\coloneqq \widetilde{u}_{\eps,i}(R_i x)$, $f_i(x)\coloneqq R_i^{2s} f(R_i x)$ and $\Omega_i\coloneqq \{x/R_i\,:\, x\in \widetilde{\Omega}_i\}$, we see
    \begin{equation*}
\left\{
\begin{alignedat}{3}
\widetilde{L}_{\eps/r_i} {u}_{\eps,i}&={f}_i&&\qquad \mbox{in  ${\Omega}_i\cap B_{1}$}, \\
{u}_{\eps,i}&=0&&\qquad  \mbox{in $B_{1}\setminus {\Omega}_i$}
\end{alignedat} \right.
\end{equation*}
with 
\begin{align*}
    B_1\cap \{x_n>0\}\subset \Omega_i\cap B_1\subset B_1\cap \{x_n>-2\eps_1\}.
\end{align*}
As $2\eps_1\leq \eps_0$ and $\eps/R_i\leq \eps_0$, we apply \autoref{lem.exp.bdry} with $\beta$ replaced by $\beta'$ to see that 
    \begin{equation}\label{large.bdry:ineq1}
    \begin{aligned}
        E(u_\eps;B_{R_i\theta}(R_ix_1))=E(u_{\eps,i};B_\theta)&\leq \theta^{\beta'}(E(u_{\eps,i};B_1)+\|f_i\|_{L^\infty(B_1)})\\
        &\leq \theta^{\beta'}(E(u_{\eps};B_{R_i}(R_ix_1))+R_i^{2s}\|f\|_{L^\infty(B_1)}).
    \end{aligned}
    \end{equation}
    As $|x_1|\leq \eps_1\leq \frac{\theta}{8}$, using \eqref{large.bdry:const.C_2}, \eqref{large.bdry:theta} and the definition of $R_i=\frac12(\frac\theta4)^i$, we further estimate
    \begin{align*}
        E(u_\eps;B_{(\theta/4)^{i+1}})=E(u_\eps;B_{R_i\theta/2})&\leq C_0^{2}\theta^{\beta'}(E(u_\eps;B_{2R_i})+R_i^{2s}\|f\|_{L^\infty(B_1)})\\
        &\leq (\theta/4)^{\beta}(E(u_\eps;B_{(\theta/4)^{i}})+(\theta/4)^{2si}\|f\|_{L^\infty(B_1)}).
    \end{align*}
    By induction and \eqref{large.bdry:bdry0}, we get for any $i\geq i_0+1$,
    \begin{align*}
        E(u_\eps;B_{(\theta/4)^{i}})\leq (\theta/4)^{\beta (i-i_0)}(E(u_\eps;B_{\theta^{i_0}})+c\|f\|_{L^\infty(B_1)})\leq c(\theta/4)^\beta(E(u_\eps;B_1)+\|f\|_{L^\infty(B_1)})
    \end{align*}
    for some constant $c=c(n,s,\Lambda,\alpha,\beta)$. Combining this and \eqref{large.bdry:bdry0},  we obtain \eqref{large.bdry:main0}.
    Therefore, using again the same argument used in the last part of the proof of \autoref{lem.large.int}, we deduce for any $r\in[\eps,1]$,
    \begin{align*}
        E(u_\eps;B_r)\leq cr^{\beta}(E(u_\eps;B_1)+c\|f\|_{L^\infty(B_1)}),
    \end{align*}
    where $c=c(n,s,\Lambda,\alpha,\beta)$. Lastly by \eqref{large.int:bdry}, we have $u_{\eps}\equiv 0$ in $B_{\rho}\cap\{x_n\leq -\rho \eps_1\}$ for any $\rho\leq r_0$, which implies for any $r\in[0,1]$,
    \begin{align*}
        \dashint_{B_{r}}|u_\eps|^2\,dx\leq c\dashint_{B_{r}}|u_\eps-(u_\eps)_{B_{r}}|^2\,dx
    \end{align*}
    for some constant $c=c(n,s,\Lambda,\alpha,\beta)$, as $r_0$ and $\eps_1$ depend only on $n,s,\Lambda,\alpha,\beta$. Thus, using this, we derive
    \begin{align*}
        E(u_\eps-(u_\eps)_{B_r};B_r)\leq cr^{\beta}(E(u_\eps;B_1)+c\|f\|_{L^\infty(B_1)}),
    \end{align*}
    which completes the proof.
\end{proof}

\subsection{Uniform estimates} In this subsection, we prove \autoref{thm.uni.est} by combining  \autoref{lem.large.int}, \autoref{lem.large.bdry} and regularity estimates for nonlocal equations with sufficiently smooth kernel coefficients.
\begin{proof}[Proof of \autoref{thm.uni.est}] 
We may assume $x_0=0$. First, we prove interior uniform estimates.  

 Let $\beta\in(0,\min\{2s,1\})$ and define $v_\eps(x)\coloneqq \frac{u_\eps(Rx)}{R^\beta}$ to see that
    \begin{align*}
        \int_{\bbR^n}\frac{v_\eps(x)-v_\eps(x-z)}{|z|^{n+2s}}K(Rx/\eps,R(x-z)/\eps)\,dy=g_\eps(x)\quad\text{in }B_1,
    \end{align*}
    where $g_\eps(x)\coloneqq R^{2s-\beta}f(Rx)$. If $\frac{\eps}{R}\leq r\leq 1$, by \autoref{lem.large.int} together with the fact that $v_\eps-(v_\eps)_{B_1}$ also solves the same equation, we have 
    \begin{align}\label{uni.est:ineq1}
        E(v_\eps-(v_\eps)_{B_r};B_r)\leq cr^{\beta}(E(v_\eps-(v_\eps)_{B_1};B_1)+\|g_\eps\|_{L^\infty})
    \end{align}
    for some constant $c=c(n,s,\Lambda,\beta)$. On the other hand, if $r\leq \frac{\eps}{R}$, then we consider a solution $w_\eps(x)\coloneqq \frac{v_\eps(\eps x/R)}{(\eps/R)^\beta}$ to see that 
    \begin{align*}
        L_1(w_{\eps}-(w_\eps)_{B_1})=(\eps/R)^{2s-\beta}g_\eps\quad\text{in }B_1.
    \end{align*}
    Thus using \autoref{lem.reg.int}, we have for $\rho=\frac{r R}{\eps}\leq1$,
    \begin{align*}
        E(w_\eps-(w_\eps)_{B_\rho};B_\rho)\leq c\rho^{\beta}(E(w_\eps-(w_\eps)_{B_{1}};B_{1})+(\eps/R)^{2s-\beta}\|g_\eps\|_{L^\infty(B_{1})})
    \end{align*}
    for some constant $c=c(n,s,\Lambda,\beta,\alpha,[K]_{C^{\alpha}})$. Scaling back, we have 
    \begin{align}\label{uni.est:ineq2}
        E(v_\eps-(v_\eps)_{B_r};B_r)\leq c\left(\frac{r R}{\eps}\right)^\beta (E(v_\eps-(v_\eps)_{B_{\eps/R}};B_{\eps/R})+(\eps/R)^{2s}\|g_\eps\|_{L^\infty(B_{1})}).
    \end{align}
    Thus combining \eqref{uni.est:ineq1} and \eqref{uni.est:ineq2}, we have for any $r\leq1$,
    \begin{equation}\label{uni.est:ineq3}
    \begin{aligned}
        E(u_\eps-(u_\eps)_{B_{rR}};B_{rR})=R^{\beta}E(v_\eps-(v_{\eps})_{B_r};B_R)&\leq c(rR)^\beta (E(v_\eps-(v_\eps)_{B_1};B_1)+\|g_\eps\|_{L^\infty(B_1)})\\
        &\leq cr^\beta (E(u_\eps-(u_\eps)_{B_R};B_R)+R^{2s}\|g\|_{L^\infty(B_R)}).
    \end{aligned}
    \end{equation}
    This gives the desired estimate. For the Lipschitz estimate, using \eqref{large.int:main.est} with $\beta=1$ and Lipschitz estimates for nonlocal equations with H\"older continuous kernel coefficients given in \autoref{lem.reg.int}, we follow the same lines as in the previous part with $\beta=1$. Thus, we have the desired Lipschitz estimate.

Next, we prove the desired boundary estimates. As in the interior case, using \autoref{lem.reg.bdd1} and \autoref{lem.large.bdry}, we deduce for any $r\leq \frac12$ and $\beta\in(0,s)$, if $z_0\in \partial\Omega\cap B_{\frac{R}2}$, then we have
 \begin{align*}
        E(u_\eps-(u_\eps)_{B_{rR}(z_0)};B_{rR}(z_0))
        \leq cr^\beta (E(u_\eps-(u_\eps)_{B_{R/2}};B_{R/2}(z_0))+R^{2s}\|g\|_{L^\infty(B_{R/2}(z_0))})
    \end{align*}
for some constant $c=c(n,s,\Lambda,\alpha,[K]_{C^\delta},\beta,r_0,M)$. In particular, we have used that for any $r\leq \frac{R}2$,
\begin{align*}
    \dashint_{B_r(z_0)}|u|\,dz\leq c(n,\alpha,r_0,M)\dashint_{B_r(z_0)}|u-(u)_{B_r(z_0)}|\,dz,
\end{align*}
as $u_\eps\equiv 0$ on $B_{r}(z_0)\setminus \Omega$.
Therefore, combining this and the interior estimate given in \eqref{uni.est:ineq3} together with a standard covering argument, we get the desired estimate \eqref{uni.est:main.est1}. This completes the proof.
\end{proof}

\section{Green function and Poisson kernel estimates}
\label{sec:green}
In this section, we establish pointwise estimates and regularity estimates of the Green function for the operator $L_\eps$ for any $\eps\in(0,1)$. 

The following theorem provides uniform estimates of Green functions and Poisson kernels. 
\begin{theorem}\label{lem.green}
   Let $\eps\in(0,1)$, $r_0,M,N>0$ $\alpha,\in(0,1]$, $K\in C^\alpha(\mathbb{T}^n)$ and $\partial\Omega\in C^{1,\alpha}$with $C^{1,\alpha}$ radius $r_0$, $C^{1,\alpha}$ norm bounded by $M$ and $\mathrm{diam}(\Omega)\leq N$. Let $G_\eps$ and $P_\eps$ be the Green function and the Poisson kernel associated to $L_\eps$ and $\Omega$. Then we have the following:
   \begin{itemize}
       \item For any $\beta\in(0,s)$, we have
       \begin{equation}
   \begin{aligned}\label{green:uppbdd}
       G_\eps(x,y)\leq c\min\left\{\frac{d_\Omega(x)}{|x-y|},1\right\}^{\beta}\min\left\{\frac{d_\Omega(y)}{|x-y|},1\right\}^{\beta}\frac{1}{|x-y|^{n-2s}},
   \end{aligned}
   \end{equation}
   where $c=c(n,s,\Lambda,[K]_{C^\alpha(\mathbb{T}^n)},\alpha,r_0,M,N,\beta)$.
   \item For any $x\in \Omega$, $\beta\in(0,s)$ and $\gamma\in(0,\min\{2s,1\})$, we have
   \begin{align}\label{green:hol}
       \rho^\gamma[G_\eps(\cdot,y)]_{C^{\gamma}(B_{\rho}(x))}\leq \frac{c}{|x-y|^{n-2s}}\min\left\{\frac{d_\Omega(x)}{|x-y|},1\right\}^{\beta}
   \end{align}
   where $\rho\coloneqq \frac18\min\{|x-y|,d_\Omega(x)\}$ and $c=c(n,s,\Lambda,[K]_{C^\alpha(\mathbb{T}^n)},\alpha,r_0,M,N,\beta,\gamma)$.
   \item For any $\beta\in(0,s)$, we have
   \begin{align}\label{green:pois}
       P_\eps(x,y)\leq \frac{cd_\Omega^{\beta}(x)}{d_\Omega^{2s-\beta}(y)(1+d_\Omega(y))^{\beta}}\frac{1}{|x-y|^n},
   \end{align}
    where $c=c(n,s,\Lambda,[K]_{C^{\alpha}(\mathbb{T}^n)},\alpha,r_0,M,N,\beta)$.
   \end{itemize}
\end{theorem}
Note that since we have proved regularity estimates for the corresponding nonlocal operator in \autoref{thm.uni.est}, pointwise estimates for the associated Green function and Poisson kernel follow by standard arguments. We therefore refer the reader to the appropriate references and omit the details.
\begin{proof}
    By \cite[Lemma 9.3]{KW25} together with \autoref{thm.uni.est}, we deduce the pointwise estimates of the Green function given in \eqref{green:uppbdd}. Next,  by following the same computations given in the proof of \cite[Corollary 1.2]{KW25} together with \eqref{green:uppbdd}, we derive
    \begin{align*}
        E(G_\eps(\cdot,y);B_{2\rho}(x))\leq \frac{c}{|x-y|^{n-2s}}\min\left\{\frac{d_\Omega(x)}{|x-y|},1\right\}^{\beta}
    \end{align*}
    for some constant $c=c(n,s,\Lambda,[K]_{C^\alpha(\mathbb{T}^n)},\alpha,r_0,M,N,\beta)$. 
    As $G_\eps(\cdot,y)$ solves $L_\eps G_\eps(\cdot,y)$ in $B_{2\rho}(x)$, using \autoref{thm.uni.est}, we have 
    \begin{align*}
        \rho^\gamma[G_\eps(\cdot,y)]_{C^\gamma(B_\rho(x))}\leq cE(G_\eps(\cdot,y);B_{2\rho}(x))\leq \frac{c}{|x-y|^{n-2s}}\min\left\{\frac{d_\Omega(x)}{|x-y|},1\right\}^{\beta}
    \end{align*}
    for some constant $c=c(n,s,\Lambda,[K]_{C^\alpha(\mathbb{T}^n)},\alpha,r_0,M,N,\beta,\gamma)$, which implies \eqref{green:hol}. 
    
    For \eqref{green:pois}, using the same argument given in \cite[Theorem 3.2,]{CS98} together with \eqref{green:uppbdd}, we obtain the desired Poisson kernel estimates. 
    This completes the proof.
\end{proof}

Next, using \eqref{green:uppbdd} and \eqref{green:hol}, we prove Sobolev type estimates of Green functions. 
\begin{lemma}\label{lem.fracsob.green}
   Let $\eps\in(0,1]$, $C_1>0$, $\alpha\in(0,s)$, $\beta\in(\alpha,s]$ and $\gamma\in(s,\min\{2s,1\})$. Let $G(x,y)$ satisfy \eqref{green:uppbdd} and \eqref{green:hol} with $c=C_1$. Assume $G(x,\cdot)=G(\cdot,y)=0$ on $\bbR^n\setminus \Omega$. Then we have for any $x\in \bbR^n$ and $y \in \Omega$,
    \begin{align}\label{fracsob.green:main.est}
        \int_{B_\eps(x)}\frac{|G(x,y)-G(z,y)|}{|x-z|^{n+\alpha}}\,dz\leq c\left(\frac{\eps^{\beta-\alpha}+d^{\beta-\alpha}_\Omega(x)}{|x-y|^{n-2s+\beta}}+\frac{d^\beta_\Omega(x)+\eps^\beta}{|x-y|^{n-2s+\alpha+\beta}}\right),
    \end{align}
    where $c=c(n,s,\beta,C_1,\gamma,\beta,\alpha)$.
\end{lemma}
Note that pointwise estimates of $|(-\Delta)^{\frac{s}2}G(\cdot,y)|$ are already established in \cite[Proposition 3.3 and Theorem 1.10]{AFLY25}. Indeed, the technique used in the proof of \cite[Proposition 3.3 and Theorem 1.10]{AFLY25} can also be applied to our case. 
\begin{proof}
    We divide the proof into two parts depending on the location of the point $x$.
    \begin{itemize}
        \item Let $x\in \bbR^n\setminus \Omega$. As $G(x,\cdot)\equiv 0$, we get
        \begin{align}\label{fracsob.green:def.J}
            J=\int_{\Omega\cap B_\eps(x)}\frac{|G(z,y)|}{|x-z|^{n+\alpha}}\,dz=\int_{\Omega_1}(\cdot)+\int_{\Omega_2}(\cdot)\eqqcolon J_1+J_2,
        \end{align}
        where
        \begin{align}\label{fracsob.green:defn.Omega12}
            \Omega_1\coloneqq \{z\in\Omega\cap B_\eps(x)\,:\,|y-z|>|x-y|/2\}\quad\text{and}\quad\Omega_2\coloneqq \{z\in\Omega\cap B_\eps(x)\,:\,|y-z|\leq|x-y|/2\}.
        \end{align}
        For $J_1$, using \eqref{green:uppbdd} and the fact that $d_\Omega(z)\leq |x-z|$ for any $x\in \bbR^n\setminus \Omega$, we have 
        \begin{align*}
            J_1\leq c\int_{\Omega_1}\frac{d_\Omega(z)^{\beta}}{|y-z|^{n-2s+\beta}}\frac{\,dz}{|x-z|^{n+\alpha}}\leq c\int_{\Omega_1}\frac{1}{|x-y|^{n-2s+\beta}}\frac{\,dz}{|x-z|^{n+\alpha-\beta}}
        \end{align*}
        for some constant $c=c(n,s,\beta,C_1)$. Thus, we have 
        \begin{align}\label{fracsob.green:est.J_1}
            J_1\leq \frac{c}{|x-y|^{n-2s+\beta}}\int_{B_\eps(x)}\frac{\,dz}{|x-z|^{n+\alpha-\beta}}\leq c\frac{\eps^{\beta-\alpha}}{|x-y|^{n-2s+\beta}}
        \end{align}
        for some constant $c=c(n,s,\beta,C_1,\alpha)$. For $J_2$, we note $|x-z|\geq\frac12|x-y|$ and $d_\Omega(z)\leq d_\Omega(x)+\eps$ if $z\in \Omega_2$. Using this together with \eqref{green:uppbdd} gives
        \begin{align}\label{fracsob.green:est.J_2}
            J_2\leq \frac{c}{|x-y|^{n+\alpha}}\int_{|y-z|\leq \frac{|x-y|}2}\frac{d_\Omega^\beta(z)\,dz}{|y-z|^{n-2s+\beta}}\leq \frac{c(d^\beta_\Omega(x)+\eps^\beta)}{|x-y|^{n-2s+(\alpha+\beta)}}
        \end{align}
        for some constant $c=c(n,s,C_1,\beta,\alpha)$.
        Combining the estimates $J_1$ and $J_2$ gives \eqref{fracsob.green:main.est} when $x\in \bbR^n\setminus \Omega$.
        \item Let $x\in \Omega$. We observe
        \begin{align*}
            \int_{B_\eps(x)}\frac{|G(x,y)-G(z,y)|}{|x-z|^{n+\alpha}}\,dz&=\int_{B_\eps(x)\setminus \Omega}\frac{|G(x,y)|}{|x-z|^{n+\alpha}}\,dz+\int_{B_\eps(x)\cap\Omega}\frac{|G(x,y)-G(z,y)|}{|x-z|^{n+\alpha}}\,dz
        \end{align*}
        and
        \begin{align}\label{fracsob.green:ineq1}
            \int_{B_\eps(x)\setminus\Omega}\frac{|G(x,y)|}{|x-z|^{n+\alpha}}\,dz\leq \int_{B_\eps(x)\setminus B_{\frac{d_\Omega(x)}2}(x)}\frac{|G(x,y)|}{|x-z|^{n+\alpha}}\,dz \leq \frac{cd_\Omega^{\beta-\alpha}(x)}{|x-y|^{n-2s+\beta}}
        \end{align}
        for some constant $c=c(n,s,\beta,C_1)$, where we have used \eqref{green:uppbdd} for the last inequality.
        Next, we prove
        \begin{align}\label{fracsob.green:est.I}
            I\coloneqq \int_{B_\eps(x)\cap\Omega}\frac{|G(x,y)-G(z,y)|}{|x-z|^{n+\alpha}}\,dz\leq c\left(\frac{\eps^{\beta-\alpha}+d^{\beta-\alpha}_\Omega(x)}{|x-y|^{n-2s+\beta}}+\frac{d^\beta_\Omega(x)+\eps^\beta}{|x-y|^{n-2s+\alpha+\beta}}\right)
        \end{align}
        for some constant $c=c(n,s,\beta,C_1,\gamma)$. First, we write
        \begin{align*}
            I=\int_{(B_\eps(x)\cap\Omega)\setminus B_\rho(x)}\frac{|G(x,y)-G(z,y)|}{|x-z|^{n+\alpha}}\,dz+\int_{(B_\eps(x)\cap\Omega)\cap B_\rho(x)}\frac{|G(x,y)-G(z,y)|}{|x-z|^{n+\alpha}}\,dz\eqqcolon I_1+I_2
        \end{align*}
        for $\rho\coloneqq \frac18\min\{|x-y|,d_\Omega(x)\}$. Using \eqref{green:uppbdd}, we estimate $I_1$ as 
        \begin{align*}
            I_1\leq c\int_{(B_\eps(x)\cap\Omega)\setminus B_\rho(x)}\frac{\,dz}{|x-z|^{n+\alpha}}\frac{d_\Omega(x)^\beta}{|x-y|^{n-2s+\beta}}+\int_{(B_\eps(x)\cap\Omega)\setminus B_\rho(x)}\frac{|G(z,y)|}{|x-z|^{n+\alpha}}\,dz\eqqcolon I_{1,1}+I_{1,2},
        \end{align*}
        where $c=c(n,s,\beta,C_1)$. We further estimate $I_{1,1}$ as
    \begin{align*}
        I_{1,1}\leq \frac{c\rho^{-\alpha}d_\Omega^\beta(x)}{|x-y|^{n-2s+\beta}}\leq  c\left(\frac{d_\Omega^{\beta - \alpha}(x)}{|x-y|^{n-2s+\beta}}+\frac{d_\Omega^{\beta}(x)}{|x-y|^{n-2s+\alpha+\beta}}\right).
    \end{align*}
    For $I_{1,2}$, by the estimate of $J$,
    \begin{align*}
        I_{1,2}\leq J\leq c\left(\frac{\eps^{\beta-\alpha}}{|x-y|^{n-2s+\beta}}+\frac{d^\beta_\Omega(x)+\eps^\beta}{|x-y|^{n-2s+\alpha+\beta}}\right)
    \end{align*}
        Lastly, we use \eqref{green:hol} to estimate $I_2$ as 
        \begin{align*}
            I_2\leq \int_{B_\rho(x)}\frac{c\rho^{-\gamma}\,dz}{|x-z|^{n+\alpha-\gamma}}\frac{1}{|x-y|^{n-2s}}\min\left\{1,\frac{d_\Omega(x)}{|x-y|}\right\}^\beta\leq \frac{c\rho^{-\alpha}}{|x-y|^{n-2s}}\min\left\{1,\frac{d_\Omega(x)}{|x-y|}\right\}^\beta.
        \end{align*}
        Therefore, by this and the estimates of $I_{1,1}$ and $I_{1,2}$, we get \eqref{fracsob.green:est.I}. Combining \eqref{fracsob.green:est.I} and \eqref{fracsob.green:ineq1} yields the desired estimate when $x\in\Omega$.
      \end{itemize}  
        This completes the proof.
\end{proof}

\section{Sharp convergence rates}
In this section, we prove convergence rates for Dirichlet problems depending on regularity assumptions of the given data. The main result is \autoref{lem.conv}, which provides $L^2$-convergence rates when the given data are smooth, which are optimal, except for the case $s=\frac12$. Next, using this result together with a mollification argument, we derive an $L^1$-convergence rate when the boundary datum merely belongs to a global H\"older space. This result is given in \autoref{lem.conv.sub} and provides the main ingredient for proving the $C^s$ regularity in the following section, and which allows us to improve the convergence rate given in \autoref{lem.conv} when $s=\frac12$ (see Section \ref{sec6} for more details).

\subsection{Convergence rates in $L^2$}

We now present the main result of this section.
\begin{theorem}\label{lem.conv}
    Let $\alpha\in(0,1]$, $\eps\in(0,1]$, $r_0,M,N>0$, $h\in C^{1}(\Omega)$, $K\in C^\alpha(\mathbb{T}^n)$ and $\partial\Omega\in C^{1,\alpha}$ with $C^{1,\alpha}$ radius $r_0$, $C^{1,\alpha}$ norm bounded by $M$, $\mathrm{diam}(\Omega)\leq N$ and $0\in\Omega$.
    Let $u_\eps$ and $\overline{u}$ be weak solutions to 
    \begin{equation}\label{conv:eq.ueps}
\left\{
\begin{aligned}
L_\eps u_\eps &= h
&\qquad& \text{in }\Omega, \\
u_\eps &= g
&\qquad& \text{in }\mathbb{R}^n\setminus\Omega,
\end{aligned}
\right.
\qquad
\left\{
\begin{aligned}
\overline{L} \overline{u} &= h
&\qquad& \text{in }\Omega, \\
\overline{u} &= g
&\qquad& \text{in }\mathbb{R}^n\setminus\Omega,
\end{aligned}
\right.
\end{equation}
where $\overline{K}\coloneqq\int_{\mathbb{T}^n\times \mathbb{T}^n}K(x,x-z)\,dz\,dx$, $h\in C^{1+\alpha}(\Omega)$ and $g\in C^{1}(B_2(\Omega)\setminus \Omega)\cap L^1_{2s}(\bbR^n\setminus \Omega)$ with $\nabla g\in L^1_{2s}(\bbR^n\setminus \Omega)$, where the set $B_2(\Omega)$ is defined in \eqref{defn.B_2Omega}. 
Then for any $\gamma\in(0,1)$, we have
\begin{equation}\label{conv:main.est}
\begin{aligned}
    \|u_\eps-\overline{u}\|_{L^2(\Omega)}\leq 
        c&\left(\eps^{2s}\mbox{\large$\chi$}_{\{s<\frac12\}}+\eps^{\gamma}\mbox{\large$\chi$}_{\{s=\frac12\}}+\eps^{2-2s}\mbox{\large$\chi$}_{\{s>1/2\}}\right)\\
        &\times\left(\|h\|_{C^{1+\alpha}(\Omega)}+\|g\|_{C^1(B_2(\Omega)\setminus \Omega)}+\| g\|_{L^1_{2s}(\bbR^n\setminus \Omega)}+\|\nabla g\|_{L^1_{2s}(\bbR^n\setminus \Omega)}\right)
\end{aligned}
\end{equation}
for some constant $c=c(n,s,\Lambda,[K]_{C^\alpha(\mathbb{T}^n)},\alpha,r_0,M,N,\gamma)$. 
\end{theorem}
Before giving the proof, we briefly explain the main idea; see also Subsection \ref{subsec:idea}. From \cite[Section 2]{PZ17}, we deduce the following formal representation of the solution $u_\eps$:
\begin{align*}
    u_\eps(x)\coloneqq \overline{u}(x)+\eps \widetilde{\phi}(x/\eps)\cdot \nabla \overline{u}(x)+\cdots,
\end{align*}
where $\widetilde{\phi}=(\widetilde{\phi}_i)$ is a periodic solution to  $L_1 \widetilde{\phi}_i=-L_1(z_i)$ for each $i\in\{1,2,\ldots,n\}$. As in the local problem (see for instance \cite{AK24}), it would be natural to estimate $L^2$-norm of the function $ \widetilde{w}_\eps\coloneqq u_\eps(x)-(\overline{u}(x)+\eps \widetilde{\phi}(x/\eps)\cdot \nabla \overline{u}(x))$ to get the desired convergence rate. However, in nonlocal problems, the above representation presents several difficulties if we use it directly for Dirichlet problems.

\begin{itemize}
\item Although the given data are smooth, in general it holds $\nabla \overline{u}(x)\notin H^s(\Omega)$, which prevents us from applying any nonlocal operator of order $s$ to $\widetilde{w}_\eps$ in the weak sense. Instead of $\overline{u}$ itself, we consider the solution $\overline{u}_\eps$ to 
      \begin{equation*}
\left\{
\begin{alignedat}{3}
\overline{L}\overline{u}_\eps&=h&&\qquad \mbox{in  $\Omega^{\eps}$}, \\
\overline{u}_\eps&=g&&\qquad  \mbox{in $\bbR^n\setminus \Omega^{\eps}$},
\end{alignedat} \right.
\end{equation*}
where $\Omega^\eps$ is an $\eps$-tubular neighborhood of $\Omega$ (see \eqref{defn.omegaeps} below for the precise definition of $\Omega^\eps$). Thus, we have $\nabla\overline{u}_\eps(x)\in H^s(\Omega)\cap L^1_{2s}(\bbR^n)$ and we can apply any nonlocal operator of order $2s$ to $\nabla\overline{u}_\eps$ in the weak sense.
\item As $L_1(z_i)$ is not well defined when $s\leq\frac12$, we consider the first order corrector $\phi=(\phi_i)_i$ by 
\begin{equation*}
\begin{aligned}
    L_1\phi_i(x)&=\begin{cases}-\int_{\bbR^n}\frac{z_i}{|z|^{n+2s}}K(x,x-z)\,dz\quad\text{if }s>\frac12,\\ 
    -\int_{\eps^{-1}\mathbb{T}^n}\frac{z_i}{|z|^{n+2s}}K(x,x-z)\,dz\quad\text{if }s\leq \frac12.
    \end{cases}
\end{aligned}
\end{equation*}
\item Due to the long-range interaction induced by the nonlocal operator, we need to introduce an additional corrector function $\psi$, defined by 
\begin{equation*}
    L_1\psi(x)=-(\overline{K}(x)-\overline{K})\overline{K}^{-1},
\end{equation*}
where $\overline{K}(x)\coloneqq \dashint_{\mathbb{T}^n}K(x,x-z)\,dz$. Such a corrector has already been used in \cite{CCKW24}.
\end{itemize}
Based on the above discussion, we estimate $w_\eps\coloneqq u_\eps-(\overline{u}_\eps+\eps \phi(\cdot/\eps)\cdot\nabla(\overline{u}_\eps \zeta)+\eps^{2s}\psi(\cdot/\eps)h)$ instead of $\widetilde{w}_\eps(x)$, where $\zeta\in C_c^\infty(B_2(\Omega))$ is a cut-off function with $\zeta\equiv 1$ in $B_1(\Omega)$. We introduce $\zeta$ to get the estimate \eqref{conv:main.est} despite assuming only $g\in C^1(B_2( \Omega)\setminus\Omega)$, but not $ g\in C^1(\bbR^n\setminus \Omega)$. 

In light of the Poisson kernel estimates in \autoref{lem.green} and the regularity of $\overline{u}_\eps,\phi,\psi,h$, we can prove the convergence rates of $\overline{u}-\overline{u}_\eps$ and the remaining terms $\eps \phi(\cdot/\eps)\cdot\nabla(\overline{u}_\eps\zeta)+\eps^{2s}\psi(\cdot/\eps)h$. Thus, it remains to obtain the error estimate for $w_\eps$. To this end,  we observe that $w_\eps$ is the sum of two weak solutions $\theta_\eps$ and $z_\eps$ of 
\begin{equation*}
\left\{
\begin{aligned}
L_\eps \theta_\eps &= 0
&\qquad& \text{in }\Omega, \\
\theta_\eps &= w_\eps
&\qquad& \text{in }\mathbb{R}^n\setminus\Omega,
\end{aligned}
\right.
\qquad
\left\{
\begin{aligned}
L_\eps z_\eps &= F+(-\Delta)^{s}\widetilde{F}
&\qquad& \text{in }\Omega, \\
z_\eps &= 0
&\qquad& \text{in }\mathbb{R}^n\setminus\Omega,
\end{aligned}
\right.
\end{equation*}
for some $F\in L^\infty(\Omega)$ and $\widetilde{F}\in H^s(\bbR^n)$. Using the Poisson kernel estimates in \autoref{lem.green}, we get the error estimate of $\theta_\eps$ in $L^2$. In addition, we obtain pointwise estimates of $F$ and $(-\Delta)^s\widetilde{F}$ in a weak sense, and then apply \autoref{lem.green} and \autoref{lem.fracsob.green} to obtain $L^2$-estimates of $z_\eps$. Thus, combining all estimates yields the desired estimate.

Next, we give the definition of the $\eps$-tubular neighborhood of $C^{1,\alpha}$-domain $\Omega$. To this end, first we recall that there is a regularized distance function $\sigma_\Omega$ given by \autoref{lem.exi.reg.dist}. In addition, for $C^{1,\alpha}$ radius $r_0>0$, $C^{1,\alpha}$ norm bounded by $M>0$ and $\mathrm{diam}(\Omega)\leq N$, there is a constant $C_0=C_0(n,\alpha,r_0,M,N)\geq1$ such that 
    \begin{align}\label{defn.C_0}
       \|\nabla\sigma_\Omega\|_{L^{\infty}(\bbR^n)}\leq \frac{C_0}{16\sqrt{n}}\quad\text{and}\quad \frac{1}{C_0}d_\Omega(x)\leq\sigma_\Omega(x)\leq  C_0d_\Omega(x).
    \end{align}
With this constant $C_0$, we now define $\Omega^\eps$ and $\Omega_\eps$ by 
    \begin{align}\label{defn.omegaeps}
        \Omega^\eps\coloneqq \Omega\cup \{x\in\bbR^n\,:\, \sigma_\Omega(x)\leq C_0\eps\}\quad\text{and}\quad \Omega_\eps\coloneqq \{x\in\Omega\,:\, \sigma_\Omega(x)\geq C_0\eps\}.
    \end{align}

Now we give the proof of \autoref{lem.conv}. Note that some of the auxiliary estimates used in the proof are established separately; see \autoref{lem.rep.f_eps}, \autoref{lem.est.F} and \autoref{lem.est.fracdiv}

\begin{proof}
Throughout the proof, for the sake of convenience in notation, whenever no dependence is specified,  the constant $c$ always depends on the given data $n,s,\Lambda,[K]_{C^\alpha(\mathbb{T}^n)},\alpha,r_0,M,N,\gamma$. In addition, $c=c(\delta)$ means that the constant $c$ depends on the given data and $\delta$.

\textbf{Step 1. Approximate the domain}.  

We recall the $\eps$-tubular neighborhood $\Omega^\eps$ defined by \eqref{defn.omegaeps} for any $\eps\in(0,1]$. First, by \autoref{lem.c1alphadom} and \autoref{lem.dist.tailint}, there is a small constant $\eps_0=\eps_0(n,r_0,M,\alpha)$ such that for any $\eps\leq \eps_0$, $\Omega^\eps$ has $C^{1,\alpha}$ radius  $r_1$ and $C^{1,\alpha}$ norm bounded by $M_1$, where $r_1$ and $M_1$ depend only on $n,\alpha,r_0,M,N$. In addition, \autoref{lem.dist.tailint} is applicable. Since $\eps_0$ depends only on the given data, we may assume 
\begin{align}\label{conv:eps0.Omega}
    \eps\leq \eps_0\quad\text{and}\quad \Omega^{2\eps}\subset B_1(\Omega).
\end{align}
By the normalization, we assume 
\begin{align*}
    \|h\|_{C^{1+\alpha}(\Omega)}+\|g\|_{C^1(B_2(\Omega)\setminus \Omega)}+\|\nabla g\|_{L^1_{2s}(\bbR^n\setminus \Omega)}+\| g\|_{L^1_{2s}(\bbR^n\setminus \Omega)}\leq 1,
\end{align*}
and by properly extending the function $h$, we also assume $h\in C_c^{1+\alpha}(B_2(\Omega))$ with $h\equiv 0$ in $\bbR^n\setminus B_1(\Omega)$,
\begin{align}\label{conv:est.h}
    \| h\|_{C^{1+\alpha}(\bbR^n)}\leq c.
\end{align}
Next, let $\overline{u}_\eps$ be the weak solution to
 \begin{equation}\label{conv:eq.overueps}
\left\{
\begin{alignedat}{3}
\overline{L}\overline{u}_\eps&=h&&\qquad \mbox{in  $\Omega^{\eps}$}, \\
\overline{u}_\eps&=g&&\qquad  \mbox{in $\bbR^n\setminus \Omega^{\eps}$}.
\end{alignedat} \right.
\end{equation}
By \autoref{lem.reg.bdd2}, for $x,y\in \Omega^\eps$ and $\mu\in(0,\min\{2s+\frac{\alpha}2,1\}]$, we have 
\begin{equation}\label{conv:est.ubareps}
\begin{aligned}
&\|\overline{u}_\eps\|_{C^{s}(B_2(\Omega))}\leq c,\quad |\nabla\overline{u}_\eps(x)|\leq cd_\eps^{s-1}(x),\quad
\frac{|\nabla\overline{u}_\eps(x)-\nabla\overline{u}_\eps(y)|}{|x-y|^\mu}\leq c\left(d_\eps^{s-(1+\mu)}(x)+d_\eps^{s-(1+\mu)}(y)\right),\\
&\frac{|\nabla^2\overline{u}_\eps(x)-\nabla^2\overline{u}_\eps(y)|}{|x-y|^\alpha}\leq c\left(d_\eps^{s-(2+\alpha)}(x)+d_\eps^{s-(2+\alpha)}(y)\right)\quad\text{if }s=\frac12,
\end{aligned}
\end{equation}
where we write $ d_\eps(x)\coloneqq \mathrm{dist}(x,\partial\Omega^\eps)$. 
Using the pointwise gradient estimate in \eqref{conv:est.ubareps} together with the fact that $d_\eps(x)\geq \eps$ for any $x\in\Omega$ by \autoref{lem.dist.app}, we deduce
\begin{equation}\label{conv:est.ubareps.L1}
\begin{aligned}
  \|\nabla\overline{u}_\eps\|_{L^2(\Omega)}\leq c\left(1+\left(\int_{1\geq d_\eps(x)\geq \eps}d_\eps^{2s-2}(x)\,dx\right)^{\frac12}\right) \leq c\left(1+\eps^{\frac{2s-1}2}\mbox{\large$\chi$}_{\{s< \frac12\}}+|\ln \eps|^{\frac12}\mbox{\large$\chi$}_{\{s=\frac12\}}\right).
 \end{aligned}
 \end{equation}

Next from \autoref{lem.bdd.periodic} and \autoref{lem.bdd.periodic2}, there are two periodic solutions $\phi_i$ and $\psi$ by 
\begin{equation}\label{conv:eq.phipsi}
\begin{aligned}
    L_1\phi_i(x)&=\begin{cases}-\int_{\bbR^n}\frac{z_i}{|z|^{n+2s}}(K(x,x-z)-\overline{K})\,dz\quad\text{if }s>\frac12,\\ 
    -\int_{\eps^{-1}\mathbb{T}^n}\frac{z_i}{|z|^{n+2s}}(K(x,x-z)-\overline{K})\,dz\quad\text{if }s\leq \frac12,
    \end{cases}
    \\L_1\psi(x)&=-\left(\overline{K}\left({x}\right)-\overline{K}\right)\overline{K}^{-1},
\end{aligned}
\end{equation}
in the sense of \eqref{bdd.periodic:weak.for} and \eqref{bdd.periodic2:weak.for}, respectively, where $\overline{K}(x)\coloneqq\int_{\mathbb{T}^n}K(x,x-z)\,dz$. 
In addition, we have 
\begin{equation}\label{conv:est.phiipsi}
\begin{aligned}
   \|\phi_i\|_{C^{s}(\mathbb{R}^n)}\leq c\quad\text{and}\quad\|\psi\|_{C^{s}(\mathbb{R}^n)}\leq c.
\end{aligned}
\end{equation}
Now we define a function $w_\eps$ by 
\begin{equation*}
    w_\eps\coloneqq u_{\eps}-\overline{u}_\eps-\eps\phi(\cdot/\eps)\cdot \nabla (\overline{u}_\eps\zeta)-\eps^{2s}\psi(\cdot/\eps)h,
\end{equation*}
where $\zeta\in C_c^\infty(B_2(\Omega))$ with
\begin{align}\label{conv:defn.zeta}
    \zeta \equiv 1\quad \text{in } B_1(\Omega)\quad\text{and}\quad \zeta\equiv 0\quad\text{in }\bbR^n\setminus B_{\frac32}(\Omega).
\end{align} 
Our goal is to prove the following estimate 
\begin{align}\label{conv:main.est.weps}
    \|w_\eps-(u_\eps-\overline{u})\|_{L^2(\Omega)}+ \|w_\eps\|_{L^2(\Omega)}\leq c\eps^{s_0},
\end{align}
where 
\begin{align}\label{conv:defn.s_0}
    s_0\coloneqq \begin{cases}
        \min\{2s,2-2s\}&\quad\text{if }s\neq\frac12,\\
        \gamma &\quad\text{if }s=\frac12.
    \end{cases}
\end{align}
\textbf{Step 2. $L^2$-estimate of $w_\eps-(u_\eps-\overline{u})$.} First, we observe 
\begin{align*}
     \|w_\eps-(u_\eps-\overline{u})\|_{L^2(\Omega)}\leq \|\overline{u}_\eps-\overline{u}\|_{L^2(\Omega)}+\eps\| \phi(\cdot/\eps)\nabla \overline{u}_\eps\|_{L^2(\Omega)}+\eps^{2s}\|\psi(\cdot/\eps)h\|_{L^2(\Omega)}\eqqcolon \sum_{i=1}^3I_i.
\end{align*}
For $I_1$, we see $\overline{L}(\overline{u}_\eps-\overline{u})=0$ in $\Omega$, $\overline{u}_\eps-\overline{u}=g-\overline{u}$ in $\Omega^\eps\setminus\Omega$ and $\overline{u}_\eps-\overline{u}\equiv 0$ in $\bbR^n\setminus \Omega^\eps$. Thus, using the Poisson formula in \autoref{lem.green}, we get for any $\beta\in(0,s)$,
\begin{align*}
    I_1\leq  c \left(\int_{\Omega}\left(\int_{\Omega^\eps\setminus \Omega}\frac{|(\overline{u}_\eps-{g})(y)|d_\Omega^{\beta}(x)}{d_\Omega^{2s-\beta}(y)|x-y|^n}\,dy\right)^2\,dx\right)^{\frac12}\leq c\eps^{s}\left(\int_{\Omega}\left(\int_{\Omega^\eps\setminus \Omega}\frac{d_\Omega^{\beta}(x)}{d_\Omega^{2s-\beta}(y)|x-y|^n}\,dy\right)^2\,dx\right)^{\frac12},
\end{align*}
where $c=c(\beta)$ and we have also used
\begin{align*}
    |(\overline{u}_\eps-{g})(y)|\leq |\overline{u}_\eps(y)-g(y_0)|+|g(y_0)-{g}(y)|\leq c|y-y_0|^{s}\leq c\eps^{s}
\end{align*}
for some $y_0\in \partial\Omega^\eps$, by \eqref{conv:est.ubareps}.
Next, we observe for any $y\in\Omega^\eps\setminus \Omega$ and $x\in\Omega$, $d_\Omega(y)\leq |y-x|$, which gives for any $\lambda>0$,
\begin{align*}
    I_1\leq c\eps^{s}\left(\int_{\Omega}\left(\int_{\Omega^\eps\setminus \Omega}\frac{d_\Omega^{\beta}(x)}{d_\Omega^{2s-\beta+\lambda}(y)|x-y|^{n-\lambda}}\,dy\right)^2\,dx\right)^{\frac12}.
\end{align*}
Here, we introduce a parameter $\lambda$ to apply \autoref{lem.singdist}. Indeed, for small $\delta>0$,  by choosing $\beta=s-\frac\delta2$ and $\lambda=(\frac12-s)_++\delta$, we apply the second claim in \autoref{lem.singdist} for $s<\frac23$ and the first claim in \autoref{lem.singdist} for $s\geq\frac23$ to see that
\begin{equation*}
\begin{aligned}
    I_1
    &\leq \begin{cases}c\eps^{s+1-2s+\beta-\lambda}\left(\int_\Omega 1+d_\Omega^{2\beta+2\lambda-2}(x)\,dx\right)^{\frac12}\leq c\eps^{s_0}&\,\text{if }s<\frac23,\\
    c\eps^{s}\left(\int_\Omega 1+d_\Omega^{-4s+4\beta}(x)\,dx\right)^{\frac12}\leq c\eps^{s_0}&\,\text{if }s\geq\frac23.
    \end{cases}
\end{aligned}
\end{equation*}
Next, using \eqref{conv:est.ubareps.L1} and \eqref{conv:est.phiipsi} together with the fact that $|\ln\eps|\leq c(\gamma)\eps^{-1+\gamma}$ for any $\gamma\in(0,1)$,  we have 
\begin{align*}
    I_2+I_3\leq c(\eps\|\nabla \overline{u}_\eps\|_{L^2(\Omega)}+\eps^{2s})\leq c\eps^{s_0}.
\end{align*}
Thus, combining the estimates of $I_i$ for each $i$, we get
\begin{align}\label{conv:weps-error}
    \|w_\eps-(u_\eps-\overline{u})\|_{L^2(\Omega)}\leq c\eps^{s_0}.
\end{align}
\textbf{Step 3. $L^2$-estimate of $w_\eps$.} To do so, we first observe 
\begin{align}\label{conv:defn.f_eps}
    L_\eps w_\eps=L_{\eps}u_\eps-\overline{L}\overline{u}_\eps+(\overline{L}-L_\eps)\overline{u}_\eps-\eps L_\eps(\phi(\cdot/\eps)\cdot\nabla(\overline{u}_\eps\zeta))-\eps^{2s} L_\eps(\psi(\cdot/\eps)h) =: f_{\eps}
\end{align}
in the weak sense. As $w_\eps\in H^{s}(\Omega^{\frac\eps2})\cap L^1_{2s}(\bbR^n)$ from  \eqref{conv:est.ubareps} and the fact that $\phi,\psi\in C^{s+\delta}(\bbR^n)$ for some small $\delta$ by \autoref{lem.bdd.periodic} and \autoref{lem.bdd.periodic2}, we get that $f_\eps\in H^{-s}(\Omega)$. Next we consider two weak solutions $\theta_\eps$ and $z_\eps$ to
\begin{equation}\label{conv:eq.zeps,thetaeps}
\left\{
\begin{aligned}
L_\eps \theta_\eps &= 0
&\qquad& \text{in }\Omega, \\
\theta_\eps &= w_\eps
&\qquad& \text{in }\mathbb{R}^n\setminus\Omega,
\end{aligned}
\right.
\qquad
\left\{
\begin{aligned}
L_\eps z_\eps &= f_\eps
&\qquad& \text{in }\Omega, \\
z_\eps &= 0
&\qquad& \text{in }\mathbb{R}^n\setminus\Omega.
\end{aligned}
\right.
\end{equation}
The existence of the solution $\theta_\eps$ follows from the standard variational argument and the solution $z_\eps$ is obtained by taking $z_\eps=w_\eps-\theta_\eps$. Our aim is to estimate the $L^2$-norms of $\theta_\eps$ and $z_\eps$.

\textbf{Step 3-1. $L^2$-estimate of $\theta_\eps$.}
For $\theta_\eps$, we use the Poisson formula in \autoref{lem.green} together with the fact that $(u_\eps-\overline{u}_\eps)\equiv 0$ in $\bbR^n\setminus \Omega^{\eps}$, $h\equiv0$ in $\bbR^n\setminus B_2(\Omega)$ and \eqref{conv:est.phiipsi} to get that for any $\beta\in(0,s)$,
\begin{equation}\label{conv:thetaeps}
\begin{aligned}
    \|\theta_\eps\|^2_{L^2(\Omega)}&\leq c \left[\int_{\Omega}\left(\int_{\Omega^\eps\setminus \Omega}\frac{|(g-\overline{u}_\eps)(y)|d_\Omega^{\beta}(x)}{d_\Omega^{2s-\beta}(y)|x-y|^n}\,dy\right)^2\,dx+\eps^{4s}\int_{\Omega}\left(\int_{B_2(\Omega)\setminus \Omega}\frac{|h(y)|d_\Omega^{\beta}(x)}{d_\Omega^{2s-\beta}(y)|x-y|^n}\,dy\right)^2\,dx\right]\\
    &\quad+c\eps^2\int_{\Omega}\left(\int_{\bbR^n\setminus \Omega}\frac{|\nabla(\overline{u}_\eps\zeta)(y)|d_\Omega^{\beta}(x)}{d_\Omega^{2s-\beta}(y)(1+d_\Omega(y))^\beta|x-y|^n}\,dy\right)^2\,dx\eqqcolon \sum_{i=1}^3J_i.
\end{aligned}
\end{equation}
By choosing $\beta$ sufficiently close to $s$ and using the fact that $h\equiv 0$ in $\bbR^n\setminus B_2(\Omega)$, as in the estimate of $I_1$, we deduce $J_1+J_2\leq c\eps^{2s_0}$.
Next, using the second estimate in \eqref{conv:est.ubareps}, we obtain
\begin{equation}\label{conv:est.J_3}
\begin{aligned}
    J_3&\leq c\eps^2\left[\eps^{-2+2s}\int_{\Omega}\left(\int_{\Omega^{\frac\eps2}\setminus \Omega}\frac{d_\Omega^{\beta}(x)}{d_\Omega^{2s-\beta}(y)|x-y|^n}\,dy\right)^2\,dx+\int_{\Omega}\left(\int_{\Omega^\eps\setminus \Omega^{\frac\eps2}}\frac{d_\eps^{s-1}(y)d_\Omega^{\beta}(x)}{d_\Omega^{2s-\beta}(y)|x-y|^n}\,dy\right)^2\,dx\right]\\
    &\quad+c\eps^2\int_{\Omega}\left(\int_{\bbR^n\setminus \Omega^\eps}\frac{|\nabla (g\zeta)(y)|d_\Omega^{\beta}(x)}{d_\Omega^{2s-\beta}(y)(1+d_\Omega(y))^\beta|x-y|^n}\,dy\right)^2\,dx \eqqcolon J_{3,1}+J_{3,2}+J_{3,3},
\end{aligned}
\end{equation}
where we have also used the fact that $d_\eps(y)\geq \eps$ in $y\in \Omega^{\frac\eps2}\setminus \Omega$ by \autoref{lem.dist.app}. As in the estimate of $I_1$, by taking $\beta$ close to $s$, we have $J_{3,1}\leq c\eps^{2s_0}$.  
Next, with the choice of $\beta=s-\frac\delta2$  for small $\delta$ and the fact that $d_\eps(y) \leq c\eps\leq cd_\Omega(y)\leq c|x-y|$ and $d_\eps(x)\geq \eps$ for any $x\in \Omega$, $y\in \Omega^{\eps}\setminus\Omega^{\frac\eps2}$ by \autoref{lem.dist.app}, we employ the first claim in \autoref{lem.singdist} to see that
\begin{align*}
    J_{3,2}\leq c\eps^{2-3\delta}\int_\Omega\left(\int_{\Omega^{\eps}\setminus \Omega^{\frac\eps2}}\frac{d_\eps^{-1+\frac\delta2}(y)d_\Omega^{\beta}(x)}{|x-y|^{n-\frac\delta2}}\,dy\right)^2\,dx\leq c\eps^{2-3\delta}\int_\Omega d_\eps^{-2+2\delta}(x)d_\Omega^{2\beta}(x)\,dx\leq c\eps^{2s_0}.
\end{align*}
For $J_{3,3}$, we recall the constant $\eps_0$ and the fact that $\Omega^{\eps_0}\subset B_2(\Omega)$ by \eqref{conv:eps0.Omega} to see that
\begin{align}\label{conv:est.J23}
   J_{3,3}\leq c\eps^2\left(\int_{B_2(\Omega)\cap\{ |y|>1/c\}}\frac{\| g\|_{W^{1,\infty}(B_2(\Omega)\setminus \Omega)}\,dy}{|y|^{n+2s}}+\int_{\Omega^{\eps_0}\setminus \Omega}\frac{\| g\|_{W^{1,\infty}(B_2(\Omega)\setminus \Omega)}d_\Omega^{\beta}(x)}{d_\Omega^{2s-\beta}(y)|x-y|^{n}}\,dy\right),
\end{align}
where we have used the fact that $1+d_\Omega(y),|x-y|\geq 1/c$ for any $x\in\Omega$ and $y\in \bbR^n\setminus \Omega^{\eps_0}$. As in the estimate of $I_1$ with $\eps=\eps_0$, taking $\beta$ close to $s$, we have $J_{3,3}\leq c\eps^2$.
Thus, combining the estimates $J_{3,1},J_{3,2},J_{3,3}$, we have
\begin{align*}
    J_2\leq c\eps^{2s_0}.
\end{align*}
Thus, we collect estimates $J_1,J_2,J_3$ and conclude that 
\begin{align}\label{conv.theta_eps}
    \|\theta_\eps\|_{L^2(\Omega)}\leq c\eps^{s_0}.
\end{align}
\textbf{Step 3-2. $L^2$-estimate of $z_\eps$.}
For $z_\eps$, we observe from \autoref{lem.rep.f_eps} that for any $\varphi\in H^{s}(\bbR^n)$ with $\varphi\equiv 0$ on $\bbR^n\setminus \Omega$,
\begin{align*}
    \langle f_\eps,\varphi\rangle=\langle F,\varphi\rangle+\langle(-\Delta)^s\widetilde{F},\varphi\rangle,
\end{align*}
where $F$ is given by \eqref{conv:defn.F} and $\widetilde{F}$ is given by \eqref{conv:rep.tildeF}. Before applying \autoref{lem.est.F}, we observe from \eqref{conv:est.ubareps} and $\|\nabla g\|_{L^1_{2s}(\bbR^n\setminus\Omega)}\leq 1$ that 
\begin{align*}
    \|\nabla\overline{u}_\eps\|_{L_{2s}^1(\bbR^n)}\leq c.
\end{align*}
Thus, using this and \eqref{conv:est.ubareps}, we apply \autoref{lem.est.F} to see that 
\begin{align*}
    |F(x)|\leq \begin{cases}
       c(\eps d_\eps^{-1-\eps}(x)+\eps^{2s})&\quad\text{if }s<\frac12,\\
        c\eps^{2-2s} d_\eps^{-2+s}(x)+\eps(1+|\ln\eps|)\mbox{\large$\chi$}_{\{s=\frac12\}}&\quad\text{if }s\geq\frac12.
    \end{cases} 
\end{align*}
Let $z_{1,\eps}$ be the solution to
\begin{equation*}
\left\{
\begin{alignedat}{3}
L_\eps z_{1,\eps}&=F&&\qquad \mbox{in  $\Omega$}, \\
z_{1,\eps}&=0&&\qquad  \mbox{in $\bbR^n\setminus \Omega$}.
\end{alignedat} \right.
\end{equation*}
Then, using the pointwise estimate of $F$, \eqref{green:uppbdd} and the first claim in \autoref{lem.singdist} and taking $\beta$ close to $s$ together with the fact that $d_\eps(x)\geq \eps$, we derive
\begin{align}\label{conv:est.z_{1,eps}}
    \|z_{1,\eps}\|^2_{L^2(\Omega)}\leq \left(\int_{\Omega}\int_\Omega G_{\eps}(x,y)F(y)\,dy)\right)^2\,dx \leq c\int_{\Omega}\left(\int_{\Omega}\frac{d^\beta_\Omega(x)d_\Omega^\beta(y)|F(y)|}{|x-y|^{n-2s+2\beta}}\right)^2\,dx\leq c\eps^{2s_0}.
\end{align}
Furthermore, we see
\begin{equation}\label{conv:eq.z1}
\left\{
\begin{alignedat}{3}
L_\eps (z_{\eps}-z_{1,\eps})&=(-\Delta)^{s}\widetilde{F}&&\qquad \mbox{in  $\Omega$}, \\
(z_{\eps}-z_{1,\eps})&=0&&\qquad  \mbox{in $\bbR^n\setminus \Omega$}
\end{alignedat} \right.
\end{equation}
and $\overline{u}_\eps\in C^{1+2s+\sigma/2}_{\mathrm{loc}}(\Omega^\eps)$ for some small $\sigma$ by \autoref{lem.reg.bdd2} together with the fact that $h\in C^{1+\alpha}(B_2(\Omega))$.
 By \autoref{lem.est.fracdiv}, we have  $ \|z_\eps-z_{1,\eps}\|_{L^2(\Omega)}\leq c\eps^{s_0}$. Thus, we obtain $\|z_\eps\|_{L^2(\Omega)}\leq c\eps^{s_0}$.
Combining all estimates $\theta_\eps,z_\eps$ and \eqref{conv:weps-error} yields \eqref{conv:main.est.weps}, which implies the desired result. 
\end{proof}

We now give three lemmas that were used in the proof of \autoref{lem.conv} to estimate the terms given in the representation of $f_\eps$; see  \eqref{conv:defn.f_eps}. 

\begin{lemma}\label{lem.rep.f_eps}
Let $h$, $u_\eps$, $\overline{u}_\eps$, $\phi$, $\psi$, $\zeta$ be given in \eqref{conv:est.h}, \eqref{conv:eq.ueps}, \eqref{conv:eq.overueps}, \eqref{conv:eq.phipsi} and \eqref{conv:defn.zeta}, respectively. 
Let $f_\eps$ be given by \eqref{conv:defn.f_eps}. Then we have for any $\varphi\in H^s(\bbR^n)$ with $\varphi\equiv 0$ in $\bbR^n\setminus \Omega$,
\begin{align*}
    \langle f_\eps,\varphi\rangle=\langle F,\varphi\rangle+\langle(-\Delta)^s\widetilde{F},\varphi\rangle,
\end{align*}
where \begin{equation}\label{conv:defn.F}
\begin{aligned}
    F(x)&=\int_{\mathbb{T}^n}\frac{(\overline{u}_\eps(x)-\overline{u}_\eps(x-z)-z\cdot\nabla\overline{u}_\eps(x))}{|z|^{n+2s}}\left(\overline{K}\left(\frac{x}\eps\right)-K\left(\frac{x}\eps,\frac{x-z}\eps\right)\right)\,dz\\
    &\quad+\int_{\bbR^n\setminus \mathbb{T}^n}\frac{(\overline{u}_\eps(x)-\overline{u}_\eps(x-z)-z\mbox{\large$\chi$}_{\{s>1/2\}}\cdot\nabla\overline{u}_\eps(x))}{|z|^{n+2s}}\left(\overline{K}\left(\frac{x}\eps\right)-K\left(\frac{x}\eps,\frac{x-z}\eps\right)\right)\,dz\\
    &\quad+\eps \int_{\bbR^n}\frac{(\phi_\eps(x)-\phi_\eps(x-z))\cdot (\nabla\overline{u}_\eps(x)-\nabla(\overline{u}_\eps\zeta)(x-z))}{|z|^{n+2s}}K\left(\frac{x}{\eps},\frac{x-z}{\eps}\right)\,dz\\
    &\quad+\eps^{2s}\int_{\bbR^n}\frac{(\psi_\eps(x)-\psi_\eps(x-z))(h(x)-h(x-z))}{|z|^{n+2s}}K\left(\frac{x}{\eps},\frac{x-z}{\eps}\right)\,dz
\end{aligned}
\end{equation}
for $x\in\Omega$, and
\begin{align}\label{conv:rep.tildeF}
   &\langle(-\Delta)^s\widetilde{F},\varphi\rangle\\
    &=-\frac\eps2\iint_{\bbR^n\times\bbR^n}\frac{(\nabla(\overline{u}_\eps\zeta)(x)-\nabla(\overline{u}_\eps\zeta)(x-z))\cdot (\phi_\eps(x)\varphi(x)-\phi_\eps(x-z)\varphi(x-z))}{|z|^{n+2s}}K\left(\frac{x}{\eps},\frac{x-z}{\eps}\right)\,dz\,dx\nonumber\\
    &\quad-\frac{\eps^{2s}}2\iint_{\bbR^n\times\bbR^n}\frac{(h(x)-h(x-z)) (\psi_\eps(x)\varphi(x)-\psi_\eps(x-z)\varphi(x-z))}{|z|^{n+2s}}K\left(\frac{x}{\eps},\frac{x-z}{\eps}\right)\,dz\,dx\nonumber
\end{align}
for $\phi_\eps(x)\coloneqq \phi(x/\eps)$  and $\psi_\eps(x)\coloneqq \psi(x/\eps)$.
\end{lemma}
\begin{proof}
    The proof is based on elementary computations. First we observe
\begin{align*}
   \langle f_\eps,\varphi\rangle= \langle(\overline{L}-L_\eps)\overline{u}_\eps,\varphi\rangle-\langle\eps L_\eps(\phi(\cdot/\eps)\cdot\nabla(\overline{u}_\eps\zeta)),\varphi\rangle-\langle\eps^{2s}L_\eps(\psi(\cdot/\eps)h),\varphi\rangle\eqqcolon \sum_{i=1}^3f_i,
\end{align*}
where we have used $\langle L_\eps u_\eps,\varphi\rangle=\langle \overline{L}\overline{u}_\eps,\varphi\rangle$.  For $f_1$, we can rewrite it as
\begin{align*}
   & \langle f_1,\varphi\rangle\\
   &=\frac12\iint_{\bbR^{d}\times \bbR^n}\frac{(\overline{u}_\eps(x)-\overline{u}_\eps(x-z))(\varphi(x)-\varphi(x-z))}{|z|^{n+2s}}\left(\overline{K}-K\left(\frac{x}\eps,\frac{x-z}\eps\right)\right)\,dz\,dx\\
    &-\frac12\iint_{\bbR^n\times\mathbb{T}^n}\frac{z\cdot(\nabla(\overline{u}_\eps\zeta)(x)\varphi(x)-\nabla(\overline{u}_\eps\zeta)(x-z)\varphi(x-z))}{|z|^{n+2s}}\left(\overline{K}-K\left(\frac{x}\eps,\frac{x-z}\eps\right)\right)\,dz\,dx\\
    &+\frac12\iint_{\bbR^n\times\mathbb{T}^n}\frac{z\cdot(\nabla(\overline{u}_\eps\zeta)(x)\varphi(x)-\nabla(\overline{u}_\eps\zeta)(x-z)\varphi(x-z))}{|z|^{n+2s}}\left(\overline{K}-K\left(\frac{x}\eps,\frac{x-z}\eps\right)\right)\,dz\,dx\\
    &-\frac{\mbox{\large$\chi$}_{\{s>\frac12\}}}2\iint_{\bbR^n\times(\bbR^n\setminus \mathbb{T}^n)}\frac{z\cdot(\nabla(\overline{u}_\eps\zeta)(x)\varphi(x)-\nabla(\overline{u}_\eps\zeta)(x-z)\varphi(x-z))}{|z|^{n+2s}}\left(\overline{K}-K\left(\frac{x}\eps,\frac{x-z}\eps\right)\right)\,dz\,dx\\
    &+\frac{\mbox{\large$\chi$}_{\{s>\frac12\}}}2\iint_{\bbR^n\times(\bbR^n\setminus \mathbb{T}^n)}\frac{z\cdot(\nabla(\overline{u}_\eps\zeta)(x)\varphi(x)-\nabla(\overline{u}_\eps\zeta)(x-z)\varphi(x-z))}{|z|^{n+2s}}\left(\overline{K}-K\left(\frac{x}\eps,\frac{x-z}\eps\right)\right)\,dz\,dx\\
    &\eqqcolon f_{1,1}+f_{1,2}+f_{1,3}+f_{1,4}+f_{1,5}.
\end{align*}
As $\varphi\equiv0$ in $\bbR^n\setminus \Omega$ and \eqref{conv:est.ubareps}, each $f_{1,i}$ is well defined. For $f_2$, we get
\begin{align*}
    &\langle f_2,\varphi\rangle\\
    &=-\frac\eps2\iint_{\bbR^n\times\bbR^n}\frac{(\phi_\eps(x)-\phi_\eps(x-z))\cdot (\nabla(\overline{u}_\eps\zeta)(x)\varphi(x)-\nabla(\overline{u}_\eps\zeta)(x-z)\varphi(x-z))}{|z|^{n+2s}}K\left(\frac{x}{\eps},\frac{x-z}{\eps}\right)\,dz\,dx\\
    &\quad-\frac\eps2\iint_{\bbR^n\times\bbR^n}\frac{(\nabla(\overline{u}_\eps\zeta)(x)-\nabla(\overline{u}_\eps\zeta)(x-z))\cdot (\phi_\eps(x)\varphi(x)-\phi_\eps(x-z)\varphi(x-z))}{|z|^{n+2s}}K\left(\frac{x}{\eps},\frac{x-z}{\eps}\right)\,dz\,dx\\
    &\quad+\eps\iint_{\bbR^n\times\bbR^n}\frac{(\phi_\eps(x)-\phi_\eps(x-z))\cdot (\nabla(\overline{u}_\eps\zeta)(x)-\nabla(\overline{u}_\eps\zeta)(x-z))\varphi(x)}{|z|^{n+2s}}K\left(\frac{x}{\eps},\frac{x-z}{\eps}\right)\,dz\,dx\\
    &\eqqcolon f_{2,1}+f_{2,2}+f_{2,3},
\end{align*}
where we write $\phi_\eps(x)\coloneqq\phi(x/\eps)$.
For $f_3$, we have 
\begin{align*}
    f_3&=-\frac{\eps^{2s}}{2}\iint_{\bbR^n\times\bbR^n}\frac{(\psi_\eps(x)-\psi_\eps(x-z)) (h(x)\varphi(x)-h(x-z)\varphi(x-z))}{|z|^{n+2s}}K\left(\frac{x}{\eps},\frac{x-z}{\eps}\right)\,dz\,dx\\
    &-\frac{\eps^{2s}}{2}\iint_{\bbR^n\times\bbR^n}\frac{(h(x)-h(x-z)) (\psi_\eps(x)\varphi(x)-\psi_\eps(x-z)\varphi(x-z))}{|z|^{n+2s}}K\left(\frac{x}{\eps},\frac{x-z}{\eps}\right)\,dz\,dx\\
    &+{\eps^{2s}}\iint_{\bbR^n\times\bbR^n}\frac{(\psi_\eps(x)-\psi_\eps(x-z))(h(x)-h(x-z))\varphi(x)}{|z|^{n+2s}}K\left(\frac{x}{\eps},\frac{x-z}{\eps}\right)\,dz\,dx\\
    &\eqqcolon f_{3,1}+f_{3,2}+f_{3,3}.
\end{align*}
Now we group the terms $f_{1,1},f_{1,2}\ldots, f_{3,3}$ in a suitable way so that each group can be estimated appropriately. 

First, we observe from \eqref{conv:eq.phipsi} that
\begin{align*}
    f_{1,3}+f_{1,5}+f_{2,1}=0.
\end{align*}
Using the symmetry of the kernel $K(x,y)=K(y,x)$, we get
\begin{align*}
    &f_{1,1}+f_{1,2}+f_{1,4}\\
    &=\iint_{\bbR^n\times \mathbb{T}^n}\frac{(\overline{u}_\eps(x)-\overline{u}_\eps(x-z)-z\cdot\nabla\overline{u}_\eps(x))\varphi(x)}{|z|^{n+2s}}\left(\overline{K}-K\left(\frac{x}\eps,\frac{x-z}\eps\right)\right)\,dz\,dx\\
    &+\iint_{\bbR^n\times (\mathbb{T}^n)^c}\frac{(\overline{u}_\eps(x)-\overline{u}_\eps(x-z)-z\mbox{\large$\chi$}_{\{s>1/2\}}\cdot\nabla\overline{u}_\eps(x))\varphi(x)}{|z|^{n+2s}}\left(\overline{K}-K\left(\frac{x}\eps,\frac{x-z}\eps\right)\right)\,dz\,dx.
\end{align*}
In addition, $\overline{u}_\eps\in C^{2s+\delta}(\Omega^{\frac\eps2})\cap L^\infty(\bbR^n)$ for any $\delta<1$, by \cite[Corollary 2.2.29]{FRRO24}  we see that $\overline{u}_\eps$ is a strong solution. This implies for any $\varphi\in L^1(\Omega)$ with $\varphi\equiv 0$ in $\bbR^n\setminus \Omega$,
\begin{align*}
    &\iint_{\bbR^n\times \mathbb{T}^n}\frac{(\overline{u}_\eps(x)-\overline{u}_\eps(x-z)-z\cdot\nabla\overline{u}_\eps(x))\varphi(x)}{|z|^{n+2s}}\,dz\,dx\\
    &+\iint_{\bbR^n\times (\mathbb{T}^n)^c}\frac{(\overline{u}_\eps(x)-\overline{u}_\eps(x-z)-z\mbox{\large$\chi$}_{\{s>1/2\}}\cdot\nabla\overline{u}_\eps(x))\varphi(x)}{|z|^{n+2s}}\,dz\,dx=\overline{K}^{-1}\int_{\bbR^n}h(x)\varphi(x)\,dx.
\end{align*}
Therefore, using this and \eqref{conv:eq.phipsi}, we have
\begin{align*}
    &f_{1,1}+f_{1,2}+f_{1,4}+f_{3,1}\\
    &=f_{1,1}+f_{1,2}+f_{1,4}+\overline{K}^{-1}\int_{\bbR^n}h(x)\left[(\overline{K}(x/\eps)-\overline{K})\varphi(x)\right]\,dx\\
    &=\iint_{\bbR^n\times \mathbb{T}^n}\frac{(\overline{u}_\eps(x)-\overline{u}_\eps(x-z)-z\cdot\nabla\overline{u}_\eps(x))\varphi(x)}{|z|^{n+2s}}\left(\overline{K}\left(\frac{x}\eps\right)-K\left(\frac{x}\eps,\frac{x-z}\eps\right)\right)\,dz\,dx\\
    &\quad+\iint_{\bbR^n\times \mathbb{T}^n}\frac{(\overline{u}_\eps(x)-\overline{u}_\eps(x-z)-z\mbox{\large$\chi$}_{\{s>1/2\}}\cdot\nabla\overline{u}_\eps(x))\varphi(x)}{|z|^{n+2s}}\left(\overline{K}\left(\frac{x}\eps\right)-K\left(\frac{x}\eps,\frac{x-z}\eps\right)\right)\,dz\,dx.
\end{align*}
So far, we have proved
\begin{align*}
    (f_{1,1}+f_{1,2}+f_{1,4}+f_{3,1})+(f_{1,3}+f_{1,5}+f_{2,1})+f_{2,3}+f_{3,3}=\int_{\bbR^n}F(x)\varphi(x)\,dx,
\end{align*}
where $F$ is determined in \eqref{conv:defn.F}. Next, as $\varphi\equiv 0$ in $\bbR^n\setminus\Omega$, using \eqref{conv:est.ubareps}, \eqref{conv:est.ubareps.L1}, \eqref{conv:est.phiipsi} and \eqref{conv:est.h}, we deduce
\begin{align*}
    |f_{2,2}+f_{3,2}|\leq c(\eps)\|\varphi\|_{H^{s}(\bbR^n)}.
\end{align*}
This implies by the representation formula that 
\begin{align*}
    f_{2,2}+f_{3,2}=\langle (-\Delta)^{s}\widetilde{F},\varphi\rangle,
\end{align*}
for some $\widetilde{F}\in H^{s}(\bbR^n)$ such that \eqref{conv:rep.tildeF} holds. This completes the proof.
\end{proof}

Next we prove the pointwise estimate of $F$ defined in \eqref{conv:defn.F}.

\begin{lemma}\label{lem.est.F} 
Let $r_0,M,N,C_1>0$, $\alpha\in(0,1]$ and let $\Omega$ be a $C^{1,\alpha}$ domain with $C^{1,\alpha}$ radius $r_0$, $C^{1,\alpha}$ norm bounded by $M$, $\mathrm{diam}(\Omega)\leq N$ and $0\in\Omega$. Let $\Omega^\eps$ be the $\eps$-tubular neighborhood by \eqref{defn.omegaeps} for $\eps\leq \eps_0$, where $\eps_0$ is determined in \eqref{conv:eps0.Omega}. Let $\phi_i$ and $\psi$ be periodic functions, and $h \in C_c^{1+\alpha}(B_2(\Omega))$ with
\begin{align}\label{est.F:ass2}
    \|\phi_i\|_{C^s(\mathbb{T}^n)}\leq C_1\quad\text{and}\quad \|\psi\|_{C^s(\mathbb{T}^n)}+\|h\|_{C^{1+\alpha}(\bbR^n)}\leq C_1.
\end{align}
Let $\delta\in(0,1)$ and let $\overline{u}_\eps$ be a function satisfying for any $x,y\in\Omega^\eps$ and $\mu\in (0,\min\{1,2s+\frac{\alpha}2\}]$,
\begin{equation}\label{est.F:ass1}
\begin{aligned}
&\|\overline{u}_\eps\|_{C^s(B_2(\Omega))}+\||\overline{u}_\eps|+|\nabla\overline{u}_\eps|\|_{L^1_{2s}(\bbR^n)}+\|\nabla\overline{u}_\eps\|_{L^\infty(B_2(\Omega)\setminus \Omega^\eps)}\leq C_1,\quad |\nabla\overline{u}_\eps(x)|\leq C_1d_\eps^{s-1}(x),\\
&\frac{|\nabla\overline{u}_\eps(x)-\nabla\overline{u}_\eps(y)|}{|x-y|^\mu}\leq C_1\left(d_\eps^{s-(1+\mu)}(x)+d_\eps^{s-(1+\mu)}(y)\right),\\
&\frac{|\nabla^2\overline{u}_\eps(x)-\nabla^2\overline{u}_\eps(y)|}{|x-y|^\alpha}\leq C_1\left(d_\eps^{s-(2+\alpha)}(x)+d_\eps^{s-(2+\alpha)}(y)\right)\quad\text{if }s=\frac12.
\end{aligned}
\end{equation}
Let $F$ be the function given in \eqref{conv:defn.F}. Then for $x\in\Omega$, we have
\begin{equation}\label{est.F:main}
    \begin{aligned}
    |F(x)|\leq \begin{cases}
        c\left(\eps d_\eps^{-1-s}(x)+\eps^{2s}\right)&\quad\text{if }s<\frac12,\\
        c\eps^{2-2s}d_\eps^{-2+s}(x)+\eps (1+|\ln\eps|)\mbox{\large$\chi$}_{\{s=\frac12\}}&\quad\text{if }s\geq\frac12
    \end{cases}
\end{aligned}
\end{equation}
for some constant $c=c(n,s,\alpha,r_0,M,N,C_1)$.
\end{lemma}
\begin{proof}
Throughout the proof, the constant $c$ depends on $n,s,\Lambda,\alpha,r_0,M,N,C_1$. Thus, for simplicity, we omit this dependence of the constant. Moreover, we write $a\eqsim b$ if there is a constant $c$ such that $\frac{a}{c}\leq b\leq ca$.

First, we estimate the functions $F_1(x)$, $F_2(x)$ and $F_3(x)$ defined by for any $x\in\Omega$,
\begin{align*}
    F_1(x)&\coloneqq \int_{\mathbb{T}^n}\frac{(\overline{u}_\eps(x)-\overline{u}_\eps(x-z)+z\cdot\nabla\overline{u}_\eps(x))}{|z|^{n+2s}}\left(\overline{K}\left(\frac{x}\eps\right)-K\left(\frac{x}\eps,\frac{x-z}\eps\right)\right)\,dz\\
    &\qquad\qquad+\int_{\bbR^n\setminus \mathbb{T}^n}\frac{(\overline{u}_\eps(x)-\overline{u}_\eps(x-z)+z\mbox{\large$\chi$}_{\{s>1/2\}}\cdot\nabla\overline{u}_\eps(x))}{|z|^{n+2s}}\left(\overline{K}\left(\frac{x}\eps\right)-K\left(\frac{x}\eps,\frac{x-z}\eps\right)\right)\,dz,\\
    F_2(x)&\coloneqq   \eps \int_{\bbR^n}\frac{(\phi_\eps(x)-\phi_\eps(x-z))\cdot (\nabla\overline{u}_\eps(x)-\nabla(\overline{u}_\eps\zeta)(x-z))}{|z|^{n+2s}}K\left(\frac{x}{\eps},\frac{x-z}{\eps}\right)\,dz,\\
    F_3(x)&\coloneqq\eps^{2s}\int_{\bbR^n}\frac{(\psi_\eps(x)-\psi_\eps(x-z))(h(x)-h(x-z))}{|z|^{n+2s}}K\left(\frac{x}{\eps},\frac{x-z}{\eps}\right)\,dz.
\end{align*}
Before we estimate each $F_i(x)$ for $i=1,2,3$, we briefly explain the strategy for the estimates.  

We estimate each integrand by considering the two regions $\|z\|_\infty\leq \frac32\eps$ and $\|z\|_\infty> \frac32\eps$. In the former case, both $x$ and $x-z$ lie inside the domain, so we directly apply the fundamental theorem of calculus together with the regularity assumptions of the given functions. In the latter case, $x-z$ can be outside of the domain $\Omega^\eps$. We therefore divide the integration domain into three regions, namely, the region far away from the domain, the region near the boundary, the region inside the domain (see \eqref{conv:defn.A_i}), and then  apply the regularity estimates of the given functions. Furthermore, when estimating the integral for the case $\|z\|_\infty>\frac32\eps$ in $F_1(x)$, we estimate the integrand on each small cube $\eps\mathbb{T}^n(p)$ for $p\in\eps\Z^n$ with $\|p\|_\infty>\eps$. Together with the periodicity of $K$, this allows us to apply the regularity estimates of $\overline{u}_\eps$ appropriately.

Now we estimate all terms $F_i$ for $i=1,2,3$.

\textbf{Estimate of $F_1$.} We divide $F_1(x)$ into two parts by
    \begin{align*}
        F_1(x)\coloneqq \int_{\|z\|_\infty\leq\frac32\eps}(\cdot)\,dz+\int_{\|z\|_\infty>\frac32\eps}(\cdot)\,dz\eqqcolon F_{1,1}(x)+F_{1,2}(x).
    \end{align*}
    For the estimate of $F_{1,1}$, we observe for any $x\in \Omega$, $\|z\|_\infty\leq \frac32\eps$ and $t\in[0,1]$, $x-tz\in \Omega^{\frac\eps2}$ by \autoref{lem.dist.app}. This implies by the fundamental theorem of calculus, \eqref{est.F:ass1} and \eqref{est.F:ass1} that
    \begin{align*}
        |F_{1,1}(x)|
        &\leq c\int_{\|z\|_\infty\leq \frac32\eps}\int_{0}^1\frac{d_\eps^{s-(1+s_1)}(x-tz)+d_\eps^{s-(1+s_1)}(x)}{|z|^{n+2s-(1+s_1)}}\,dt\,d\xi\,dz,
    \end{align*}
    where we write 
    \begin{equation*}
        s_1\coloneqq \min\{2s,1\}.
    \end{equation*}
    By \autoref{lem.dist.app}, we have $d_\eps(x)\geq 8\sqrt{n}\eps$, which gives $|z|<\frac32\eps\sqrt{n}< \frac{d_\eps(x)}{2}$. With the aid of \eqref{dist.aux2:main1}, we estimate $F_{1,1}$ as
    \begin{align}\label{conv:est.F11}
        |F_{1,1}(x)|\leq c\int_{|z|<{\frac32\eps\sqrt{n}}}\frac{d_\eps^{s-(1+s_1)}(x)}{|z|^{n+2s-(1+s_1)}}\,dz \leq c\nu(x),
    \end{align}
    where $\nu(x)$ is defined by
    \begin{equation}\label{est.F:defn.nu}
    \nu(x)=\eps^{2-2s}d_\eps^{-2+s}(x)\mbox{\large$\chi$}_{\{s\geq\frac12\}}+\eps d_\eps^{-1-s}(x)\mbox{\large$\chi$}_{\{s<\frac12\}}.
\end{equation}
    For the estimate of $F_{1,2}(x)$, we observe that
    \begin{align*}
         \bbR^n\setminus \{\|z\|_\infty>3\eps/2\}=\bigcup_{p\in \eps\Z^n,\,\|p\|_\infty>\eps}\eps\mathbb{T}^n(p)\quad\text{and}\quad\int_{\eps\mathbb{T}^n(p)}K\left(\frac{x}{\eps},\frac{x-z}{\eps}\right)\,dz=\int_{\eps\mathbb{T}^n(p)}\overline{K}\left(\frac{x}{\eps}\right)\,dz
    \end{align*} 
    for all $p\in \eps\Z^n$, where $\eps\Z^n\coloneqq\{\eps q\,:\, q\in \Z^n\}$.
    Using these, we have
    \begin{equation}\label{est.F:F12.rep}
    \begin{aligned}
        F_{1,2}(x)&=\sum_{p\in \eps\Z^n,\|p\|_\infty>\eps}\int_{\eps\mathbb{T}^n(p)}H(x,z)\left[K\left(\frac{x}{\eps},\frac{x-z}{\eps}\right)-\overline{K}\left(\frac{x}{\eps}\right)\right]\,dz\\
        &=\sum_{p\in \eps\Z^n,\|p\|_\infty>\eps}\int_{\eps\mathbb{T}^n(p)}\left[(H(x,z)-\widetilde{H}(x,p))K\left(\frac{x}{\eps},\frac{x-z}{\eps}\right)+(\widetilde{H}(x,p)-H(x,z))\overline{K}\left(\frac{x}{\eps}\right)\right]\,dz,    
        \end{aligned}
        \end{equation}
where 
\begin{align*}
        H(x,z)\coloneqq\begin{cases}\frac{\overline{u}_\eps(x)-\overline{u}_\eps(x-z)-z\cdot \nabla\overline{u}_\eps(x)}{|z|^{n+2s}}&\quad\text{if }s>\frac12,\\
        \frac{\overline{u}_\eps(x)-\overline{u}_\eps(x-z)-z\mbox{\large$\chi$}_{\{z\in \mathbb{T}^n\}}\cdot\nabla\overline{u}_\eps(x)}{|z|^{n+2s}}&\quad\text{if }s\leq\frac12
        \end{cases}
    \end{align*} 
     and
    \begin{align*}
        \widetilde{H}(x,p)\coloneqq\begin{cases}\frac{\overline{u}_\eps(x)-(\overline{u}_\eps)_{\eps\mathbb{T}^n(x-p)}-p\cdot \nabla\overline{u}_\eps(x)}{|p|^{n+2s}}&\quad\text{if }s>\frac12,\\
        \frac{\overline{u}_\eps(x)-(\overline{u}_\eps)_{\eps\mathbb{T}^n(x-p)}-p\mbox{\large$\chi$}_{\{z\in \mathbb{T}^n\}}\cdot\nabla\overline{u}_\eps(x)}{|p|^{n+2s}}&\quad\text{if }s\leq\frac12.
        \end{cases}
    \end{align*} 
We now consider three sets $A_1$, $A_2$ and $A_3$ defined by 
\begin{equation}\label{conv:defn.A_i}
    \begin{aligned}
        A_1&\coloneqq \{p\,:\, x-p\in \bbR^n\setminus \Omega^{\frac32\eps}\},\\
        A_2&\coloneqq \{p\,:\, x-p\in \Omega^{\frac32\eps}\setminus \Omega_{\frac32\eps}\},\\
        A_3&\coloneqq \{p\,:\, x-p\in  \Omega_{\frac32\eps}\}.
    \end{aligned}
    \end{equation}
    Then we observe the following: 
    \begin{equation}\label{conv:obs.A_i}
    \begin{aligned}
        &\mbox{If $z\in \eps\mathbb{T}^n(p)\cap A_1$ for some $p\in \eps\mathbb{Z}^n$, then $\eps\mathbb{T}^n(x-p)\subset\bbR^n\setminus \Omega^\eps$, $|z|\geq\frac{d_\eps(x)}{2}$.}\\
        &\mbox{If $z\in \eps\mathbb{T}^n(p)\cap A_2$ for some $p\in \eps\mathbb{Z}^n$, then $\eps\mathbb{T}^n(x-p)\subset \Omega^{2\eps}\setminus\Omega_{2\eps}$ and $d_\eps(x-z)+d_\eps(x-p)\leq c\eps$.}\\
        &\mbox{If $x-z\in \Omega^{2\eps}\setminus \Omega_{2\eps}$ and $|z|\leq \frac{d_\eps(x)}{4}$, then $d_\eps(x)\eqsim \eps$.}
    \end{aligned}
    \end{equation}
The first observation follows from \autoref{lem.dist.app} and \eqref{dist.aux2:main2}. Next, by the definition of the set $A_2$ and \autoref{lem.dist.app}, we can verify the second observation. For the third one, as $|z|\leq \frac{d_\eps(x)}{4}$, by \eqref{dist.aux2:main1} and \autoref{lem.dist.app}, we observe $\frac12d_\eps(x)\leq d_\eps(x-z)\leq \frac32 d_\eps(x)$ and $8\sqrt{n}\eps\leq d_\eps(x)$. Thus, combining these observations together with the second observation yields the third one in \eqref{conv:obs.A_i}.

Now we write $F_{1,2}(x)$ as
\begin{align*}
    F_{1,2}(x)=\sum_{i=1}^3\sum_{p\in \eps\Z^n\cap A_i,\|p\|_\infty>\eps}\int_{\eps\mathbb{T}^n(p)}(\cdot)\eqqcolon \sum_{i=1}^3 F_{1,2,1}
\end{align*}
and estimate $F_{1,2,i}$ for each $i=1,2,3$.
\begin{itemize}[leftmargin=*]
    \item Estimate of $F_{1,2,1}$ and $F_{1,2,2}$. First, using the fact that $z\in \eps\mathbb{T}^n(p)$ and $p\in \eps\Z^n$ with $\|p\|_{\infty}
    >\eps$,
    \begin{align}\label{est.F:pandxi}
        |p|\eqsim |z|\quad\text{and}\quad\left|\frac1{|p|^{n+2s}}-\frac1{|z|^{n+2s}}\right|\leq \frac{c\eps}{|z|^{n+2s+1}},
    \end{align}
    we estimate $F_{1,2,i}$ as 
    \begin{equation}\label{est.F:F121}
    \begin{aligned}
        |F_{1,2,i}|&\leq c\sum_{p\in\eps\Z^n\cap A_i,\|p\|_\infty>\eps}\int_{\eps\mathbb{T}^n(p)}\frac{|\overline{u}_\eps(x-z)-(\overline{u}_\eps)_{\eps\mathbb{T}^n(x-p)}|+\eps|\nabla\overline{u}_\eps(x)|}{|z|^{n+2s}}\,dz\\
        &+c\sum_{p\in\eps\Z^n\cap A_i,\|p\|_\infty>\eps}\eps\int_{\eps\mathbb{T}^n(p)}\frac{|\overline{u}_\eps(x)-\overline{u}_\eps(x-z)|+|z||\nabla\overline{u}_\eps(x)|}{|z|^{n+2s+1}}\,dz.
    \end{aligned}
    \end{equation}
    For $i=1$, by the first observation, \eqref{est.F:ass1}, we further estimate
    \begin{align*}
        |F_{1,2,1}|&\leq c\sum_{p\in\eps\Z^n\cap A_1,\|p\|_\infty>\eps}\int_{\eps\mathbb{T}^n(p)}\frac{|\overline{u}_\eps(x-z)-(\overline{u}_\eps)_{\eps\mathbb{T}^n(x-p)}|}{|z|^{n+2s}}\,dz(\eqqcolon H_{1})\\
        &\quad+c\eps\left[\int_{|z|>\frac{d_\eps(x)}2}\frac{d_\eps^{s-1}(x)}{|z|^{n+2s}}+\int_{1>|z|>\frac{d_\eps(x)}2}\frac{\,dz}{|z|^{n+s+1}}+\int_{|z|>1}\frac{|\overline{u}_\eps(x)-\overline{u}_\eps(x-z)|}{|z|^{n+2s+1}}\,dz\right]\\
        &\leq c(H_1+\eps d_\eps^{-s-1}(x)).
    \end{align*}
     Using the fact that $\eps\mathbb{T}^n(p)\subset \bbR^n\setminus \Omega^\eps$ by \eqref{conv:obs.A_i}, we estimate $H_1$ as
    \begin{align*}
        H_{1}&\leq  c\left(\sum_{\substack{p\in \eps\Z^n,\\\|p\|_\infty\geq1/2}}\int_{\eps\mathbb{T}^n(p)\cap A_1}\frac{|\overline{u}_\eps(x-z)-(\overline{u}_\eps)_{\eps\mathbb{T}^n(x-p)}|}{|z|^{n+2s}}\,dz+\int_{|z|\geq\frac{d_\eps(x)}2}\frac{\eps\|\nabla u_\eps\|_{L^\infty(B_1(\Omega)\setminus \Omega^\eps)}}{|z|^{n+2s}}\,dz\right).
    \end{align*}
     By  the fact that $|p|\eqsim |z|$ by $|p|>\frac12$ and $z\in \eps\mathbb{T}^n(p)$, \eqref{conv:obs.A_i} and the Poincar\'e inequality, we have 
    \begin{equation}\label{est.F:h111}
    \begin{aligned}
        H_{1}&\leq c\left(\sum_{\substack{p\in \eps\Z^n,\|p\|_\infty>1/2,\\
        \eps\mathbb{T}^n(x-p)\subset \bbR^n\setminus \Omega^\eps}}\int_{\eps\mathbb{T}^n(p)}\frac{|\overline{u}_\eps(x-z)-(\overline{u}_\eps)_{\eps\mathbb{T}^n(x-p)}|}{|p|^{n+2s}}\,dz+\eps d_\eps^{-2s}(x)\right)\\
        &\leq c\eps\left(\sum_{\substack{p\in \eps\Z^n,\|p\|_\infty>1/2,\\
        \eps\mathbb{T}^n(x-p)\subset \bbR^n\setminus \Omega^\eps}}\int_{\eps\mathbb{T}^n(p)}\frac{|\nabla \overline{u}_\eps(x-z)|}{|p|^{n+2s}}\,dz+d_\eps^{-2s}(x)\right).
    \end{aligned}
    \end{equation}
    
    Moreover, we have for any $|p|>\frac12$, $(|x-z|+1)\leq |x|+|p|+|z-p|+1\leq c|p|$, which gives
    \begin{align}\label{est.F:tail.h111}
        \sum_{\substack{p\in \eps\Z^n,\|p\|_\infty>1/2,\\
        \eps\mathbb{T}^n(x-p)\subset \bbR^n\setminus \Omega^\eps}}\int_{\eps\mathbb{T}^n(p)}\frac{|\nabla \overline{u}_\eps(x-z)|}{|p|^{n+2s}}\,dz\leq c\int_{x-z\in\bbR^n\setminus \Omega^\eps}\frac{|\nabla \overline{u}_\eps(x-z)|}{(1+|x-z|)^{n+2s}}\,dz\leq c\|\nabla \overline{u}_\eps\|_{L^1_{2s}(\bbR^n)}.
    \end{align} 
Plugging this into \eqref{est.F:h111}, we deduce
\begin{equation*}
    |F_{1,2,1}|\leq c\eps d_\eps^{-s-1}(x).
\end{equation*}
For $i=2$, by the second observation in \eqref{conv:obs.A_i}, $C^s$-regularity estimates and pointwise gradient estimates in \eqref{est.F:ass1} and the second and third observations in \eqref{conv:obs.A_i}, we deduce
\begin{align*}
    |F_{1,2,2}|&\leq c\left(\int_{\eps<|z|}\frac{\eps^s+\eps d_\eps^{s-1}(x)}{|z|^{n+2s}}\mbox{\large$\chi$}_{\{x-z\in \Omega^{2\eps}\setminus \Omega_{2\eps}\}}\,dz+\eps\int_{\eps<|z|}\frac{|z|^s+|z| d_\eps^{s-1}(x)}{|z|^{n+2s+1}}\mbox{\large$\chi$}_{\{x-z\in \Omega^{2\eps}\setminus \Omega_{2\eps}\}}\,dz\right)\\
   & \eqqcolon H_2+H_3.
\end{align*}
First, we further estimate $H_2$ as
\begin{align*}
         H_2&\leq c\left(\int_{\eps<|z|<\frac{d_\eps(x)}4}\frac{\eps^s+\eps d_\eps^{s-1}(x)}{|z|^{n+2s}}\mbox{\large$\chi$}_{\{x-z\in \Omega^{2\eps}\setminus \Omega_{2\eps}\}}\,dz+\int_{|z|>\frac{d_\eps(x)}4}\frac{\eps^{s}+\eps d_\eps^{s-1}(x)}{|z|^{n+2s}}\mbox{\large$\chi$}_{\{x-z\in \Omega^{2\eps}\setminus \Omega_{2\eps}\}}\,dz\right)\\
        &\leq c\left(\int_{\frac{d_\eps(x)}{c}<|z|<\frac{d_\eps(x)}4}\frac{\eps d_\eps^{s-1}(x)}{|z|^{n+2s}}\,dz+\int_{|z|>\frac{d_\eps(x)}4}\frac{\eps^{s}}{|z|^{n+2s}}\mbox{\large$\chi$}_{\{x-z\in \Omega^{2\eps}\setminus \Omega_{{2\eps}}\}}\,dz+\eps d_\eps^{-s-1}(x)\right)\\
        &\leq c\eps d^{-1-s}_\eps(x),
    \end{align*}
    where we have also used \eqref{dist.tailint:main4} for the last inequality. 
    Similarly, we estimate $H_3$ as 
    \begin{align*}
       H_{3}&\leq c\eps\left(\int_{\eps<|z|<\frac{d_\eps(x)}4}\frac{|z|^s+|z| d^{s-1}_\eps(x)}{|z|^{n+2s+1}}\mbox{\large$\chi$}_{\{x-z\in \Omega^{2\eps}\setminus \Omega_{2\eps}\}}\,dz+\int_{|z|>\frac{d_\eps(x)}4}\frac{|z|^s+|z| d^{s-1}_\eps(x)}{|z|^{n+2s+1}}\,dz\right)\\
        &\leq c\eps \left(\int_{\frac{d_\eps(x)}{c}<|z|<\frac{d_\eps(x)}4}\frac{|z|^s+|z| d^{s-1}_\eps(x)}{|z|^{n+2s+1}}\,dz+d^{-1-s}_\eps(x)\right)\\
        &\leq c\eps d^{-1-s}_\eps(x).
    \end{align*}
    Using the estimates $H_2$ and $H_3$, we have
    \begin{equation*}
        |F_{1,2,2}|\leq c\eps d_\eps^{-1-s}(x).
    \end{equation*}
    \item Estimate of $F_{1,2,3}$. To this end, we write 
    \begin{align*}
        |F_{1,2,3}|&\leq \left|\sum_{p\in\eps\Z^n\cap A_3,\eps<\|p\|_\infty<\frac{d_\eps(x)}4}\int_{\eps\mathbb{T}^n(p)}H(x,z)\left[K\left(\frac{x}{\eps},\frac{x-z}{\eps}\right)-\overline{K}\left(\frac{x}{\eps}\right)\right]\,dz\right|\\
        &\quad+\left|\sum_{p\in\eps\Z^n\cap A_3,\|p\|_\infty\geq\frac{d_\eps(x)}4}\int_{\eps\mathbb{T}^n(p)}H(x,z)\left[K\left(\frac{x}{\eps},\frac{x-z}{\eps}\right)-\overline{K}\left(\frac{x}{\eps}\right)\right]\,dz\right|\eqqcolon H_4+H_5.
    \end{align*}
    To estimate $H_4$, we observe that if $d_\eps(x)<100{n}(1+C)^2\eps$ for some large constant $C\geq C_0$, where the constant $C_0$ is determined in \eqref{defn.C_0}, then as in the estimate of $F_{1,1}$, we have
    \begin{equation*}
        H_4\leq c\int_{|z|<c\eps}|H(x,z)|\,dz\leq  c\eps d_\eps^{-1-s}(x),
    \end{equation*}
    where the constant $c$ depends only on the given data and $C$.
    Thus, we assume $d_\eps(x)\geq 100{n}(1+C)^2\eps$. Then, by taking $C$ sufficiently large depending only on the given data, we have 
    \begin{equation}\label{est.F:inc.H_4}
       \left\{\frac{3\eps}2<\|z\|_\infty<\frac{2k+1}{2}\eps\right\}\subset \bigcup_{p\in\eps\Z^n\cap A_3,\eps<\|p\|_\infty<\frac{d_\eps(x)}4}\eps\mathbb{T}^n(p)
    \end{equation}
    for some integer $k\in \N$ such that 
    \begin{equation}\label{est.F:defn.k}
        \frac{d_\eps(x)}{8(1+C)^2{n}}\leq \frac{2k+1}{2}\eps\leq \frac{d_\eps(x)}{4(1+C)^2{n}}.
    \end{equation}
    Now, we choose a cutoff function $\varphi\in C_c^\infty$ such that
    \begin{equation}\label{est.F:varphi}
        \varphi\equiv 1 \text{ in }\frac{5\eps}2<\|z\|_\infty<\frac{k+1}{2}\eps,\quad \varphi\equiv 0 \text{ in }\|z\|_\infty<\frac{3\eps}2\text{ and }\|z\|_\infty>k\eps.
    \end{equation}
    Using this function, we estimate $H_4$ as
    \begin{align*}
        H_4&\leq \left|\sum_{p\in\eps\Z^n\cap A_3,\eps<\|p\|_\infty<\frac{d_\eps(x)}4}\int_{\eps\mathbb{T}^n(p)}H(x,z)\varphi(z)\left[K\left(\frac{x}{\eps},\frac{x-z}{\eps}\right)-\overline{K}\left(\frac{x}{\eps}\right)\right]\,dz\right|\\
        &\quad+\left|\sum_{p\in\eps\Z^n\cap A_3,\eps<\|p\|_\infty<\frac{d_\eps(x)}4}\int_{\eps\mathbb{T}^n(p)}H(x,z)(1-\varphi(z))\left[K\left(\frac{x}{\eps},\frac{x-z}{\eps}\right)-\overline{K}\left(\frac{x}{\eps}\right)\right]\,dz\right|\eqqcolon H_{4,1}+H_{4,2}.
    \end{align*}
    Next, in light of \eqref{est.F:varphi}, we get
    \begin{align*}
        H_{4,1}=\left|\int_{\bbR^n}H(x,z)\varphi(z)\left[K\left(\frac{x}{\eps},\frac{x-z}{\eps}\right)-\overline{K}\left(\frac{x}{\eps}\right)\right]\,dz\right|.
    \end{align*}
    To further estimate $H_{4,1}$, we consider a periodic function $W_x(z)$ such that
    \begin{equation*}
        (-\Delta_z)W_x(z)=K\left(\frac{x}\eps,\frac{x}\eps-z\right)-\overline{K}\left(\frac{x}\eps\right)
    \end{equation*}
    with
    \begin{equation}\label{est.F:Wxest}
        \|W_x\|_{W^{1,\infty}(\bbR^n)}\leq c.
    \end{equation}
    Then we have $\eps^2((-\Delta_z)W_{x}(\cdot/\eps))(z)=K\left(\frac{x}{\eps},\frac{x-z}{\eps}\right)-\overline{K}\left(\frac{x}{\eps}\right)$. Thus, we have
    \begin{align*}
        H_{4,1}\leq \eps^2\left|\int_{\bbR^n}H(x,z)\varphi(z)((-\Delta_z)W_{x}(\cdot/\eps))(z)\,dz\right|\leq \begin{cases}
            c\eps^2\|D_z^2(H\varphi)(x,z)\|_{L^1(\bbR^n)}&\quad\text{if }s\geq \frac12,\\
            c\eps\|D_z(H\varphi)(z,x)\|_{L^1(\bbR^n)}&\quad\text{if }s<\frac12,
        \end{cases}
    \end{align*}
    where we have used integration by parts together with \eqref{est.F:varphi} and \eqref{est.F:Wxest}. Note from \eqref{est.F:varphi}, \eqref{est.F:ass1} and \eqref{dist.aux2:main1} that for any $j\in\{1,2\}$,
    \begin{align*}
        |D^j\varphi(z)|\leq c\eps^{-j}\mbox{\large$\chi$}_{\{(3\eps)/2<\|z\|_\infty<(5\eps)/2\}}+cd_\eps^{-j}(x)\mbox{\large$\chi$}_{\{(k+1)\eps/2<\|z\|_\infty<k\eps\}}
    \end{align*}
    and 
    \begin{align*}
    |H(x,z)|&\leq 
        c\left(|z|^{-n-2s+(1+\min\{2s+\frac\alpha2,1\})}d_\eps(x)^{s-(1+\min\{2s+\frac\alpha2,1\})}+d_\eps^{s-1}(x)\right)
     \\
        |D_zH(x,z)|&\leq c\left(|z|^{-n-2s+\min\{2s+\frac\alpha2,1\}}d^{s-(1+\min\{2s+\frac\alpha2,1\})}_\eps(x)++d_\eps^{s-1}(x)\right),\\
        |D^2_zH(x,z)|&\leq c\left(|z|^{-n-2s}d^{s-2}_\eps(x)+d_\eps^{s-1}(x)\right)\quad\text{if }s\geq\frac12.
    \end{align*}
    Now using the above two inequalities together with \eqref{est.F:varphi}, we derive
    \begin{align*}
        H_{4,1}\leq c\eps^{2-2s}d_\eps^{-2+s}(x)\mbox{\large$\chi$}_{\{s\geq\frac12\}}+c\eps d_\eps^{-1-s}(x)\mbox{\large$\chi$}_{\{s<\frac12\}}.
    \end{align*}
\end{itemize}
    Next, using \eqref{est.F:varphi}, we have
    \begin{equation*}
    \begin{aligned}
        H_{4,2}&\leq c\int_{\|z\|_\infty\leq 3\eps}|H(x,z)|\,dz\\
        &\quad+\left|\sum_{p\in\eps\Z^n\cap A_3,\frac{d_\eps(x)}{c}<\|p\|_\infty<\frac{d_\eps(x)}4}\int_{\eps\mathbb{T}^n(p)}H(x,z)(1-\varphi(z))\left[K\left(\frac{x}{\eps},\frac{x-z}{\eps}\right)-\overline{K}\left(\frac{x}{\eps}\right)\right]\,dz\right|,
    \end{aligned}
    \end{equation*}
    where $\frac{d_\eps(x)}{c}>4\eps\sqrt{n}$. 
    In addition, as in the estimate of $F_{1,1}(x)$ together with \eqref{est.F:F12.rep}, we get
    \begin{align*}
        H_{4,2}\leq c\nu(x)+c\sum_{p\in\eps\Z^n,\frac{d_\eps(x)}{c}<\|p\|_\infty<\frac{d_\eps(x)}4}\int_{\eps\mathbb{T}^n(p)}|H(x,z)(1-\varphi(z))-H(x,p)(1-\varphi(p))|\,dz.
    \end{align*}
    Using \eqref{est.F:varphi}, \eqref{est.F:pandxi}, \eqref{est.F:ass1} and \autoref{lem.dist.aux2}, we obtain for any $z\in \eps\mathbb{T}^n(p)$ with $p\in\eps\Z^n$ and $\frac{d_\eps(x)}{c}<\|p\|_\infty<\frac{d_\eps(x)}2$,  $\frac{d_\eps(x)}{2c}<|z|<\frac{d_\eps(x)}2$ and 
    \begin{equation}\label{est.F:diffH(x,z)}
    \begin{aligned}
     &|H(x,z)(1-\varphi(z))-H(x,p)(1-\varphi(p))|\\
     &\leq |H(x,z)-H(x,p)|(1-\varphi(z))|+|H(x,p)||\varphi(z)-\varphi(p)|\\
     &\leq c\left(\eps d_\eps^{s-1}(x)|z|^{-(n+2s)}+\eps (|z|^s+|z|d_\eps^{s-1}(x))|z|^{-n-2s-1}+\eps d_\eps^{-1}(x)(|z|^s+|z|d_\eps^{s-1}(x))|z|^{-(n+2s)}\right).
    \end{aligned}
    \end{equation}
    Therefore, we have 
    \begin{align*}
        H_{4,2}\leq c\nu(x).
    \end{align*}
    For the estimate of $H_5$, using  \eqref{est.F:diffH(x,z)} with $\varphi=0$, \eqref{est.F:ass1} and the last claim in \autoref{lem.dist.app}, we have 
    \begin{align*}
        H_5&\leq c\left|\sum_{p\in\eps\Z^n\cap A_3,\|p\|_\infty\geq\frac{d_\eps(x)}4}\int_{\eps\mathbb{T}^n(p)}|H(x,z)-H(x,p)|\,dz\right|\\
        &\leq c \int_{|z|>\frac{d_\eps(x)}4}\frac{\eps d_\eps^{s-1}(x-z)+\eps d_\eps^{s-1}(x)\mbox{\large$\chi$}_{\{x-z\in \Omega_{\eps}\}}}{|z|^{n+2s}}+\eps\frac{|z|^s+|z|d_\eps^{s-1}(x)}{|z|^{n+2s+1}}\,dz.
    \end{align*}
    Now using \eqref{dist.tailint:main3}, we further estimate $H_5$ as 
    \begin{align*}
        H_5\leq c\eps d_\eps^{-1-s}(x).
    \end{align*}
    Combining the estimates $H_4$ and $H_5$, we have
    \begin{equation*}
        |F_{1,2,3}|\leq c\nu(x),
    \end{equation*}
    where $\nu(x)$ is determined in \eqref{est.F:defn.nu}.
As a result, using the estimate $F_{1,2}(x)$ and \eqref{conv:est.F11}, we have 
\begin{equation*}
    F_1(x)\leq c\nu(x).
\end{equation*}

\noindent
\textbf{Estimate of $F_2$.} First, we write
\begin{align*}
    F_2(x)=\eps\int_{\|z\|_\infty\leq\frac32\eps}(\cdot)\,dz+\eps\int_{\|z\|_\infty>\frac32\eps}(\cdot)\,dz\eqqcolon F_{2,1}(x)+F_{2,2}(x).
\end{align*}
Using \eqref{est.F:ass2} and \eqref{est.F:ass1}, 
we have
\begin{align*}
    F_{2,1}(x)&\leq c\eps\int_{|z|<\frac{3}2\eps\sqrt{n}}[\phi_\eps]_{C^{s}}\frac{d_\eps^{s-(1+s_1)}(x)+d^{s-(1+s_1)}_\eps(x-z)}{|z|^{n+s-s_1}}\,dz\leq c\nu(x),
\end{align*}
where we have used the fact that $d_\eps(x-z)\geq \frac{d_\eps(x)}2$ by \eqref{dist.aux2:main1} together with the fact that $|z|\leq 2\eps\sqrt{n}\leq \frac{d_\eps(x)}{2} $.

We estimate the term $F_{2,2}(x)$ by dividing the proof into two cases $s\neq\frac12$ and $s=\frac12$.
\begin{enumerate}
\item Let $s\neq\frac12$. We split it into the three cases as in the estimate of $F_{1,2}$:
\begin{align*}
    F_{2,2}(x)=\eps\sum_{i=1}^3\int_{z\in \widetilde{A}_i,\|z\|_\infty\geq\frac{3}2\eps}(\cdot)\,dz\eqqcolon \sum_{i=1}^3 F_{2,2,i}(x),
\end{align*}
where the sets $\widetilde{A}_i$ are defined as in the set $A_i$ given by \eqref{conv:defn.A_i} with $p$ replaced by $z$. 
\begin{itemize}
    \item Estimate of $F_{2,2,1}(x)$. As in the first observation in \eqref{conv:obs.A_i}, we have $z\in \widetilde{A}_1$, then $|z|>\frac{d_\eps(x)}{4}$. Using this, \eqref{est.F:ass1} and the boundedness of $\phi_i$ by \eqref{est.F:ass2}, we have
    \begin{align*}
       F_{2,2,1}(x)&\leq c\eps\left(\int_{{|z|>\frac{d_\eps(x)}4}}\frac{d_\eps(x)^{s-1}}{|z|^{n+2s}}\,dz+\int_{|z|>\frac{d_\eps(x)}4}\frac{\|\nabla(\overline{u}_\eps\zeta)\|_{L^\infty(B_2(\Omega)\setminus\Omega^\eps)}\,dz}{|z|^{n+2s}}\right)\leq c\eps d_\eps(x)^{-1-s}.
    \end{align*}
    \item Estimate of $F_{2,2,2}(x)$. Similarly, as in the second and third observations given by \eqref{conv:obs.A_i}, we observe if $z\in A_2$ and $|z|\leq\frac{d_\eps(x)}{4}$, then $|z|\eqsim d_\eps(x-z)\eqsim d_\eps(x)\eqsim \eps$. Thus, together with \eqref{est.F:ass1} and \eqref{dist.tailint:main4}, the above observation yields
    \begin{align*}
         F_{2,2,2}(x)&\leq c\eps\int_{\{\frac{d_\eps(x)}{c}<|z|<\frac{d_\eps(x)}4\}\cap A_2}\frac{d_\eps^{s-1}(x)+d_\eps^{s-1}(x-z)\mbox{\large$\chi$}_{\{x-z\in \Omega^{\frac32\eps}\setminus \Omega_{\frac32\eps}\}}+1}{|z|^{n+2s}} \,dz\\
         &\quad +c\eps\int_{\{|z|>\frac{d_\eps(x)}4\}\cap A_2}\frac{d_\eps^{s-1}(x)+d_\eps^{s-1}(x-z)\mbox{\large$\chi$}_{\{x-z\in \Omega^{\frac32\eps}\setminus \Omega_{\frac32\eps}\}}+1}{|z|^{n+2s}} \,dz\\
         &\leq c\eps d_\eps^{-1-s}(x).
    \end{align*}
    In particular, we have used $|\nabla(\overline{u}_\eps\zeta)(x-z)|\leq c$ when $x-z\in \Omega^{\frac32\eps}\setminus \Omega^\eps\subset B_2(\Omega)\setminus \Omega^\eps$ by \eqref{est.F:ass1}.
    \item Estimate of $F_{2,2,3}(x)$. Using the fact that $x-z\in \Omega_\eps$, \eqref{est.F:ass1}, we have 
    \begin{align*}
        F_{2,2,3}(x)&\leq c\eps\left(\int_{\eps<|z|<\frac{d_\eps(x)}2}\frac{|\nabla\overline{u}_\eps(x)-\nabla\overline{u}_\eps(x-z)|}{|z|^{n+2s}}\,dz+\int_{|z|\geq\frac{d_\eps(x)}2}\frac{d_\eps^{s-1}(x)+d_\eps^{s-1}(x-z)\mbox{\large$\chi$}_{\{x-z\in\Omega_\eps\}}}{|z|^{n+2s}}\,dz\right)\\
        &\leq c\eps\left(\int_{\eps<|z|<\frac{d_\eps(x)}2}\frac{d_\eps^{s-(1+\min\{2s+\frac\alpha2,1\})}(x)+d_\eps^{s-(1+\min\{2s+\frac\alpha2,1\})}(x-z)}{|z|^{n+2s-\min\{2s+\frac\alpha2,1\}}}\,dz+d_\eps^{-1-s}(x)\right)\\
        &\leq c\nu(x)
    \end{align*}
    where we have used \eqref{dist.tailint:main3}, \eqref{dist.aux2:main1} and $\nu(x)$ is determined in \eqref{est.F:defn.nu}.
    
\end{itemize}
    \item The case $s=\frac12$. We write
    \begin{equation*}
        \mathcal{H}(x,x-z)=(\nabla\overline{u}_\eps(x)-\nabla(\overline{u}_\eps\zeta)(x-z))\quad\text{and}\quad \mathcal{K}(x,x-z)\coloneqq (\phi_\eps(x)-\phi_\eps(x-z))K\left(\frac{x}\eps,\frac{x-z}\eps\right)
    \end{equation*}
    to see that $\mathcal{K}$ is $\eps\Z^n$-periodic and 
    \begin{align*}
        F_{2,2}(x)&=\eps\int_{\frac{3}2\eps<|z|<\frac{d_\eps(x)}4}\mathcal{H}(x,x-z) \mathcal{K}(x,x-z)\,dz+\eps\int_{\frac{d_\eps(x)}4\leq|z|}\mathcal{H}(x,x-z) \mathcal{K}(x,x-z)\,dz+\widetilde{F}_{2,2}(x)\\
        &\eqqcolon F_{2,2,1}(x)+F_{2,2,2}(x)+\widetilde{F}_{2,2}(x),
    \end{align*}
    where $|\widetilde{F}_{2,2}(x)|\leq c\int_{|z|<\frac{3\sqrt{n}}2\eps}|\mathcal{H}(x,x-z)\mathcal{K}(x,x-z)|\leq c\nu(x)$ by the estimate of $F_{2,1}(x)$.
    Next, by \eqref{est.F:ass1} and \eqref{est.F:ass2}, we estimate $F_{2,2,2}(x)$ as 
    \begin{align*}
        |F_{2,2,2}(x)|&\leq c\eps\int_{\frac{d_\eps(x)}4\leq|z|}\frac{d_\eps(x)^{s-1}+d_\eps(x-z)^{s-1}\mbox{\large$\chi$}_{\{x-z\in\Omega^\eps\}}+\|\nabla \overline{u}\|_{L^\infty(B_2(\Omega)\setminus \Omega^\eps)}}{|z|^{n+2s}}\,dz\\
        &\leq c\eps d_\eps^{-s-1}(x)+c\eps\int_{\frac{d_\eps(x)}{4}<|z|}\frac{d_\eps(x-z)^{s-1}\mbox{\large$\chi$}_{\{x-z\in\Omega_\eps\}}+d_\eps(x-z)^{s-1}\mbox{\large$\chi$}_{\{x-z\in\Omega^\eps\setminus \Omega_\eps\}}}{|z|^{n+2s}}\,dz\\
        &\leq c\eps d_\eps^{-s-1}(x),
    \end{align*}
    where we have used \eqref{dist.tailint:main3} and \eqref{dist.tailint:main4} for the last inequality.

    To estimate $F_{2,2,1}(x)$, we first write $\mathcal{K}(x)=\int_{\eps\mathbb{T}^n}\mathcal{K}(x,x-z)\,dz$ and see that
    \begin{align*}
        F_{2,2,1}(x)&=\eps\int_{\frac{3}2\eps<|z|<\frac{d_\eps(x)}4}\mathcal{H}(x,x-z) \mathcal{K}(x)\,dz+\eps\int_{\frac{3}2\eps<|z|<\frac{d_\eps(x)}4}\mathcal{H}(x,x-z) (\mathcal{K}(x,x-z)-\mathcal{K}(x))\,dz\\
        &\eqqcolon \mathcal{F}_{1}(x)+\mathcal{F}_2(x).
    \end{align*}
    First, using the fact that $x-z\in \Omega^\eps$ when $|z|<\frac{d_\eps(x)}{2}$, \eqref{dist.aux2:main1}, \eqref{est.F:ass1} and \eqref{est.F:ass2}, we get
    \begin{align*}
        \mathcal{F}_1(x)&=\eps\int_{\frac{3}2\eps<|z|<\frac{d_\eps(x)}4}\left(\mathcal{H}(x,x-z)-\frac{z\nabla^2\overline{u}_\eps(x)}{|z|^{n+2s}}\right) \mathcal{K}(x)\,dz\leq c\eps\int_{|z|<\frac{d_\eps(x)}4}\frac{d^{-1-s-\alpha}_\eps(x)}{|z|^{n-\alpha}}\,dz
        \leq c\eps d_\eps^{-1-s}(x).
    \end{align*}
     For $\mathcal{F}_2(x)$, we are going to use the same argument as in the estimate of $H_4$. To this end, we write
    \begin{align*}
        \mathcal{F}_2(x)&=\eps\int_{\frac{3}2\eps<|z|<\frac{d_\eps(x)}2}\widetilde{\varphi}(z)\mathcal{H}(x,x-z) (\mathcal{K}(x,x-z)-\mathcal{K}(x))\,dz\\
        &\quad+\eps\int_{\frac{3}2\eps<|z|<\frac{d_\eps(x)}2}(1-\widetilde{\varphi}(z))\mathcal{H}(x,x-z) (\mathcal{K}(x,x-z)-\mathcal{K}(x))\,dz\eqqcolon \mathcal{F}_{2,1}(x)+\mathcal{F}_{2,2}(x),
    \end{align*}
    where $\widetilde{\varphi}$ is determined as in \eqref{est.F:varphi} with $\|\cdot\|_\infty$ replaced by $|\cdot|$. Let $\mathcal{W}$ be a solution to $(-\Delta)\mathcal{W}(z)=\mathcal{K}(x,x-z)-\mathcal{K}(x)$. As $\mathcal{K}$ is $\eps\Z^n$-periodic, we have 
    \begin{align*}
        \|\mathcal{W}\|_{W^{1,\infty}(\bbR^n)}\leq c\eps.
    \end{align*}
    In addition, for any $|z|<\frac{d_\eps(x)}{2}$, using \eqref{est.F:ass1} and \eqref{dist.aux2:main1}, we have 
    \begin{align*}
        |\mathcal{H}(x,x-z)|\leq c|z|^{-n}d_\eps^{s-2}(x),\quad |D_z\mathcal{H}(x,x-z)|\leq cd_\eps^{s-2}(x)|z|^{-n-1}.
    \end{align*}
    Therefore, as in the estimate of $H_{4,1}$ together with the fact that $s=\frac12$, we have 
    \begin{align*}
        |F_{2,1}(x)|\leq c\eps d_\eps^{-1-s}(x).
    \end{align*}
    Next, by \eqref{est.F:varphi}, we have 
    \begin{align*}
        |\mathcal{F}_{2,2}(x)|\leq c\eps\left(\int_{\frac32\eps<|z|<c\eps}\frac{d^{s-2}_\eps(x)}{|z|^{n}}\,dz+\int_{\frac{d_\eps(x)}c<|z|<\frac{d_\eps(x)}2}\frac{d^{s-2}_\eps(x)}{|z|^{n}}\,dz\right)\leq c\eps d_\eps^{-1-s}(x).
    \end{align*}
\end{enumerate}
Altogether, we have proved
\begin{align*}
    |F_2(x)|\leq  c\nu(x).
\end{align*}
It remains to estimate $F_3$. We follow the same argument as in the estimate of $F_2$ with $\phi_\eps$ and $\nabla\overline{u}_\eps$ replaced by $\psi_\eps$ and $h$, respectively.

\textbf{Estimate of $F_3$.}  
First, we write
\begin{align*}
    F_3(x)=\int_{\|z\|_\infty\leq\frac32\eps}(\cdot)\,dz+\int_{\|z\|_\infty>\frac32\eps}(\cdot)\,dz\eqqcolon F_{3,1}(x)+F_{3,2}(x).
\end{align*}
As in the estimate of $F_{2,1}$ together with \eqref{est.F:ass2}, we derive
\begin{align*}
    F_{3,1}(x)\leq  c\eps^{2s}\int_{\frac32\eps\mathbb{T}^n}[\psi_\eps]_{C^{s}}\frac{1}{|z|^{n+2s-(1+s)}}\,dz\leq c\eps.
\end{align*}
For $F_{3,2}$, using the boundedness of $\psi$ by \eqref{est.F:ass2}, we have
\begin{align*}
    F_{3,2}(x)\leq c\eps^{2s}\left(\int_{\eps<|z|<1}\frac{\|h\|_{C^1}\,dz}{|z|^{n+2s-1}}+\int_{|z|>1}\frac{\|h\|_{L^\infty}\,dz}{|z|^{n+2s}} \right)\leq c\left(\eps^{2-2s}\mbox{\large$\chi$}_{\{s>\frac12\}}+\eps^{2s} \mbox{\large$\chi$}_{\{s<\frac12\}}+\eps|\ln\eps|\mbox{\large$\chi$}_{\{s=\frac12\}}\right).
\end{align*}
Therefore, we have
\begin{align*}
    |F_3(x)|\leq c\left(\eps^{2-2s}\mbox{\large$\chi$}_{\{s>\frac12\}}+\eps^{2s} \mbox{\large$\chi$}_{\{s<\frac12\}}+\eps|\ln\eps|\mbox{\large$\chi$}_{\{s=\frac12\}}\right).
\end{align*}
Finally, combining all estimates of $F_1,F_2,F_3$ yields the desired estimate. This completes the proof.
\end{proof}

Next, we prove $L^2$ error estimates of solutions to Dirichlet problems with fractional divergence type data that was used to estimate $L^2$-norm of $z_\eps-z_{1,\eps}$ given in \eqref{conv:eq.z1}.

\begin{lemma}\label{lem.est.fracdiv}
Let $r_0,M,N,C_1>0$, $\gamma\in(\frac12,1)$, $\alpha\in(0,1]$ and let $K\in C^\alpha(\mathbb{T}^n)$ and $\Omega$ be a $C^{1,\alpha}$ domain with $C^{1,\alpha}$ radius $r_0$, $C^{1,\alpha}$ norm bounded by $M$, $\mathrm{diam}(\Omega)\leq N$ and $0\in\Omega$. Let $\Omega^\eps$ be the $\eps$-tubular neighborhood given in \eqref{defn.omegaeps} for $\eps\leq \eps_0$, where $\eps_0$ is determined in \eqref{conv:eps0.Omega}. Let $\overline{u}_\eps $ satisfy \eqref{est.F:ass1} and let $\phi_i, \psi$ be periodic and $h\in C_c^{1+\alpha}(B_2(\Omega))$ with \eqref{est.F:ass2}. Recall the function $(-\Delta)^s\widetilde{F}\in H^{-s}(\Omega)$ given in \eqref{conv:rep.tildeF}. 
Then for the weak solution $v$ of
    \begin{equation*}
\left\{
\begin{alignedat}{3}
L_\eps v&=(-\Delta)^s\widetilde{F}&&\qquad \mbox{in  $\Omega$}, \\
v&=0&&\qquad  \mbox{in $\bbR^n\setminus \Omega$},
\end{alignedat} \right.
\end{equation*}
there is a constant $c=c(n,s,\Lambda,[K]_{C^\alpha(\mathbb{T}^n)},\alpha,r_0,M,N,C_1,\gamma)$ such that
\begin{align}\label{fracdiv.mainres3}
        \|v\|_{L^2(\Omega)}\leq c\eps^{s_0},
    \end{align}
    where the constant $s_0$ is determined in \eqref{conv:defn.s_0}.
\end{lemma}
\begin{proof}
In the proof, since the constant $c$ always depends only on $n,s,\Lambda,[K]_{C^\delta},\alpha,r_0,M,N,C_1,\gamma$, we omit the dependence of the generic constant $c$ for the convenience of notation.

Our strategy is to treat the function $(-\Delta)^s\widetilde{F}(x)$ in a pointwise way and apply the Green formula to represent the solution $v$. Recall that
\begin{align*}
    &\langle(-\Delta)^s\widetilde{F},\varphi\rangle\\
    &=-\frac\eps2\iint_{\bbR^n\times\bbR^n}\frac{(\nabla(\overline{u}_\eps\zeta)(x)-\nabla(\overline{u}_\eps\zeta)(x-z))\cdot (\phi_\eps(x)\varphi(x)-\phi_\eps(x-z)\varphi(x-z))}{|z|^{n+2s}}K\left(\frac{x}{\eps},\frac{x-z}{\eps}\right)\,dz\,dx\nonumber\\
    &\quad-\frac{\eps^{2s}}2\iint_{\bbR^n\times\bbR^n}\frac{(h(x)-h(x-z)) (\psi_\eps(x)\varphi(x)-\psi_\eps(x-z)\varphi(x-z))}{|z|^{n+2s}}K\left(\frac{x}{\eps},\frac{x-z}{\eps}\right)\,dz\,dx \\
    &=: \langle \widetilde{f}_{1},\varphi\rangle + \langle \widetilde{f}_{2},\varphi\rangle
\end{align*}

However, when $s\geq\frac12$, $(-\Delta)^s\widetilde{F}(x)$ is not well defined. To avoid this issue, we mollify the associated kernel to ensure Lipschitz regularity in both variables, i.e. to use
\begin{align*}
    |K(x,y)-K(x+\xi,y+\zeta)|\lesssim (|\xi|+|\zeta|).
\end{align*}

First, we mollify the kernel $K$. As $\Lambda^{-1}\leq K\leq \Lambda$, there is a sequence of smooth, symmetric, and periodic functions $(K_i)_i$  such that $\Lambda^{-1}\leq K_i\leq \Lambda$ and 
\begin{align}\label{conv:K_i.moll}
    \lim_{i\to\infty}\|K_i-K\|_{L^p(\mathbb{T}^n\times \mathbb{T}^n)}=0
\end{align}
for any $p\in(1,\infty)$. Note that such sequences are obtained via a standard mollification argument using convolution.
Define $f_i(x)$ by
\begin{align*}
    f_i(x)\coloneqq& {-\eps\phi_\eps(x)\cdot\int_{|z|<\eps}\frac{(\nabla \overline{u}_\eps(x)-\nabla(\overline{u}_\eps\zeta)(x-z))}{|z|^{n+2s}}K_i\left(\frac{x}{\eps},\frac{x-z}{\eps}\right)\,dz}\\
    &-\eps\phi_\eps(x)\cdot\int_{|z|\geq\eps}\frac{(\nabla \overline{u}_\eps(x)-\nabla (\overline{u}_\eps\zeta)(x-z))}{|z|^{n+2s}}K\left(\frac{x}\eps,\frac{x-z}\eps\right)\,dz\eqqcolon f_{i,1}(x)+f_{i,2}(x).
\end{align*}
We observe from the estimates of $F_{2,2}(x)$ and $F_{3,2}(x)$ given in the proof of \autoref{lem.est.F} together with the boundedness of $\phi_\eps$ by \eqref{est.F:ass2} that
\begin{align}\label{est.fracdiv:est.I_2}
  {{ |f_{i,2}(x)|}}\leq c\nu(x),
\end{align}
where $\nu(x)$ is given in \eqref{est.F:defn.nu}. 
Next, to see that $f_{i,1}$ is also bounded in $\Omega$, we rewrite $f_{i,1}(x)$ as
\begin{align*}
    f_{i,1}(x)&=-\eps\phi_\eps(x)\cdot\int_{|z|<\eps}\frac{\nabla\overline{u}_\eps(x)-\nabla(\overline{u}_\eps\zeta)(x-z)}{|z|^{n+2s}}\left(K_i\left({x}/{\eps},{(x-z)}/{\eps}\right)-K_i\left({x}/{\eps},{x}/{\eps}\right)\right)\,dz\\
    &\quad-\eps\phi_\eps(x)\cdot\int_{|z|<\eps}\frac{\nabla\overline{u}_\eps(x)-\nabla(\overline{u}_\eps\zeta)(x-z)-z\nabla^2 \overline{u}_\eps(x)}{|z|^{n+2s}}K_i\left({x}/{\eps},{x}/{\eps}\right)\,dz
\end{align*}
by the fact that $\int_{|z|<\eps}z/|z|^{n+2s}\,dz=0$. Note that for any $x\in \Omega$ and $|z|<\eps$, $x-z\in \Omega^{\frac\eps2}$. Thus, using this and the fact that $\nabla\overline{u}_\eps\in C^{2s+\alpha/2}_{\mathrm{loc}}(\Omega^\eps)$ and that the kernels $K_i$ are smooth, we have $f_{i,1}\in L^\infty(\Omega)$. Similarly, we mollify $h$ to see that there are smooth functions $h_i$ such that
\begin{equation*}
    \lim_{i\to\infty}\|h_i-h\|_{C^{1}(\bbR^n)}=0.
\end{equation*}
Then we get that for 
\begin{align*}
    F_i(x)&\coloneqq -\eps^{2s}\psi_\eps(x)\int_{|z|<\eps}\frac{(h_i(x)-h_i(x-z))}{|z|^{n+2s}}K_i\left(\frac{x}{\eps},\frac{x-z}{\eps}\right)\,dz\\
    &\quad -{\eps^{2s}}\psi_\eps(x)\int_{|z|\geq\eps}\frac{(h(x)-h(x-z))}{|z|^{n+2s}}K\left(\frac{x}\eps,\frac{x-z}\eps\right)\,dz\eqqcolon F_{i,1}+F_{i,2}(x),
\end{align*}
the function $F_i$ is bounded in $\Omega$ with $|F_{i,2}(x)|\leq c\eps^{\min\{2s,1\}}(1+|\ln \eps|)$.
Therefore, there is a sequence of weak solutions $(v_{i})_i$ to
\begin{equation*}
\left\{
\begin{alignedat}{3}
L_\eps v_{i}&=f_i+F_i&&\qquad \mbox{in  $\Omega$}, \\
v_i&=0&&\qquad  \mbox{in $\bbR^n\setminus \Omega$}.
\end{alignedat} \right.
\end{equation*}
Moreover, using the symmetry of the kernel, we observe for any $\varphi\in L^1(\bbR^n)$ with $\varphi\equiv 0$ on $\bbR^n\setminus \Omega$ 
\begin{equation}\label{conv:fi.rep}
\begin{aligned}
    \int_\Omega f_{i,1}\varphi&=-\frac\eps2\iint_{\bbR^n\times\{|z|<\eps\}}\frac{(\nabla(\overline{u}_\eps \zeta)(x)-\nabla(\overline{u}_\eps\zeta)(x-z))\cdot ((\phi_\eps\varphi)(x)-(\phi_\eps\varphi)(x-z))}{|z|^{n+2s}}{K}_i\left(\frac{x}{\eps},\frac{x-z}{\eps}\right) \,d z \, d x.
\end{aligned}
\end{equation}

Hence, we obtain  by the Cauchy--Schwarz inequality and the symmetry of the kernels 
\begin{align*}
\big|\langle\widetilde{f}_1-f_{i},\varphi\rangle \big|
&=\frac\eps2 \Bigg| \iint_{\bbR^n \times \{|z|<\eps \} }
\frac{\bigl(\nabla(\overline u_\eps\zeta)(x)-\nabla(\overline u_\eps\zeta)(x-z)\bigr)}{|z|^{n+2s}} \cdot\bigl((\phi_\eps\varphi)(x)-(\phi_\eps\varphi)(x-z)\bigr) \\
&\qquad\qquad\qquad \times\left(K_i\left(\frac{x}{\eps},\frac{x-z}{\eps}\right)
-K\left(\frac{x}{\eps},\frac{x-z}{\eps}\right)\right)\,dz\,dx \Bigg| \\
&\leq c(\eps) \Bigg( \iint_{ \Omega \times \{|z|<\eps\} }
\frac{|\nabla(\overline u_\eps\zeta)(x)-\nabla(\overline u_\eps\zeta)(x-z)|^2}{|z|^{n+2s}} \\
&\qquad\qquad\qquad \times \left|K_i\left(\frac{x}{\eps},\frac{x-z}{\eps}\right) -K\left(\frac{x}{\eps},\frac{x-z}{\eps}\right)\right|^2\,dz\,dx
\Bigg)^{1/2}\Vert\varphi\Vert_{H^s(\bbR^n)}\\
&\leq c(\eps,p)\Vert K_i-K \Vert_{L^p(\mathbb T^n\times\mathbb T^n)} \Vert \varphi \Vert_{H^s(\bbR^n)}
\end{align*}
for sufficiently large $p>2$, where the last inequality follows from the fact that $\nabla\overline{u}_\eps\in C^{2s+\alpha/2}_{\mathrm{loc}}(\Omega^\eps)$, H\"older's inequality, the periodicity of $K_i-K$, and the fact that $\phi_\eps \in C^{\sigma}(\mathbb T^n)$ for some $\sigma > s$; see \autoref{lem.bdd.periodic}.

Similarly, we have 
\begin{equation*}
\begin{aligned}
    \int_\Omega F_{i,1}\varphi&=-\frac{\eps^{2s}}{2}\iint_{\bbR^n\times\{|z|<\eps\}}\frac{(h_i(x)-h_i(x-z))\cdot ((\psi_\eps\varphi)(x)-(\psi_\eps\varphi)(x-z))}{|z|^{n+2s}}{K}_i\left(\frac{x}{\eps},\frac{x-z}{\eps}\right) \,d z \, d x
\end{aligned}
\end{equation*}
and estimate
\begin{align*}
|\langle\widetilde f_{2}-F_{i},\varphi\rangle|
\leq c(\eps,p)\left(
\Vert K_i-K \Vert_{L^p(\mathbb T^n\times\mathbb T^n)}
+\Vert h_i-h\Vert_{C^1(\bbR^n)}
\right)\Vert\varphi\Vert_{H^s(\bbR^n)}.
\end{align*}

Altogether, and using also that by construction
\begin{align*}
(-\Delta)^s\widetilde F-(f_i+F_i)=(\widetilde f_1-f_{i})+(\widetilde f_2-F_{i}),
\end{align*}
in the weak sense, we deduce for every $\varphi\in H^s(\bbR^n)$ vanishing outside $\Omega$ that
\begin{align*}
\left|\left\langle(-\Delta)^s\widetilde F-(f_i+F_i),\varphi\right\rangle\right| \leq c(\eps,p)\left(
\Vert K_i-K\Vert_{L^p(\mathbb T^n\times\mathbb T^n)}
+\Vert h_i-h\Vert_{C^1(\bbR^n)}
\right)\Vert\varphi\Vert_{H^s(\bbR^n)}.
\end{align*}

Therefore, testing the equation $L_\eps(v-v_{i})=(-\Delta)^s\widetilde{F}-(f_i+F_i)$ in $\Omega$ with $v-v_{i}$ and using ellipticity and the fractional Poincar\'e inequality, we deduce for sufficiently large $p>2$,
\begin{align*}
    \|v-v_i\|_{L^2(\Omega)} \leq c [ (-\Delta)^s\widetilde F-(f_i+F_i)]_{H^{-s}(\Omega)} \leq c(\eps,p)(\|K_i-K\|_{L^p(\mathbb{T}^n\times \mathbb{T}^n)}+\|h_i-h\|_{C^1(\bbR^n)}),
\end{align*}
where $c$ is independent of $i$. Moreover, we have also used the fact that if $x\notin \Omega^{\frac\eps2}$ and $|z|<\eps$, then $x-z\notin \Omega$ by \autoref{lem.dist.app}.
Thus, taking $i$ sufficiently large, by \eqref{conv:K_i.moll}, we have
\begin{align}\label{est.z2i.seq}
    \|v-v_i\|_{L^2(\Omega)}\leq \eps.
\end{align}
Next, using the Green representation of $v_{i}$, we are going prove the following $L^2$ estimate:
\begin{align*}
    \|v_{i}\|_{L^2(\Omega)}\leq c\eps^{s_0}.
\end{align*}
To this end, we estimate $J_i(x)$ given by
\begin{align*}
    v_{i}(x)&=\int_{\Omega}G_\eps(x,y)f_{i,1}(y)\,dy+\int_{\Omega}G_\eps(x,y)f_{i,2}(y)\,dy+\int_{\Omega}G_\eps(x,y)F_{i,1}(y)\,dy+\int_{\Omega}G_\eps(x,y)F_{i,2}(y)\,dy\\
    &\eqqcolon \sum_{i=1}^4J_i(x).
\end{align*}
\begin{itemize}[leftmargin=*]
    \item Estimate of $J_1(x)$. Plugging \eqref{conv:fi.rep} with $\varphi=G_\eps$ and rewriting this, we get
\begin{align*}
J_1(x)&=-\frac\eps2\int_{\Omega}G_\eps(x,y)\int_{|z|<\eps}\frac{(\nabla\overline{u}_\eps(y)-\nabla(\overline{u}_\eps\zeta)(y-z))\cdot(\phi_\eps(y)-\phi_\eps(y-z))}{|z|^{n+2s}}K_i\left(\frac{y}{\eps},\frac{y-z}{\eps}\right)\,dz\,dy\\
    &-\frac\eps2\int_{\bbR^n}\int_{|z|<\eps}\frac{(\nabla\overline{u}_\eps(y)-\nabla(\overline{u}_\eps\zeta)(y-z))(G_\eps(x,y)-G_\eps(x,y-z))}{|z|^{n+2s}}K_i\left(\frac{y}{\eps},\frac{y-z}{\eps}\right)\cdot\phi_\eps(y-z)\,dz\,dy\\
    &\eqqcolon J_{1,1}(x)+J_{2,1}(x).
\end{align*}
Following the same lines as in the estimate of $F_{2,1}(x)$ given in the proof of \autoref{lem.est.F} together with \eqref{est.F:ass1} and \eqref{est.F:ass2}, we have 
\begin{align}\label{conv:est.J_{1,1}}
J_{1,1}(x)\leq c\int_\Omega G_{\eps}(x,y)\nu(y)\,dy,
\end{align}
where $\nu(y)$ is determined in \eqref{est.F:defn.nu}.
Next, by the boundedness of $\phi$, \eqref{est.F:ass1} and \eqref{est.F:ass2}, we have
\begin{align*}
    J_{1,2}(x)\leq c\eps\int_{0}^1\int_{\Omega^{\frac\eps2}}\int_{|z|<\eps}\frac{(d_\eps^{s-(1+\min\{\frac{3s}2,1\})}(y)+d_\eps^{s-(1+\min\{\frac{3s}2,1\})}(y-tz))|(G_\eps(x,y)-G_{\eps}(x,y-z)|}{|z|^{n+2s-\min\{\frac{3s}2,1\}}}\,dz\,dy,
\end{align*}
where we have also used the fact that $G_\eps(x,y)=G_\eps(x,y-z)\equiv0$ when $y\notin \bbR^n\setminus \Omega^{\frac\eps2}$ and $|z|<\eps$, as $y-z\in \bbR^n\setminus \Omega^{\frac\eps4}$ by \autoref{lem.dist.app}. Next we observe $d_\eps(y)\geq 4\sqrt{n}\eps$ for any $y\in \Omega^{\frac\eps2}$ by \autoref{lem.dist.app}, which implies for $|z|<\eps$, $t\in[0,1]$, $|z|<\frac{d_\eps(y)}{2}$ and  $d_\eps(y-tz)\geq \frac{d_\eps(y)}{2}$ by \eqref{dist.aux2:main1}. With the observation, we apply \autoref{lem.fracsob.green} to further estimate $J_{1,2}$ as
\begin{align*}
    J_{1,2}(x)&\leq c\eps\int_{\Omega^{\frac\eps2}}d_\eps^{s-(1+\min\{\frac{3s}2,1\})}(y)\int_{B_\eps(y)}\frac{|(G_\eps(y,x)-G_{\eps}(z,x)|}{|y-z|^{n+\max\{\frac{s}2,2s-1\}}}\,dz\,dy\\
    &\leq c\eps\int_{\Omega^{\frac\eps2}}d_\eps^{s-(1+\min\{\frac{3s}2,1\})}(y)\left(\frac{\eps^{\beta-\sigma_0}+d^{\beta-\sigma_0}_\Omega(x)}{|x-y|^{n-2s+\beta}}+\frac{d^\beta_\Omega(x)+\eps^\beta}{|x-y|^{n-2s+\sigma_0+\beta}}\right)\,dy
\end{align*}
for any $\beta\in(\max\{\frac{s}{2},2s-1\},s)$, where we write $\sigma_0\coloneqq \max\{\frac{s}{2},2s-1\}$. Using the fact that $d_\eps(y)\geq d_\Omega(y)/c$, $d_\eps(y)\geq \eps$ for $y\in \Omega^{\frac\eps2}$ and applying the first claim in \autoref{lem.singdist}, we further estimate $J_{1,2}$ as 
\begin{align*}
    J_{1,2}(x)\leq c\eps^{s_0}\int_{\Omega^{\frac\eps2}}\left(\frac{d_\eps^{-s_0+\beta-s}(y)} {|y-x|^{n-2s+\beta}}+\frac{d_\eps(y)^{-s_0-s+\beta+\sigma_0}}{|y-x|^{n-2s+\sigma_0+\beta}}\right)\,dy
    \leq c\eps^{s_0}(1+d_\Omega^{s-s_0}(x))
\end{align*}
by taking $\beta$ close to $s$. As in \eqref{conv:est.z_{1,eps}}, using \eqref{conv:est.J_{1,1}} and the estimate $J_{1,2}$, we have
\begin{align*}
    \|J_1(x)\|_{L^2(\Omega)}\leq c\eps^{s_0}.
\end{align*}

\item Estimate of $J_2(x)$. Using \eqref{est.fracdiv:est.I_2} and \eqref{conv:est.z_{1,eps}}, we have
\begin{align*}
    \|J_2(x)\|_{L^2(\Omega)}\leq c\eps^{s_0}.
\end{align*}
\item Estimate of $J_3(x)$. We follow the same lines as in the estimate of $J_1(x)$ with $\nabla\overline{u}_\eps$ replaced by $h$ and use the fact that $\|h_i\|_{C^{1+\alpha}(\bbR^n)}\leq c$, $h_i\in C_c^{1+\alpha}(B_2(\Omega))$, in order to see that
\begin{align*}
    \|J_3(x)\|_{L^2(\Omega)}\leq c\eps^{2s}.
\end{align*}

\item Estimate of $J_4(x)$. As in the estimate of $J_2(x)$ with $\nabla\overline{u}_\eps$ replaced by $h$, by the estimate of $F_{3,2}(x)$ in the proof of \autoref{lem.est.F}, we deduce 
\begin{align*}
    \|J_4(x)\|_{L^2(\Omega)}\leq c(\eps^{2s}_{\mbox{\large$\chi$}_{\{s<\frac12\}}}+\eps|\ln\eps|_{\mbox{\large$\chi$}_{\{s=\frac12\}}}+\eps^{2-2s}_{\mbox{\large$\chi$}_{\{s>\frac12\}}}).
\end{align*}

\end{itemize}
Combining all estimates $J_1,J_2,J_3,J_4$ yields
\begin{align*}
    \|v_{i}\|_{L^2(\Omega)}\leq c\eps^{s_0}.
\end{align*}
Together with \eqref{est.z2i.seq}, this proves \eqref{fracdiv.mainres3}. This completes the proof.
\end{proof}

\subsection{Convergence rates in $L^1$}

We end this section by providing a convergence rate for the Dirichlet problem when the boundary data is given by non-smooth data $g\in L^\infty(B_3(\Omega))\cap L^1_{2s}(\bbR^n)$.
\begin{lemma}\label{lem.conv.sub}
     Let $\alpha\in(0,1]$, $\eps\in(0,1]$, $r_0,M,N>0$, $K\in C^\alpha(\mathbb{T}^n)$ and $\partial\Omega\in C^{1,\alpha}$ with $C^{1,\alpha}$ radius $r_0$, $C^{1,\alpha}$ norm bounded by $M$, $\mathrm{diam}(\Omega)\leq N$ and $0\in\Omega$.
    Let $u_\eps$ and $\overline{u}$ be weak solutions to 
    \begin{equation*}
\left\{
\begin{aligned}
L_\eps u_\eps &= f
&\qquad& \text{in }\Omega, \\
u_\eps &= g
&\qquad& \text{in }\mathbb{R}^n\setminus\Omega,
\end{aligned}
\right.
\qquad
\left\{
\begin{aligned}
\overline{L} \overline{u} &= f
&\qquad& \text{in }\Omega, \\
\overline{u} &= g
&\qquad& \text{in }\mathbb{R}^n\setminus\Omega,
\end{aligned}
\right.
\end{equation*}
where $\overline{K}\coloneqq\int_{\mathbb{T}^n\times \mathbb{T}^n}K(x,x-z)\,dz\,dx$, $f\in L^2(\Omega)$ and $g\in C^{\gamma}(\bbR^n)\cap H^s(B_2(\Omega))$ for some $\gamma\in(0,1)$. Then there is a small constant $\sigma=\sigma(s,\gamma)$ such that
\begin{equation}\label{conv.sub:main.est}
\begin{aligned}
    \|u_\eps-\overline{u}\|_{L^1(\Omega)}\leq c\eps^\sigma(\|f\|_{L^\infty(\Omega)}+\|g\|_{C^{\gamma}(\bbR^n)})
\end{aligned}
\end{equation}
for some constant $c=c(n,s,\Lambda,[K]_{C^\alpha(\mathbb{T}^n)},\alpha,r_0,M,N,\gamma)$. 
\end{lemma}
\begin{proof}
We may assume $f=0$ in $\bbR^n\setminus \Omega$. Now, we mollify the data $f,g$. Let $\eta\in C_c^\infty(B_1)$ with $\int_{B_1}\eta=1$ and write $\eta_\eps(x)\coloneqq \eps^{-n}\eta(x/\eps)$. Let $\delta\coloneqq \frac{\min\{s,2-2s\}}4$. Define $f_\eps\coloneqq f*\eta_{\eps^\delta}$ and $g_\eps\coloneqq g*\eta_{\eps^{\delta}}$ to see that 
\begin{align}\label{conv.sub:mol0}
    \|f_\eps-f\|_{H^{-s}(\bbR^n)}\leq c\eps^{\delta s}\|f\|_{L^2(\Omega)},\quad \|f_\eps\|_{C^{1+\alpha}(\Omega)}\leq c\eps^{-\delta(1+\alpha)}\|f\|_{L^\infty(\Omega)}
\end{align}
and
\begin{align}\label{conv.sub:mol}
    &\|\nabla g_\eps\|_{L^1_{2s}(\bbR^n)}+\| g_\eps\|_{C^1(B_{2}(\Omega))}\leq c\eps^{-\delta}\|g\|_{L^\infty(\bbR^n)},\quad \|g_\eps\|_{C^\gamma(\bbR^n)}\leq c\|g\|_{C^\gamma(\bbR^n)}
\end{align}
for some constant $c=c(n,s,\gamma)$. Let $w_\eps$ and $\overline{w}$ be weak solutions to
\begin{equation*}
\left\{
\begin{aligned}
L_\eps w_\eps &= f_\eps
&\qquad& \text{in }\Omega, \\
w_\eps &= g
&\qquad& \text{in }\mathbb{R}^n\setminus\Omega,
\end{aligned}
\right.
\qquad
\left\{
\begin{aligned}
\overline{L} \overline{w} &= f_\eps
&\qquad& \text{in }\Omega, \\
\overline{w} &= g
&\qquad& \text{in }\mathbb{R}^n\setminus\Omega.
\end{aligned}
\right.
\end{equation*}
Then by standard energy estimates, we have
\begin{align}\label{conv.sub:RHS}
    \|w_\eps-u_\eps\|_{L^2(\Omega)}+\|\overline{w}-\overline{u}\|_{L^2(\Omega)}\leq c(n,s,\Lambda,N)\|f_\eps-f\|_{H^{-s}(\bbR^n)}\leq c(n,s,\Lambda,N)\eps^{\delta s}\|f\|_{L^2(\Omega)}.
\end{align}
Next, let $v_\eps$ and $\overline{v}$ be weak solutions to 
\begin{equation*}
\left\{
\begin{aligned}
L_\eps v_\eps &= f_\eps
&\qquad& \text{in }\Omega, \\
v_\eps &= g_\eps
&\qquad& \text{in }\mathbb{R}^n\setminus\Omega,
\end{aligned}
\right.
\qquad
\left\{
\begin{aligned}
\overline{L} \overline{v} &= f_\eps
&\qquad& \text{in }\Omega, \\
\overline{v} &= g_\eps
&\qquad& \text{in }\mathbb{R}^n\setminus\Omega.
\end{aligned}
\right.
\end{equation*}
Then, by \autoref{lem.conv} and \eqref{conv.sub:mol}, we have 
\begin{equation}\label{conv.sub:ineq1}
\begin{aligned}
    \|v_\eps-\overline{v}\|_{L^2(\Omega)}&\leq c\left(\max\{\eps^{s},\eps^{2-2s}\}\right)\left(\|g_\eps\|_{C^1(B_2(\Omega)\setminus \Omega)}+\| g_\eps\|_{L^1_{2s}(\bbR^n\setminus \Omega)}+\| \nabla g_\eps\|_{L^1_{2s}(\bbR^n\setminus \Omega)}+\|f_\eps\|_{C^{1+\alpha}(\Omega)}\right)\\
        &\leq  c\eps^{\frac{\min\{s,2-2s\}}2}(\|g\|_{L^\infty(\bbR^n)}+\|f\|_{L^\infty(\Omega)})
\end{aligned}
\end{equation}
for some constant $c=c(n,s,\Lambda,\alpha,[K]_{C^\alpha(\mathbb{T}^n)},r_0,M,N,\gamma)$. Now we estimate the $L^1$ norms of $v_\eps-u_\eps$ and $\overline{v}-\overline{u}$. As $L_\eps(v_\eps-u_\eps)=0$, by the Poisson formula, we get
\begin{align*}
    (v_\eps-w_\eps)(x)=\int_{\bbR^n\setminus \Omega}({g_\eps(y)-g(y)})P_\eps(x,y)\,dy.
\end{align*}
By the definition of $g_\eps$ and using H\"older regularity of $g$, we have 
\begin{align*}
    |(v_\eps-w_\eps)(x)|&=\left|\int_{\bbR^n\setminus \Omega}\left(\int_{B_{\eps^\delta}}(g(y-z)-g(y))\eta_{\eps^\delta}(z)\,dz\right)P_\eps(x,y)\,dy\right|\leq c\eps^{\delta\gamma}[g]_{C^\gamma(\bbR^n)}\int_{\bbR^n\setminus \Omega}P_\eps(x,y)\,dy.
\end{align*}
Thus, using \eqref{green:pois}, we further estimate the above quantity as follows
\begin{align*}
    |(v_\eps-w_\eps)(x)|\leq c\eps^{\delta\gamma}[g]_{C^\gamma(\bbR^n)}\left(\underbrace{\int_{B_2(\Omega)\setminus \Omega}\frac{d_\Omega^\beta(x)}{d_\Omega^{2s-\beta}(y)|x-y|^{n}}\,dy}_{\eqqcolon J}+\int_{\bbR^n\setminus B_2(\Omega)}\frac{\,dy}{(1+|x-y|)^{n+2s}}\,dy\right),
\end{align*}
where $\beta\in(0,s)$ and $c=c(n,s,\Lambda,[K]_{C^\alpha(\mathbb{T}^n)},\alpha,r_0,M,N,\gamma,\beta)$. By choosing $\beta$ close to $s$, $\lambda$ close to $0$ and applying the first claim in \autoref{lem.singdist}, we estimate $J$ by
\begin{align*}
    J\leq c\int_{B_2(\Omega)\setminus B_{4\eps^\delta}(\Omega)}\frac{d_\Omega^{\beta-\lambda}(x)}{d_\Omega^{2s-\beta}(y)|x-y|^{n-\lambda}}\,dy\leq cd_\Omega^{-2s+2\beta}(x),
\end{align*}
where $c=c(n,s,\Lambda,[K]_{C^\alpha(\mathbb{T}^n)},\alpha,r_0,M,N,\gamma,\beta)$. Note that although the first claim in \autoref{lem.singdist} is stated in the domain $\Omega^1$, we can also get the same estimates when we replace it with $B_2(\Omega)$ by following the same lines as in the proof of \autoref{lem.singdist}. Thus, we have 
\begin{align}\label{conv.sub:estI_21}
    |(v_\eps-w_\eps)(x)|\leq c\eps^{\delta\gamma}[g]_{C^\gamma(\bbR^n)}(d_\Omega^{-2s+2\beta}(x)+1),
\end{align}
where $c=c(n,s,\Lambda,[K]_{C^\alpha(\mathbb{T}^n)},\alpha,r_0,M,N,\gamma,\beta)$. Choosing $\beta$ close to $s$, we deduce
\begin{align*}
    \|w_\eps-v_\eps\|_{L^1(\Omega)}\leq c\eps^{\delta\gamma}[g]_{C^\gamma(\bbR^n)},
\end{align*}
where $c=c(n,s,\Lambda,[K]_{C^\alpha(\mathbb{T}^n)},\alpha,r_0,M,N,\gamma)$.
Similarly, we have 
\begin{align*}
    \|\overline{w}-\overline{v}\|_{L^1(\Omega)}\leq c\eps^{\delta\gamma}[g]_{C^\gamma(\bbR^n)},
\end{align*}
where $c=c(n,s,\Lambda,[K]_{C^\alpha(\mathbb{T}^n)},\alpha,r_0,M,N,\gamma)$. 
Combining the above two inequalities, \eqref{conv.sub:RHS} and \eqref{conv.sub:ineq1} and choosing $\sigma=\sigma(s,\gamma)$ sufficiently small, we have the desired estimate.
\end{proof}

\section{Large scale boundary regularity of optimal order}
\label{sec6}
In this section, we establish the optimal boundary regularity, namely $C^s$ regularity, for nonlocal homogenization problems. Using this result, we then derive sharp estimates for the Green function and the Poisson kernel. In the last subsection, using the sharp Green function estimates, we improve the convergence rate given in \autoref{lem.conv} when $s=\frac12$.
\subsection{Optimal H\"older regularity}
The following theorem is the main result of this subsection.
\begin{theorem}\label{thm.bdry.C^s}
Let $\alpha\in(0,1]$, $\eps\in(0,1]$, $r_0,M$, $K\in C^\alpha(\mathbb{T}^n)$ and $\partial\Omega\in C^{1,\alpha}$ with $C^{1,\alpha}$ radius $r_0$, $C^{1,\alpha}$ norm bounded by $M$. Let $u_\eps$ be a weak solution to 
\begin{equation*}
\left\{
\begin{alignedat}{3}
L_\eps u_\eps&=f&&\qquad \mbox{in  $\Omega\cap B_R(x_0)$}, \\
u_\eps&=0&&\qquad  \mbox{in $B_R(x_0)\setminus \Omega$},
\end{alignedat} \right.
\end{equation*}
where $f\in L^\infty(B_R(x_0))$. Then we have 
\begin{align*}
    [u_\eps]_{C^s(B_{R/2}(x_0))}\leq cR^{-s}(E(u_\eps;B_R(x_0))+\|R^{2s}f\|_{L^\infty(B_R(x_0))})
\end{align*}
for some constant $c=c(n,s,\Lambda,\alpha,[K]_{C^\alpha(\mathbb{T}^n)},r_0,M)$.
\end{theorem}
To prove this, we first obtain suitable decay estimates at the boundary via some perturbation argument together with \autoref{lem.conv.sub}.
\begin{lemma}\label{lem.decay.C^s}
   Let $\eps\in(0,1]$ and let $\alpha\in(0,s)$, $r_0,M>0$. Let $\Omega$ be a $C^{1,\alpha}$ domain with $C^{1,\alpha}$ radius $r_0$ and $C^{1,\alpha}$ norm bounded by $M$. Let $ f\in L^\infty(B_1)$ and $u_\eps$ be a weak solution to 
    \begin{equation*}
\left\{
\begin{alignedat}{3}
L_\eps u_\eps&=f&&\qquad \mbox{in  $\Omega\cap B_1$}, \\
u_\eps&=0&&\qquad  \mbox{in $B_1\setminus \Omega$}
\end{alignedat} \right.
\end{equation*}
with $0\in\partial\Omega$. Then we have 
\begin{align}\label{decay.C^s:main.goal}
    E(u_\eps-(u_\eps)_{B_r};B_r)\leq cr^s(E(u_\eps;B_1)+\|f\|_{L^\infty(B_1)})
\end{align}
for some constant $c=c(n,s,\Lambda,\alpha,[K]_{C^\alpha(\mathbb{T}^n)},r_0,M)$.
\end{lemma}
\begin{proof}
 By the assumption on the domain, there is a constant $c_1=c_1(n,\alpha,r_0,M)$ such that the measure density condition given in \cite[(2.1)]{BKK25} holds. Thus, as in the proof of \cite[Lemma 2.7]{BKK25}, we have for any $r\in(0,\frac14]$,
    \begin{align}\label{decay.C^s:loc.bdd}
        r^{-\gamma}[u_\eps]_{C^\gamma(B_{\frac{r}{2}})}+\|u_\eps\|_{L^\infty(B_{\frac{r}{2}})}\leq c(n,s,\Lambda,\alpha,r_0,M)(E_1(u_\eps;B_{r})+\|r^{2s}f\|_{L^\infty(B_r)}),
    \end{align}
    where $\gamma=\gamma(n,s,\Lambda,\alpha,r_0,M)$.
    We now divide the proof into two steps.
    
\textbf{Step 1. Expansion estimates at the boundary point.} First, we are going to prove the following: For any $\rho\in(0,\frac1{32}]$ and $r\in(0,\frac14]$, we have 
\begin{equation}\label{decay.C^s:first.goal}
\begin{aligned}
   \inf_{a\in\bbR} E(u_\eps-a\varphi;B_{\rho r})&\leq c\left[\rho^{s+\alpha}\inf_{a\in\bbR}E(u_\eps-a\varphi;B_r)+
    (\eps/r)^{\sigma}\rho^{-n}E(u_\eps;B_r)+\rho^{-n}\|r^{2s}f\|_{L^\infty(B_r)}\right]
\end{aligned}
\end{equation}
for some constant $c=c(n,s,\Lambda,\alpha,[K]_{C^\alpha(\mathbb{T}^n)},r_0,M)$, where $\varphi$ is a weak solution to 
\begin{equation*}
\left\{
\begin{alignedat}{3}
\overline{L} \varphi&=0&&\qquad \mbox{in  $\Omega\cap B_{1}$}, \\
\varphi&=g&&\qquad  \mbox{in $\bbR^n\setminus ( \Omega\cap B_1)$}
\end{alignedat} \right.
\end{equation*}
for some $g\in C_c^\infty(B_{4}\setminus B_{2})$ with $g\equiv 1$ in $B_{3}\setminus B_{\frac52}$. By \cite[Proposition 2.6.4 and Proposition 2.6.6]{FRRO24}, we have 
\begin{align}\label{decay.C^s:reg.varphi}
   \|\varphi\|_{C^s(\bbR^n)}\leq c\quad\text{and}\quad  d_\Omega^s/c\leq \varphi \leq cd_\Omega^s\quad\text{in }\Omega\cap B_{\frac34}
\end{align}
for some constant $c=c(n,s,\Lambda,\alpha,r_0,M)$. Moreover, by modifying the argument given in the proof of \autoref{lem.funct}, we can prove that for any $R\leq \frac14$, there is a constant $c_R\in\bbR$ such that
\begin{align*}
    \inf_{a\in\bbR}E(u_\eps-a\varphi;B_R)= E(u_\eps-c_R\varphi;B_R).
\end{align*}
In addition, for any $R\leq\frac14$, by \eqref{decay.C^s:reg.varphi}, we have 
\begin{align}\label{decay.C^s:c_r}
    R^s|c_R|\leq c\dashint_{B_R}{c_R\varphi }\leq c(E(u_\eps-c_R\varphi;B_R)+E(u_\eps;B_R))\leq cE(u_\eps;B_R)
\end{align}
for some constant $c=c(n,s,\alpha,r_0,M)$.

To prove \eqref{decay.C^s:first.goal}, fix $r\in(0,\frac14]$. Then there is a $C^{1,\alpha}$-domain $D_r$ such that
     \begin{align*}
        \Omega\cap B_{\frac{r}{16}}\subset D_r\subset \Omega\cap B_{\frac{r}8}
    \end{align*}
    and such that the $C^{1,\alpha}$ radius of the domain $\frac{8}rD_r$ is $r_D$ and the $C^{1,\alpha}$ norm of $\frac{8}rD_r$ is bounded by $M_D$, where $r_D=r_D(n,\alpha,r_0,M)$ and $M_D=M_D(n,\alpha,r_0,M)$. Let $\varphi_\eps$ be the weak solution to 
    \begin{equation*}
\left\{
\begin{aligned}
{L}_\eps \varphi_\eps&=0&&\qquad \mbox{in  $D_r$}, \\
\varphi_\eps&=\varphi&&\qquad  \mbox{in $\bbR^n\setminus D_r$}
\end{aligned}
\right.
\end{equation*}
and note that for $\varphi_{\eps,r}(x)\coloneqq \varphi_\eps(rx/8)$ and $\varphi_r(x)\coloneqq \varphi(rx/8)$, $\overline{L}\varphi_r=L_{\eps/r}\varphi_{\eps,r}=0$ in $8D_r/r$. Thus, applying \autoref{lem.conv.sub} and \eqref{decay.C^s:reg.varphi}, we have
\begin{equation}\label{decay.C^s:varphieps2}
    \dashint_{B_r}|\varphi_\eps-\varphi|\,dx\leq cr^s(\eps/r)^{\sigma}
\end{equation}
for some $\sigma=\sigma(n,s)$. In addition, as in \eqref{decay.C^s:loc.bdd}, we deduce
\begin{equation}\label{decay.C^s:varphieps}
\begin{aligned}
    r^{\gamma}[u_\eps-c_r\varphi_\eps]_{C^\gamma(B_{\frac{3r}{64}})}+\|u_\eps-c_r\varphi_\eps\|_{L^\infty(B_{\frac{3r}{64}})}&\leq c(E(u_\eps-c_r\varphi_\eps;B_r)+r^{2s}\|f\|_{L^\infty(B_r)})\\
    &\leq c(E(u_\eps-c_r\varphi;B_r)+(r/\eps)^\sigma E(u_\eps;B_r)+r^{2s}\|f\|_{L^\infty(B_r)}),
\end{aligned}
\end{equation}
where we have used \eqref{decay.C^s:varphieps2} and \eqref{decay.C^s:c_r} in the last step.

Let $\xi\equiv 1$ in $B_{\frac{5r}{128}}$ with $\xi\in C_c^\infty(B_{\frac{3r}{64}})$ and $v_{\eps}\coloneqq (u_{\eps}-c_r\varphi_\eps)\xi$ to see that 
\begin{equation*}
    L_{{\eps}}v_{\eps}=h_{\eps}\quad\text{in }\frac12D_r
\end{equation*}
 and 
 \begin{equation}\label{decay.C^s:h_eps}
 \begin{aligned}
     \|h_{\eps}\|_{L^\infty(B_{\frac{3r}{64}})}&\leq cr^{-2s}E(v_{\eps};B_r)+\|f\|_{L^\infty(B_r)}\\
     &\leq c\left(r^{-2s}(E(u_\eps-c_r\varphi;B_r)+(\eps/r)^\sigma E(u_\eps;B_r))+\|f\|_{L^\infty(B_r)}\right),
     \end{aligned}
 \end{equation}
 where we have used \eqref{decay.C^s:c_r} and \eqref{decay.C^s:varphieps2}. Next, let $\overline{v}$ be a weak solution to 
\begin{equation*}
\left\{
\begin{aligned}
\overline{L} \overline{v}&=h_\eps&&\qquad \mbox{in  $\frac12D_r$}, \\
\overline{v}&=v_{\eps}&&\qquad  \mbox{in $\bbR^n\setminus \frac12D_r$}.
\end{aligned}
\right.
\end{equation*}
Then we have for $v_{\eps,r}(x)=v_\eps(rx/8)$ and $\overline{v}_r(x)=\overline{v}(rx/8)$, $L_\eps v_{\eps,r}(x)=\overline{L} \overline{v}_r(x)=r^{2s}h_\eps(rx/8)$ in $\frac{4}rD_r$ and $v_{\eps,r}=\overline{v}_r$ in $\bbR^n\setminus \frac{4}rD_r$.
Applying \autoref{lem.conv.sub} together with  \eqref{decay.C^s:reg.varphi}, \eqref{decay.C^s:varphieps} and \eqref{decay.C^s:c_r}, we have 
\begin{align*}
    \int_{\frac{8}{r}D_r}|v_{\eps,r}-\overline{v}_r|\,dx&\leq c({\eps}/{r})^{\sigma}\left(\|v_{\eps,r}\|_{C^\gamma(\bbR^n)}+\|r^{2s}h_{\eps}(rx/8)\|_{L^\infty(\frac{4}{r}D_r)}\right)\\
    &\leq c((\eps/r)^\sigma(E(u_\eps;B_r))+r^{2s}\|f\|_{L^\infty(B_r)})
\end{align*}
for some constant $c=c(n,s,\Lambda,\alpha,[K]_{C^\alpha(\mathbb{T}^n)},r_0,M)$. As $\frac{r}8\left(B_3(\frac{8}{r}D_r)\right)\subset B_{\frac{r}2}$, scaling back, we have 
\begin{align}\label{decay.C^2:l1comp}
    \dashint_{B_r}|v_\eps-\overline{v}|\leq c((\eps/r)^\sigma(E(u_\eps;B_r))+r^{2s}\|f\|_{L^\infty(B_r)}),
\end{align}
where $c=c(n,s,\Lambda,\alpha,[K]_{C^\alpha(\mathbb{T}^n)},r_0,M)$. By following the proof in \cite[Lemma 8.1]{KW25} together with \cite[Proposition 2.7.8]{FRRO24} and a careful inspection of the dependence of the constant on the domain, we deduce for any $\rho\leq \frac{1}{16}$,
\begin{align}\label{decay.C^s:vbar.reg}
 \|\overline{v}/\varphi\|_{L^\infty(\Omega_{\frac{r}{32}})}+  \rho^{-\alpha}\dashint_{\Omega_{\rho r}}|\overline{v}/\varphi-(\overline{v}/\varphi)_{\Omega_{\rho r}}|\,dx \leq c\left(\frac{1}{r^{s}}E(\overline{v};B_{\frac{r}{16}})+r^{s}\|h_\eps\|_{L^\infty(B_r)}\right)
\end{align}
for some constant $c=c(n,s,\Lambda,\alpha,r_0,M)$, where we write $\Omega_\rho\coloneqq \Omega\cap B_\rho$. 
Using this, we want to prove
\begin{align}\label{decay.C^s:est.I}
   I\coloneqq E(\overline{v}-\varphi(\overline{v}/\varphi)_{\Omega_{\rho r}};B_{\rho r})\leq c\rho^{s+\alpha}\left(E(\overline{v};B_{\frac{r}{16}})+r^{2s}\|h_\eps\|_{L^\infty(B_r)}\right)
\end{align}
for some constant $c=c(n,s,\Lambda,\alpha,[K]_{C^\alpha(\mathbb{T}^n)},r_0,M)$. To this end, from the tail estimate in \cite[Lemma 2.10]{KNS22} and the fact that $\|\varphi\|_{L^\infty(B_{\rho })}\leq c\rho^{s}$ for any $\rho\leq \frac12$ by \eqref{decay.C^s:reg.varphi} and $0\in\partial\Omega$, we deduce
\begin{equation*}
\begin{aligned}
    I&\leq c\left[\sum_{i=1}^k 2^{-si}(\rho r)^s\sum_{j=0}^i\dashint_{\Omega_{2^j\rho r}}|\overline{v}/\varphi-(\overline{v}/\varphi)_{\Omega_{2^j\rho r}}|+2^{-2sk}r^{s}\dashint_{\Omega_{\frac{r}{32}}}|\overline{v}/\varphi-(\overline{v}/\varphi)_{\Omega_{ \frac{r}{32}}}|\right]\\
    &\quad+ c(\rho r)^{2s}\int_{{\bbR^n\setminus B_{\frac{r}{32}}}}\frac{|\overline{v}-\varphi(\overline{v}/\varphi)_{\Omega_{\frac{r}{32}}}|}{|y|^{n+2s}}\,dy,
    \end{aligned}
\end{equation*}
for some positive integer $k$ such that $\frac1{64}<2^k\rho\leq \frac{1}{32}$.
Therefore, using \eqref{decay.C^s:reg.varphi} and \eqref{decay.C^s:vbar.reg}, we further estimate $I$ as 
\begin{align*}
    I&\leq c\left[\sum_{i=1}^k\sum_{j=0}^i2^{-si}\rho^s (2^j\rho)^{\alpha}\left(E(\overline{v};B_{\frac{r}{16}})+r^{2s}\|h_\eps\|_{L^\infty(B_r)}\right)+\rho^{2s}(E(\overline{v};B_{\frac{r}{16}})+r^{2s}\|h_\eps\|_{L^\infty(B_r)}\right]\\
    &\leq c\rho^{s+\alpha}(E(\overline{v};B_{\frac{r}{16}})+r^{2s}\|h_\eps\|_{L^\infty(B_r)})
\end{align*}
for some constant $c=c(n,s,\Lambda,\alpha,[K]_{C^\alpha(\mathbb{T}^n)},r_0,M)$, which proves \eqref{decay.C^s:est.I}.

Therefore, using \eqref{decay.C^s:est.I}, \eqref{decay.C^s:varphieps2} and \eqref{decay.C^2:l1comp}, we derive 
\begin{align*}
    E(u_\eps-\varphi((\overline{v}/\varphi)_{\Omega_{\rho r}}+c_r);B_{\rho r})&\leq E(u_\eps-c_r\varphi-v_\eps;B_{\rho r})+E(v_\eps-\overline{v};B_{\rho r})+E(\overline{v}-\varphi(\overline{v}/\varphi)_{\Omega_{\rho r}};B_{\rho r})\\
    &\leq c(\rho^{2s}E(u_\eps-c_r\varphi;B_r)+\rho^{-n}(\eps/r)^\sigma E(u_\eps;B_r))\\
    &\quad+cr^{2s}\rho^{-n}\|f\|_{L^\infty(B_r)}+\rho^{s+\alpha}E(\overline{v};B_{\frac{r}{16}})\\
    &\leq c(\rho^{2s}E(u_\eps-c_r\varphi;B_r)+\rho^{-n}((\eps/r)^\sigma E(u_\eps;B_r)+r^{2s}\|f\|_{L^\infty(B_r)}.
\end{align*}
Now using the definition of $c_r$ in \eqref{decay.C^s:c_r}, we have \eqref{decay.C^s:first.goal}.

\textbf{Step 2. Decay estimates at the boundary.} Now using \eqref{decay.C^s:first.goal}, we prove \eqref{decay.C^s:main.goal}. First, for each $r\leq\frac14$, we define a new excess functional $\mathcal{E}$ by
\begin{align*}
    \mathcal{E}(u;B_r)\coloneqq \frac{E(u_\eps-c_r\varphi;B_r)}{r^s}.
\end{align*}
For any $\rho\leq \frac1{16}$ and $i\in\N$, by \eqref{decay.C^s:first.goal} with $r=\rho^i$, we have 
\begin{align*}
    \mathcal{E}(u_\eps;B_{\rho^{i+1}})\leq c\left[\rho^{\alpha}\mathcal{E}(u_\eps;B_{\rho^i})+(\eps/\rho^i)^\sigma\rho^{-n}(\mathcal{E}(u_\eps;B_{\rho^i})+|c_{\rho^i}|)+\rho^{-n}\|\rho^{si}f\|_{L^\infty(B_{\rho^i})}\right].
\end{align*}
Now we choose a small constant $\rho=\rho(n,s,\Lambda,\alpha,[K]_{C^\alpha(\mathbb{T}^n)},r_0,M)$ such that
\begin{equation*}
    c\rho^{\alpha}\leq \frac{1}{8}.
\end{equation*}
Next, we observe 
\begin{align}\label{decay.C^s:cii}
    |c_{\rho ^i}-c_{\rho^{i-1} }|\leq c\rho^{-si}\dashint_{B_{\rho^{i}}}|(c_{\rho^i}-c_{\rho^{i-1}})\varphi|\leq c(\mathcal{E}(u_\eps;B_{\rho^i})+\mathcal{E}(u_\eps;B_{\rho^{i-1}}))
\end{align}
for some constant $c=c(n,s,\Lambda,\alpha,[K]_{C^\alpha(\mathbb{T}^n)},r_0,M)$. 
Therefore, we deduce
\begin{align}\label{decay.C^s:ineq1}
    \sum_{i=1}^l\mathcal{E}(u_\eps;B_{\rho^{i+1}})\leq \sum_{i=1}^l\left(\frac{1}{8}+c(\eps/\rho^i)^\sigma\right)\mathcal{E}(u_\eps;B_{\rho^{i}})+c\sum_{i=1}^l(\eps/\rho^i)^\sigma\sum_{j=1}^{i}\mathcal{E}(u_\eps;B_{\rho^{j}})+c\sum^l_{i=1}\rho^{si}\|f\|_{L^\infty(B_1)}
\end{align}
for some constant $c=c(n,s,\Lambda,\alpha,[K]_{C^\alpha(\mathbb{T}^n)},r_0,M)$, where we have used $|c_\rho|\leq cE(u_\eps;B_\rho)$ by \eqref{decay.C^s:c_r}. Now we choose a small constant $\eps_0$ depending only on $n,s,\Lambda,\alpha,[K]_{C^\alpha(\mathbb{T}^n)},r_0,M$ such that 
\begin{align*}
    \frac{c\eps_0^\sigma}{1-\rho^{\sigma}}\leq \frac18.
\end{align*}

Next, suppose that there is a positive integer $l\in\N$ such that $\eps/\eps_0\leq \rho^l$. Then for any $i\leq l$,
\begin{align*}
    c(\eps/\rho^i)^\sigma\leq c(\eps/\rho^l)^\sigma\leq \frac18\quad\text{and}\quad \sum_{i=1}^{l}(\eps/\rho^i)^\sigma\leq \frac{(\eps/\rho^l)^{\sigma}}{1-\rho^\sigma}\leq \frac{\eps_0^{\sigma}}{1-\rho^\sigma}\leq \frac18.
\end{align*}
Now combining this and \eqref{decay.C^s:ineq1}, we get
\begin{align*}
    \sum_{i=1}^l\mathcal{E}(u_\eps;B_{\rho^{i+1}})&\leq \frac14\sum_{i=1}^l\mathcal{E}(u_\eps;B_{\rho^{i}})+\sum_{j=1}^l\mathcal{E}(u_\eps;B_{\rho^j})\sum_{i=j}^l(\eps/\rho^i)^\sigma+c\|f\|_{L^\infty(B_1)}\\
    &\leq \frac12\sum_{i=1}^l\mathcal{E}(u_\eps;B_{\rho^{i}})+c\|f\|_{L^\infty(B_1)}.
\end{align*}
Therefore, using this and \eqref{decay.C^s:c_r} and \eqref{decay.C^s:cii}, if $\eps\leq \eps_0 \rho^l$ for some $l\in\N$, then 
\begin{align*}
   \sum_{i=1}^l \mathcal{E}(u_\eps;B_{\rho^{i+1}})\leq c(\mathcal{E}(u_\eps;B_\rho)+\|f\|_{L^\infty(B_1)})\leq c(E(u_\eps;B_1)+\|f\|_{L^\infty(B_1)})
\end{align*}
and 
\begin{align*}
    |c_{\rho^{l+1}}|\leq \sum_{j=1}^{l}|c_{\rho^{j+1}}-c_{\rho^{j}}|+|c_\rho|\leq c\sum_{j=1}^{l+1}\mathcal{E}(u_\eps;B_{\rho^i})\leq c(E(u_\eps;B_1)+\|f\|_{L^\infty(B_1)}),
\end{align*}
where $c=c(n,s,\Lambda,\alpha,[K]_{C^\alpha(\mathbb{T}^n)},r_0,M)$. Thus, we have $\eps\leq \eps_0\rho^l$ for some $l$, then 
\begin{align*}
    \rho^{-s(l+1)}E(u_\eps;B_{\rho^{l+1}})\leq \mathcal{E}(u_\eps;B_{\rho_{l+1}})+c|\rho_{l+1}|\leq c(E(u_\eps;B_1)+\|f\|_{L^\infty(B_1)}).
\end{align*}
Since $\eps_0$ and $\rho$ depend only on the given data, we deduce for any $\eps\leq r$,
\begin{align*}
    E(u_\eps;B_r)\leq cr^{s}(E(u_\eps;B_1)+\|f\|_{L^\infty(B_1)}),
\end{align*}
where $c=c(n,s,\Lambda,\alpha,[K]_{C^\alpha(\mathbb{T}^n)},r_0,M)$.
On the other hand, when $r\leq\eps$, we deduce from \autoref{lem.reg.bdd1} that 
\begin{align*}
   E(u_\eps;B_r)\leq c(r/\eps)^{s}(E(u_\eps;B_\eps)+\eps^{2s}\|f\|_{L^\infty(B_\eps)}), 
\end{align*}
where $c=c(n,s,\Lambda,\alpha,[K]_{C^\alpha(\mathbb{T}^n)},r_0,M)$. Thus, combining the above two inequalities yields the desired estimate. This completes the proof.
\end{proof}

Now we are ready to prove \autoref{thm.bdry.C^s}. 
\begin{proof}[Proof of \autoref{thm.bdry.C^s}]
    Combining the boundary estimate and interior estimate given in \autoref{lem.decay.C^s} and \autoref{lem.exp.sol} with $\beta=s$, respectively, together with a standard covering argument, we derive the desired result.
\end{proof}
Therefore, combining \autoref{thm.uni.est} and \autoref{thm.bdry.C^s}, we get \autoref{thm:intro:regularity}.
\subsection{Sharp Green function estimate}
Using \autoref{thm.bdry.C^s} and following the argument provided in \cite{KW25}, we have the following Green function and Poisson kernel estimates.
\begin{theorem}\label{thm.green.C^s}
     Let $\eps\in(0,1)$, $r_0,M,N>0$ $\alpha\in(0,1]$, $K\in C^\alpha(\mathbb{T}^n)$ and $\partial\Omega\in C^{1,\alpha}$with $C^{1,\alpha}$ radius $r_0$, $C^{1,\alpha}$ norm bounded by $M$ and $\mathrm{diam}(\Omega)\leq N$. Let $G_\eps$ and $P_\eps$ be the Green function and the Poisson kernel associated to $L_\eps$ and $\Omega$. Then we have \begin{equation}\label{green.C^s:green}
   \begin{aligned}
       G_\eps(x,y)\leq c\min\left\{\frac{d_\Omega(x)}{|x-y|},1\right\}^{s}\min\left\{\frac{d_\Omega(y)}{|x-y|},1\right\}^{s}\frac{1}{|x-y|^{n-2s}},
   \end{aligned}
   \end{equation}
   In addition, we have
   \begin{align}\label{green.C^s:pois}
       P_\eps(x,y)\leq c\frac{d_\Omega^{s}(x)}{d_\Omega^{s}(y)(1+d_\Omega(y))^{s}}\frac{1}{|x-y|^n},
   \end{align}
    where $c=c(n,s,\Lambda,,[K]_{C^\alpha(\mathbb{T}^n)},\alpha,r_0,M,N)$. 
\end{theorem}
The proof of \eqref{green.C^s:green} and \eqref{green.C^s:pois} is exactly the same as in the proof of \autoref{lem.green} by applying \autoref{thm.bdry.C^s} instead of \autoref{thm.uni.est}.

\subsection{Sharp convergence rate for $s=\frac12$}
In this subsection, we obtain sharp convergence rates when $s=\frac12$ by using the improved Green function estimates from \autoref{thm.green.C^s} instead of \autoref{lem.green}. Here is the main result of this subsection.
\begin{theorem}\label{lem.con.s=1/2}
     Let $s=\frac12$, $\alpha\in(0,1]$, $\eps\in(0,1]$, $r_0,M,N>0$, $h\in C^{1+\alpha}(\Omega)$, $K\in C^\alpha(\mathbb{T}^n)$ and $\partial\Omega\in C^{1,\alpha}$ with $C^{1,\alpha}$ radius $r_0$, $C^{1,\alpha}$ norm bounded by $M$, $\mathrm{diam}(\Omega)\leq N$ and $0\in\Omega$.
    Let $u_\eps$ and $\overline{u}$ be weak solutions to 
    \begin{equation*}
\left\{
\begin{aligned}
L_\eps u_\eps &= h
&\qquad& \text{in }\Omega, \\
u_\eps &= g
&\qquad& \text{in }\mathbb{R}^n\setminus\Omega,
\end{aligned}
\right.
\qquad
\left\{
\begin{aligned}
\overline{L} \overline{u} &= h
&\qquad& \text{in }\Omega, \\
\overline{u} &= g
&\qquad& \text{in }\mathbb{R}^n\setminus\Omega,
\end{aligned}
\right.
\end{equation*}
where $\overline{K}\coloneqq\int_{\mathbb{T}^n\times \mathbb{T}^n}K(x,x-z)\,dz\,dx$ and $g\in C^{1}(B_2(\Omega)\setminus \Omega)\cap L^1_{2s}(\bbR^n\setminus \Omega)$ with $\nabla g\in L^1_{2s}(\bbR^n\setminus \Omega)$. 
Then we have
\begin{equation*}
\begin{aligned}
    \|u_\eps-\overline{u}\|_{L^2(\Omega)}\leq 
        c\eps(1+|\ln\eps|)^{\frac32}\left(\|h\|_{C^{1+\alpha}(\Omega)}+\|g\|_{C^{1}(B_2(\Omega)\setminus \Omega)}+\| \nabla g\|_{L^1_{2s}(\bbR^n\setminus \Omega)}+\| g\|_{L^1_{2s}(\bbR^n\setminus \Omega)}\right)
\end{aligned}
\end{equation*}
for some constant $c=c(n,\Lambda,[K]_{C^\alpha(\mathbb{T}^n)},\alpha,r_0,M,N)$. 
\end{theorem}
\begin{proof}
The proof is exactly the same as that of \autoref{lem.conv} except that we use \autoref{thm.green.C^s}, instead of \autoref{lem.green}. First, for simplicity, we omit the dependence of the constant, when it depends only on $n,\Lambda,[K]_{C^\alpha(\mathbb{T}^n)},\alpha,r_0,M,N$. 
 
First, we assume
 \begin{align*}
     \|h\|_{C^{1,\alpha}(\Omega)}+\|g\|_{C^1(B_2(\Omega)\setminus \Omega)}+\|g\|_{L^1_{2s}(\bbR^n\setminus \Omega)}+\|\nabla g\|_{L^1_{2s}(\bbR^n\setminus \Omega)}\leq 1.
 \end{align*}
We may assume $h\in C_c^{1,\alpha}(B_2(\Omega))$ with $\|h\|_{C^{1,\alpha}(\bbR^n)}\leq c$. Let $\overline{u}_\eps$ be a weak solution to
  \begin{equation*}
\left\{
\begin{aligned}
\overline{L} \overline{u}_\eps &= h
&\qquad& \text{in }\Omega^\eps, \\
\overline{u} &= g
&\qquad& \text{in }\mathbb{R}^n\setminus\Omega^\eps,
\end{aligned}
\right.
\end{equation*}
where $\Omega^\eps$ is the $\eps$-neighborhood defined by \eqref{defn.omegaeps}. We may assume \eqref{conv:eps0.Omega} by taking $\eps$ sufficiently small. By \autoref{lem.reg.bdd2}, we have for any $x,y\in \Omega^\eps$ and $\mu\in(0,1]$,
\begin{align}\label{conv.s=1/2:ineq0}
    &\|\overline{u}_\eps\|_{C^s(B_2(\Omega)}\leq c, \quad|\nabla{u}_\eps(x)|\leq cd_\eps^{-\frac12}(x),\quad \frac{|\nabla\overline{u}_\eps(x)-\nabla\overline{u}_\eps(y)|}{|x-y|^\mu}\leq c(d_\eps^{-\frac12-\mu}(x)+d_\eps^{-\frac12-\mu}(y)),\\
    &\frac{|\nabla^2\overline{u}_\eps(x)-\nabla^2\overline{u}_\eps(y)|}{|x-y|^{\alpha}}\leq c(d_\eps^{-\frac32-\alpha}(x)+d_\eps^{-\frac32-\alpha}(y)).
\end{align}
Next, we consider the function $w_\eps\coloneqq u_\eps-(\overline{u}_\eps+\eps\phi(\cdot/\eps)\cdot\nabla(\overline{u}_\eps\zeta)+\eps \psi(\cdot/\eps)h)$, where $\phi=(\phi_i)_i$, $\psi$ and $\zeta$ are determined in \eqref{conv:eq.phipsi} and \eqref{conv:defn.zeta}. As in \textbf{Step 2} given in the proof of \autoref{lem.conv}, using \autoref{thm.green.C^s} in place of \autoref{lem.green}, we have 
\begin{align}\label{conv.s=1/2:defn.I}
    \|w_\eps-(u_\eps-\overline{u}_\eps)\|^2_{L^2(\Omega)}\leq c\eps\underbrace{\int_\Omega\left(\int_{\Omega^\eps\setminus \Omega}\frac{d_\Omega^{1/2}(x)}{d_\Omega^{1/2}(y)|x-y|^n}\,dy\right)^2\,dx}_{\eqqcolon I}+c\eps^2(1+|\ln\eps|).
\end{align}
We further estimate $I$ as 
\begin{align*}
    I&\leq c\left[\int_{\Omega_\eps}\left(\int_{\Omega^\eps\setminus \Omega}\frac{d_\Omega^{1/2}(x)}{d_\Omega^{1/2}(y)|x-y|^n}\,dy\right)^2\,dx+\int_{\Omega\setminus \Omega_\eps}\left(\int_{\Omega^\eps\setminus \Omega}\frac{d_\Omega^{1/4}(x)}{d_\Omega^{1/2}(y)|x-y|^{n-1/4}}\,dy\right)^2\,dx\right]\eqqcolon I_1+I_2,
\end{align*}
where we have used the fact that $d_\Omega(x)\leq |x-y|$ for $x\in\Omega$ and $y\in\Omega^c$.
For $I_1$, we use \eqref{singdist:2} to see that
\begin{align*}
    I_1\leq c\int_{\eps/c<d_\Omega(x)<c}\eps d_\Omega^{-1}(x)\,dx\leq c\eps(1+|\ln\eps|).
\end{align*}
For $I_2$, we use \eqref{singdist:3} to deduce 
\begin{align*}
    I_2\leq \int_{\Omega\setminus \Omega_\eps}\left(\int_{\Omega^\eps\setminus \Omega}\frac{d_\Omega^{1/4}(x)}{d_\Omega(y)^{1/2}|x-y|^{n-1/4}}\,dy\right)^2\,dx\leq c\int_{d_\Omega(x)\leq c\eps}d^{1/2}_\Omega(x)(\eps^{-1/2}+d_\Omega(x)^{-1/2})\,dx\leq c\eps.
\end{align*}
Combining all estimates yields $I\leq c\eps^2(1+|\ln\eps|)$, which implies
\begin{align}\label{con.s=1/2:weps-uubar}
     \|w_\eps-(u_\eps-\overline{u}_\eps)\|_{L^2(\Omega)}\leq c\eps(1+|\ln\eps|)^{1/2}.
\end{align}
Next, we consider two solutions $\theta_\eps$ and $z_\eps$ to
\begin{equation*}
\left\{
\begin{aligned}
L_\eps \theta_\eps &= 0
&\qquad& \text{in }\Omega, \\
\theta_\eps &= w_\eps
&\qquad& \text{in }\mathbb{R}^n\setminus\Omega,
\end{aligned}
\right.
\qquad
\left\{
\begin{aligned}
L_\eps z_\eps &= f_\eps
&\qquad& \text{in }\Omega, \\
z_\eps &= 0
&\qquad& \text{in }\mathbb{R}^n\setminus\Omega,
\end{aligned}
\right.
\end{equation*}
where $f_\eps $ is defined by \eqref{conv:defn.f_eps} with $u_\eps$ and $\overline{u}_\eps$ replaced by $v_\eps$ and $\overline{u}_\eps$, respectively. Using \eqref{green.C^s:pois}, we have \eqref{conv:thetaeps} with $\beta=s=1/2$. Moreover, for the estimate of terms $J_1,J_2$ and $J_{3,1},J_{3,3}$ given in \eqref{conv:thetaeps} and \eqref{conv:est.J_3} with $s=\beta=1/2$, we use the same argument as in the estimate of $I$ given by \eqref{conv.s=1/2:defn.I}. Then we have
\begin{align*}
    \|\theta_\eps\|^2_{L^2(\Omega)}\leq c\left(\eps^{2}(1+|\ln\eps|)+\eps^2\int_{\Omega}\left(\int_{\Omega^\eps\setminus \Omega^{\frac\eps2}}\frac{d_\eps^{-1/2}(y)d_\Omega^{1/2}(x)}{d_\Omega^{1/2}(y)|x-y|^n}\,dy\right)^2\,dx\right).
\end{align*}
 we have used the estimate of $I$. To further estimate, we observe
\begin{align*}
    L\coloneqq\int_{\Omega}\left(\int_{\Omega^\eps\setminus \Omega^{\frac\eps2}}\frac{d_\eps^{-1/2}(y)d_\Omega^{1/2}(x)}{d_\Omega^{1/2}(y)|x-y|^n}\,dy\right)^2\leq c\eps^{-1}  \int_{\Omega}\left(\int_{\Omega^\eps\setminus \Omega^{\frac\eps2}}\frac{\,dy}{d_\eps^{1/2}(y)(|x-y|+\eps)^{n-\frac12}}\right)^2 \, dx,
\end{align*}
where we have used the fact that $d_\Omega(x)\leq |x-y|$ and $|x-y|>\eps/c$ for any $x\in\Omega$ and $y\in\Omega^\eps\setminus \Omega^{\frac\eps2}$. Therefore,  by following the same flattening argument as in the estimate of $J$ given by \eqref{singdist:defn.J}, we further estimate $L$ as
\begin{equation*}
    L\leq c\int_{\Omega}\left(1+ \frac{1}{d_\Omega(x) + \eps} \right)\,dx\leq c(1 + |\ln \eps|).
\end{equation*}
Thus, we derive
\begin{align}\label{con.s=1/2:thetaeps}
    \|\theta_\eps\|_{L^2(\Omega)}\leq c\eps(1+|\ln\eps|)^{\frac12}.
\end{align}
Now we use \autoref{lem.rep.f_eps} to see that there is a function $z_{1,\eps}$ such that
\begin{equation}\label{conv.s=1/2:eqz1eps}
\left\{
\begin{aligned}
L_\eps z_{1,\eps} &= F
&\qquad& \text{in }\Omega, \\
z_{1,\eps} &= 0
&\qquad& \text{in }\mathbb{R}^n\setminus\Omega,
\end{aligned}
\right.
\qquad
\left\{
\begin{aligned}
L_\eps (z_\eps-z_{1,\eps}) &= (-\Delta)^s\widetilde{F}
&\qquad& \text{in }\Omega, \\
(z_\eps-z_{1,\eps}) &= 0
&\qquad& \text{in }\mathbb{R}^n\setminus\Omega,
\end{aligned}
\right.
\end{equation}
where $F$ and $\widetilde{F}$ are determined in \autoref{lem.rep.f_eps}. Moreover, by \eqref{conv.s=1/2:ineq0} and \autoref{lem.est.F}, we have
\begin{align*}
    |F(x)|\leq c\eps d_\eps^{-\frac32}(x).
\end{align*}
By \autoref{thm.green.C^s} together with the fact that $d_\Omega(y)\leq d_\eps(y)$ for any $y\in\Omega$, we have
\begin{align}\label{con.s=1/2:ineq1}
    \|z_{1,\eps}\|_{L^2(\Omega)}^2\leq c\eps^2{\int_\Omega\left(\int_{\Omega}\frac{d_\eps^{-1}(y)}{|x-y|^{n-\frac12}}\,dy\right)^2\,dx}.
\end{align}
Using \eqref{singdist:5} with $\lambda=1/2$, we further estimate
\begin{equation}\label{conv.s=1/2:ineq11}
\begin{aligned}
    \|z_{1,\eps}\|_{L^2(\Omega)}^2&\leq c\eps^2\int_\Omega\left(\int_\Omega \frac{\,dy}{(d_\Omega(y)+\eps)|x-y|^{n-\frac12}}\right)^2\,dx\\
    &\leq c\eps^2\left(\int_{d_\Omega(x)<c\eps} \eps^{-1}(1+|\ln d_\Omega(x)|)^2\,dx+\int_{d_\Omega(x)\geq c\eps}  d_\Omega^{-1}(x) (1+|\ln \eps|)^2\,dx\right)\\
    &\leq c \eps^2(1+|\ln\eps|)^3,
\end{aligned}
\end{equation}
where we have used the fact that $\int_{0<t<\eps}|\ln t|^a\,dt\leq c(a)\eps(1+|\ln \eps|)^a$ for $a\geq1$. Thus, plugging this into \eqref{con.s=1/2:ineq1} yields
\begin{align}\label{conv.s=1/2:z_{1,eps}}
    \|z_{1,\eps}\|_{L^2(\Omega)}\leq c\eps (1+|\ln\eps|)^{3/2}.
\end{align}
Next, note that $\overline{u}_\eps\in C^{2+\alpha}_{\mathrm{loc}}(\Omega^\eps)$ and $h\in C^{1+\alpha}(B_2(\Omega))$ by \autoref{lem.reg.bdd2}. Thus, following the same lines as in the proof of \autoref{lem.est.fracdiv}, but using \eqref{green.C^s:green} in place of \eqref{green:uppbdd}, respectively, we derive
\begin{align*}
    \|z_{\eps}-z_{1,\eps}\|^2_{L^2(\Omega)}&\leq c\eps^2\int_\Omega\left(\int_{\Omega}d_{\eps}^{-\frac32}(y)G_{\eps}(x,y)\,dy\right)^2\,dx\\
    &\quad+c\eps^2\int_{\Omega}\left(\int_{\Omega^{\frac\eps2}} d_\eps^{-1-\frac{1}4}(y)\left(\frac{\eps^{\frac{1}4}+d^{\frac{1}4}_\Omega(y)}{|x-y|^{n-\frac12}}+\frac{d^{\frac12}_\Omega(y)+\eps^{\frac12}}{|x-y|^{n-\frac{1}4}}\right)\,dy\right)^2\,dx+c\eps^2(1+|\ln\eps|)^2,
\end{align*}
where we have used \autoref{lem.fracsob.green} with $s=\beta=\frac12$ and $\alpha=\frac14$.
As in \eqref{conv.s=1/2:ineq11} and the fact that $d_\eps(y)\geq \eps/c$ for $y\in\Omega^{\frac\eps2}\setminus \Omega$, we have 
\begin{align*}
     \|z_{\eps}-z_{1,\eps}\|^2_{L^2(\Omega)}&\leq c\eps^2(1+|\ln\eps|)^3+c\eps^2\int_{\Omega}\left(\int_{\Omega} \frac{(d_\Omega(y)+\eps)^{-1}\,dy}{|x-y|^{n-\frac12}} \right)^2\,dx\\
    &\quad+c\int_\Omega \left(\int_{\Omega^{\frac\eps2}\setminus \Omega}\frac{\,dy}{|x-y|^{n-\frac12}}+\frac{\eps^{\frac14}\,dy}{|x-y|^{n-\frac14}}\right)^2\,dx\eqqcolon L_1+L_2+L_3.
\end{align*}
For $L_2$, we apply \eqref{singdist:5} with $k=1$ and $\lambda=\frac12$ to get
\begin{align*}
    L_2\leq c\eps^2\left(\int_{\{d_\Omega\leq c\eps\}}(1+|\ln d_\Omega(x)|)^{2}\eps^{-1}\,dx+\int_{\{d_\Omega>c\eps\}}(1+|\ln\eps|)^{2}d_\Omega(x)^{-1}\,dx\right)\leq c\eps^{2}(1+|\ln\eps|)^3.
\end{align*}
Next we use \eqref{singdist:4} $\lambda=\frac12$ to estimate $L_3$ as
\begin{align*}
    L_3\leq c\left(\int_{\{d_\Omega\leq c\eps\}}(1+|\ln\eps|)^{2}\eps\,dx+\int_{\{d_\Omega>c\eps\}}\eps^2(1+|\ln\eps|)^{2}d_\Omega(x)^{-1}\,dx\right)\leq c\eps^{2}(1+|\ln\eps|)^3.
\end{align*}
Combining the estimates $L_i$ yields
\begin{align*}
     \|z_{\eps}-z_{1,\eps}\|^2_{L^2(\Omega)}\leq \eps^2(1+|\ln\eps|)^3.
\end{align*}
Therefore, using this, \eqref{con.s=1/2:weps-uubar}, \eqref{con.s=1/2:thetaeps} and \eqref{conv.s=1/2:z_{1,eps}}, we have the desired estimate. This completes the proof.
\end{proof}
Therefore, combining \autoref{lem.conv} and \autoref{lem.con.s=1/2}, we get \autoref{thm:intro:homogenization}.

Lastly, we prove an almost linear convergence rate in $L^1$ when $s\leq\frac12$ and the right-hand side is zero. 
\begin{theorem}\label{conv.sleq12}
   Let $s\in(0,\frac12]$.  Let $\alpha\in(0,1]$, $\eps\in(0,1]$, $r_0,M,N>0$, $K\in C^\alpha(\mathbb{T}^n)$ and $\partial\Omega\in C^{1,\alpha}$ with $C^{1,\alpha}$ radius $r_0$, $C^{1,\alpha}$ norm bounded by $M$, $\mathrm{diam}(\Omega)\leq N$ and $0\in\Omega$.
    Let $u_\eps$ and $\overline{u}$ be weak solutions to 
    \begin{equation*}
\left\{
\begin{aligned}
L_\eps u_\eps &= 0
&\qquad& \text{in }\Omega, \\
u_\eps &= g
&\qquad& \text{in }\mathbb{R}^n\setminus\Omega,
\end{aligned}
\right.
\qquad
\left\{
\begin{aligned}
\overline{L} \overline{u} &= 0
&\qquad& \text{in }\Omega, \\
\overline{u} &= g
&\qquad& \text{in }\mathbb{R}^n\setminus\Omega,
\end{aligned}
\right.
\end{equation*}
where $\overline{K}\coloneqq\int_{\mathbb{T}^n\times \mathbb{T}^n}K(x,x-z)\,dz\,dx$ and $g\in C^{1}(B_2(\Omega))\cap L^1_{2s}(\bbR^n)$.
Then, we have
\begin{equation}\label{conv.sleq12:main}
\begin{aligned}
    \|u_\eps-\overline{u}\|_{L^1(\Omega)}\leq 
        c\eps(1+|\ln\eps|)\left(\|g\|_{C^1(B_2(\Omega)\setminus \Omega)}+\| g\|_{L^1_{2s}(\bbR^n\setminus \Omega)}+\|\nabla g\|_{L^1_{2s}(\bbR^n\setminus \Omega)}\right)
\end{aligned}
\end{equation}
for some constant $c=c(n,s,\Lambda,[K]_{C^\alpha(\mathbb{T}^n)},\alpha,r_0,M,N)$. 
\end{theorem}
\begin{proof}
We may assume 
\begin{equation*}
    \left(\|g\|_{C^1(B_2(\Omega)\setminus \Omega)}+\| g\|_{L^1_{2s}(\bbR^n\setminus \Omega)}+\|\nabla g\|_{L^1_{2s}(\bbR^n\setminus \Omega)}\right)\leq 1
\end{equation*}
and that the constant $\eps$ is sufficiently small. Moreover, we omit the dependence of the constant, as it always depends only on $n,\Lambda,[K]_{C^\alpha(\mathbb{T}^n)},\alpha,r_0,M,N$.  

The proof is essentially the same as in the proof of \autoref{lem.conv} except that we employ \autoref{thm.green.C^s} in place of \autoref{lem.green}, and prove an $L^1$ convergence rate.

Let $\overline{u}_\eps$ be a weak solution to 
    \begin{equation*}
\left\{
\begin{aligned}
\overline{L} \overline{u}_\eps &= 0
&\qquad& \text{in }\Omega^\eps, \\
\overline{u} &= g
&\qquad& \text{in }\mathbb{R}^n\setminus\Omega^\eps.
\end{aligned}
\right.
\end{equation*}
Then we have \eqref{conv:est.ubareps}. Let $w_\eps\coloneqq u_\eps-\overline{u}_\eps-\eps\nabla\overline{u}_\eps\cdot\phi(\cdot/\eps)$, where $\phi$ is determined in \eqref{conv:eq.phipsi}. Then by \autoref{thm.green.C^s} and \eqref{singdist:1}, we have 
\begin{align}\label{conv.sleq12:ineq1}
    \|\overline{u}_\eps-\overline{u}\|_{L^1(\Omega)}\leq c\eps^s\int_\Omega\int_{\Omega^\eps\setminus\Omega}\frac{d_\Omega^{s}(x)}{d^s_\Omega(y)|x-y|^n}\,dy\,dx\leq c\eps\int_\Omega d_\Omega^{-1+s}(x)\,dx\leq c\eps.
\end{align}
Moreover, we have 
\begin{align}\label{conv.sleq12:ineq2}
    \|\eps\nabla\overline{u}_\eps\cdot\phi(\cdot/\eps)\|_{L^1(\Omega)}\leq c\eps.
\end{align}
Next, we consider the weak solutions $\theta_\eps$ and $z_\eps$ given by \eqref{conv:eq.zeps,thetaeps}. As in \textbf{Step 3.1} in the proof of \autoref{lem.conv} with $\beta=s$, we have
\begin{equation*}
\begin{aligned}
    \|\theta_\eps\|_{L^1(\Omega)}\leq c\left(\eps^{s}\int_\Omega\int_{\Omega^\eps\setminus\Omega}\frac{d_\Omega^{s}(x)}{d^s_\Omega(y)|x-y|^n}\,dy\,dx+\eps\int_\Omega\int_{\Omega^{\eps}\setminus \Omega^{\frac\eps2}}\frac{d_\eps^{s-1}(y)d^s_\Omega(x)}{d_\Omega^s(y)|x-y|^n}\,dy\,dx+\eps\right).
\end{aligned}
\end{equation*}
By \eqref{conv.sleq12:ineq1}, the fact that $d_\Omega(y)\geq \eps/c$ in $y\in \bbR^n\setminus\Omega^{\frac\eps2}$ and \eqref{singdist:2}, we have 
\begin{align*}
     \|\theta_\eps\|_{L^1(\Omega)}\leq c\left(\eps+\eps^{1-s}\int_\Omega\int_{\Omega^{\eps}\setminus \Omega^{\frac\eps2}}\frac{d_\eps^{s-1}(y)d_\Omega^s(x)}{d_\Omega^s(y)|x-y|^n}\,dy\,dx\right)\leq c\eps\left(1+\int_{\Omega}d_\Omega^{-1+s}(x)\,dx\right)\leq c\eps.
\end{align*}
 Next, we consider two solutions $z_{1,\eps}$ and $z_{\eps}-z_{1,\eps}$ given by \eqref{conv.s=1/2:eqz1eps}, where $F$ and $\widetilde{F}$ are defined in \autoref{lem.rep.f_eps} with $h=0$. Then, using only the estimate $F_1$ and $F_2$ given in the proof of \autoref{lem.est.F}, we deduce
\begin{align*}
    |F(x)|\leq c\eps d_\eps^{-1-s}(x),
\end{align*}
which implies, by \autoref{thm.green.C^s} and \eqref{singdist:5} together with the fact that $d_\eps(y)\leq c (\eps+d_\Omega(y))$ for any $y\in\Omega$,
\begin{equation}\label{conv.sleq12:ineq3}
\begin{aligned}
    \|z_{1,\eps}\|_{L^1(\Omega)}&\leq c\eps\int_{\Omega}\int_\Omega \frac{\,dy}{(d_\Omega(y)+\eps)|x-y|^{n-s}}\,dy\,dx\\
    &\leq c\eps \left[\int_{d_\Omega(x)\leq c\eps}(1+|\ln d_\Omega(x)|)\eps^{-1+s} \,dx+\int_{d_\Omega(x)>c\eps}(1+|\ln \eps|)d_\Omega^{-1+s}(x)\,dx\right]\\
    &\leq c\eps(1+|\ln\eps|).
\end{aligned}
\end{equation}
Next, as in the proof of \autoref{lem.fracsob.green}, we have 
\begin{align*}
    \|z_{\eps}-z_{1,\eps}\|_{L^1(\Omega)}&\leq c\int_{\Omega}G_\eps(x,y)\eps d_\eps^{-1-s}(x)\,dx\\
    &\quad+c\eps\int_{\Omega}\int_{\Omega^{\frac\eps2}} d_\eps^{-1-\frac{s}2}(y)\left(\frac{\eps^{\frac{s}2}+d^{\frac{s}2}_\Omega(x)}{|x-y|^{n-s}}+\frac{d^s_\Omega(x)+\eps^s}{|x-y|^{n-\frac{s}2}}\right)\,dy \,dx\eqqcolon I_1+I_2.
\end{align*}
As in the estimate of $z_{1,\eps}$, we have $I_1\leq c\eps(1+|\ln\eps|)$. Next, we observe 
\begin{align*}
    I_2&\leq c\left(\int_\Omega\int_{\Omega^{\frac\eps2}\setminus \Omega}\frac{1}{|x-y|^{n-\frac{s}2}}\,dy\,dx+\eps \int_{\Omega}\int_\Omega\frac{1}{(d_\Omega(y)+\eps)|x-y|^{n-\frac{s}2}}\,dy\,dx\right)\eqqcolon I_{2,1}+I_{2,2}.
\end{align*}
By applying \eqref{singdist:4} and \eqref{singdist:5} to $I_{2,1}$ and $I_{2,2}$, respectively, we derive $I_2\leq c\eps(1+|\ln\eps|)$. Thus, we have 
\begin{equation*}
    \|z_\eps-z_{1,\eps}\|_{L^1(\Omega)}\leq c\eps(1+|\ln\eps|).
\end{equation*}
Combining this, \eqref{conv.sleq12:ineq1}, \eqref{conv.sleq12:ineq2} and \eqref{conv.sleq12:ineq3}  yields the desired estimate. This completes the proof.
\end{proof}

\appendix

\section{Auxiliary estimates}\label{sec.appendix}
In this appendix, we recall several useful definitions, introduce the regularized distance function, and provide several technical lemmas. 
\begin{definition}
Let $\alpha\in(0,1)$, $r_0,M>0$. We say that $\Omega$ has $C^{1,\alpha}$ radius $r_0$ and $C^{1,\alpha}$ norm bounded by $M$, if for any $x_0\in \partial\Omega$, there is a function $\varphi_0$ such that 
\begin{align}\label{graph}
    \mathcal{R}(B_{r_0}(x_0)\cap \Omega)=\{z\in B_{r_0}(x_0',0)\,:\, z_n>\varphi_0(z')\}
\end{align}
where $\mathcal{R}$ is a linear function given by a composition of a translation and a rotation, and $\|\nabla\varphi\|_{C^{\alpha}}\leq M$.
\end{definition}

\begin{lemma}\label{lem.exi.reg.dist}
Let $\alpha\in(0,1]$. Let $r_0,M,N>0$ and let $\partial\Omega\in C^{1,\alpha}$-domain with $C^{1,\alpha}$ radius $r_0$, $C^{1,\alpha}$ norm bounded by $M$ and $\mathrm{diam}(\Omega)\leq N$. There is a regularized distance function $\mathrm{\sigma}_\Omega:\bbR^n\to\bbR\geq0$ such that 
\begin{align*}
    \frac{d_\Omega(x)}c\leq \sigma_\Omega(x)\leq cd_\Omega(x)\quad\text{and}\quad \|\nabla\sigma_\Omega\|_{L^{\infty}({\bbR^n})}+\|\sigma_\Omega\|_{C^{1,\alpha}({\Omega})}\leq c
\end{align*}
for some constant $c=c(n,\alpha,r_0,M,N)$. 
\end{lemma}
\begin{proof}
    From \cite[Lemma A.2]{AR20}, there is the weak solution $u$ to 
    \begin{equation*}
\left\{
\begin{alignedat}{3}
(-\Delta)u&=1&&\qquad \mbox{in  $\Omega$}, \\
u&=0&&\qquad  \mbox{in $\partial\Omega$}
\end{alignedat} \right.
\end{equation*}
with 
\begin{align*}
\frac{d_\Omega(x)}c    \leq u(x)\leq cd_\Omega(x)
\end{align*}
for some constant $c=c(n,\alpha,r_0,M,N)$. Since $\Omega$ has the $C^{1,\alpha}$-radius $r_0$, by using the standard flattening argument, we can extend the function $u:\bbR^n\to \bbR$ such that $\|u\|_{C^{1,\alpha}}\leq c$ and 
\begin{align*}
\frac{d_\Omega(x)}c    \leq -u(x)\leq cd_\Omega(x)\quad\text{for any }x\in \bbR^n\setminus \Omega
\end{align*}
for some constant $c=c(n,\alpha,r_0,M,N)$. Thus by taking $\sigma_\Omega(x)\coloneqq |u(x)|$, we get the desired result.
\end{proof}

Using the definition, if the domain $\Omega$ has $C^{1,\alpha}$-radius $r_0>0$ and $C^{1,\alpha}$ norm bounded by $M>0$, then there is a homeomorphism $\Psi_0:B_{r_0}(x_0)\to \Psi_0(B_{r_0}(x_0))\subset B_{Cr_0}(y_0)$ such that 
\begin{equation}\label{defn.hom}
\begin{aligned}
&y_0=\Psi_0(x_0)\quad\text{with } y_{0,n}=0\quad\text{and}\quad \|\nabla\psi_0\|_{C^{\alpha}(B_{r_0}(x_0))}\leq M\\
      & \Psi_0(B_{r_0}(x_0)\cap\Omega)\subset B_{Cr_0}(y_0)\cap \{z_n>0\},\quad \Psi_0(B_{r_0}(x_0)\setminus \Omega)\subset B_{Cr_0}(y_0)\cap \{z_n\leq 0\},\\
        & \frac{|y_1-y_2|}C\leq |\Psi_0(y_1)-\Psi_0(y_2)|\leq C|y_1-y_2|\quad\text{for any }y_1,y_2\in\Psi_0(B_{r_0}(x_0)),\\
        &\frac{|y_n|}{C}\leq d_\Omega(x)\leq C| y_n|\quad\text{for any }x\in B_{r_0}(x_0)\text{ with } y=\Psi_0(x),
\end{aligned}
\end{equation}
for some constant $C=C(n,\alpha,r_0,M,N)$. 

Recall the $\eps$-tubular neighborhood $\Omega^\eps$ by 
\begin{align}\label{defn.app.omegaeps}
    \Omega^\eps\coloneqq \Omega\cup \{x\in\bbR^n\,:\,\sigma_\Omega(x)\leq C_0\eps\}
\end{align}
for some $C_0=C_0(n,\alpha,r_0,M,N)$ such that
\begin{align}\label{app.const.C_0}
    \|\nabla\sigma_\Omega\|_{L^{\infty}(\bbR^n)}\leq \frac{C_0}{16\sqrt{n}}.
\end{align}

Then we have two basic lemmas on the domain $\Omega^\eps$.
\begin{lemma}\label{lem.c1alphadom}
    Let $\alpha\in(0,1]$, $r_0,M,N>0$ and let $\partial\Omega\in C^{1,\alpha}$-domain with $C^{1,\alpha}$ radius $r_0$, $C^{1,\alpha}$ norm bounded by $M$ and $\mathrm{diam}(\Omega)\leq N$. Then there is a small constant $\eps_0=\eps_0(n,\alpha,r_0,M,N)$ such that for any $\eps\leq \eps_0$,  $\Omega^\eps$ has $C^{1,\alpha}$-radius $r_1$ and $C^{1,\alpha}$-norm bounded by $M_1$, where $r_1=r_1(n,\alpha,r_0,M,N)$ and $ M_1=M_1(n,\alpha,r_0,M,N)$.
\end{lemma}
\begin{proof}
We consider the signed regularized distance function $\widetilde{\sigma}_\Omega$ defined by $\widetilde{\sigma}_\Omega(x)={\sigma}_\Omega(x)$ if $x\in\Omega$ and $\widetilde{\sigma}_\Omega(x)=-{\sigma}_\Omega(x)$ if $x\in\bbR^n\setminus \Omega$ to ensure that $\widetilde{\sigma}_\Omega$ is $C^{1,\alpha}$ in $\bbR^n$, where $\sigma_\Omega$ is constructed in \autoref{lem.c1alphadom}. Then, we observe for any $x\in \partial\Omega$, $|\nabla\widetilde{\sigma}_\Omega(x)|\geq \frac1c$ for some constant $c=c(n,\alpha,r_0,M,N)$. In addition, as $\widetilde{\sigma}_\Omega\in C^{1,\alpha}$, there is a small constant $\eps_0=\eps_0(n,\alpha,r_0,M,N)$ such that $|\nabla\widetilde{\sigma}_\Omega(x)|\geq \frac1c$ on $\Omega^{2\eps_0}\setminus \Omega_{2\eps_0}$. Then using the implicit function theorem, there is a constant $r_1=r_1(n,\alpha,r_0,M,N)$ such that for any $\eps\leq \eps_0$ and for each $x_0\in \{x\in\bbR^n\setminus \Omega\,:\,\widetilde{\sigma}_\Omega(x)=-C_0\eps\}=\{x\in\bbR^n\setminus \Omega\,:\,{\sigma}_\Omega(x)=C_0\eps\}=\partial\Omega^{\eps}$, there exists a map $\varphi_0$ satisfying \eqref{graph} with $r_0$ replaced by $r_1$. Moreover, we get $\|\nabla\varphi_0\|_{C^{\alpha}}\leq M_1$ for some constant $M_1=M_1(n,\alpha,r_0,M,N)$. This completes the proof. 
\end{proof}

\begin{lemma}\label{lem.dist.app}
    Let $\alpha\in(0,1]$, $r_0,M,N>0$ and let $\partial\Omega\in C^{1,\alpha}$-domain with $C^{1,\alpha}$ radius $r_0$, $C^{1,\alpha}$ norm bounded by $M$ and $\mathrm{diam}(\Omega)\leq N$. Let $\eps\in(0,1]$ and $\Omega^\eps$ be the domain defined by \eqref{defn.app.omegaeps}. Fix the constant $C_0$ in \eqref{app.const.C_0}. Then we have the following:
    \begin{itemize}
        \item Let $x\in \Omega$. Then $d_\eps(x)\geq 8\sqrt{n}\eps$ and $x-z\in\Omega^{\frac\eps2}$ for any $|z|<8\sqrt{n}\eps$.
        \item Let $x\in \Omega^{\frac\eps2}$. Then $d_\eps(x)\geq 4\sqrt{n}\eps$.
        \item Let $x\in \bbR^n\setminus \Omega^{\frac\eps2}$. Then $d_\Omega(x)\geq 4\sqrt{n}\eps$ and $x-z\in \bbR^n\setminus \Omega^{\frac\eps4}$ for any $z\in2\eps\mathbb{T}^n$.
        \item Let $x-p\in \bbR^n\setminus \Omega^{\frac32\eps}$ for some $p\in\bbR^n$. Then $\eps\mathbb{T}^n(x-p)\subset \bbR^n\setminus \Omega^\eps$.
        \item Let $x-p\in \Omega^{\frac32\eps}\setminus \Omega_{\frac{3}2\eps}$ for some $p\in\bbR^n$. Then $\eps\mathbb{T}^n(x-p)\subset \Omega^{2\eps}\setminus\Omega_{2\eps}$.
        \item Let $x-p\in \Omega_{\frac32\eps}$ for some $p\in\bbR^n$. Then $\eps\mathbb{T}^n(x-p)\subset \Omega_\eps$.

    \end{itemize}
\end{lemma}
\begin{proof}
    We only use \eqref{app.const.C_0} and the triangle inequality.
\begin{itemize}[leftmargin=*]
    \item Let $x\in \Omega$. First, assume $d_\eps(x)<8\sqrt{n}\eps$. Let $x_0\in \partial\Omega^\eps$ such that $d_\eps(x)=|x-x_0|$. Then there is $x_1\in\partial\Omega$ such that $|x-x_1|\leq |x-x_0|$. Thus, we have 
    \begin{align*}
        \sigma_\Omega(x_0)=\sigma_\Omega(x_0)-\sigma_\Omega(x_1)+\sigma_\Omega(x_1)\leq |x_0-x_1|\|\nabla\sigma_\Omega\|_{L^\infty}\leq \frac{C_0\eps}2,
    \end{align*}
    which contradicts $x_0\in \partial\Omega^\eps$, that is, $\sigma_\Omega(x_0)\geq C_0\eps$. Therefore, we get $d_\eps(x)\geq 8\sqrt{n}\eps$.
    
    Let $|z|<8\sqrt{n}\eps$. If $x-z\in \Omega$, it is done. So we may assume $x-z\in\bbR^n\setminus \Omega$. Then there is $x_0\in\partial\Omega$ such that $|x-z-x_0|\leq |z|$ and
    \begin{align*}
        \sigma_\Omega(x-z)=\sigma_\Omega(x-z)-\sigma_\Omega(x_0)+\sigma_\Omega(x_0)\leq \|\nabla\sigma_\Omega\|_{L^\infty}|x-z-x_0|\leq \frac{C_0}{16\sqrt{n}}|z|\leq \frac{C_0}{2}.
    \end{align*}
    This gives $x-z\in \Omega^{\frac\eps2}$.
    \item The second claim follows as in the proof of the first claim.
    \item Let $x\in \bbR^n\setminus \Omega^{\frac\eps2}$. Assume $d_\Omega(x)< 4\sqrt{n}\eps$. Let $x_0\in\partial\Omega$ such that $d_\Omega(x)=|x-x_0|$. Then we have 
    \begin{align*}
        \sigma_\Omega(x)=\sigma_\Omega(x)-\sigma_\Omega(x_0)\leq |x-x_0|\|\nabla\sigma_\Omega\|_{L^\infty}\leq \frac{C_0\eps}{4},
    \end{align*}
    which contradicts $x\in \bbR^n\setminus \Omega^{\frac\eps2}$. In addition, for any $z\in 2\eps\mathbb{T}^n$, 
    \begin{align*}
        \sigma_\Omega(x-z)= \sigma_\Omega(x-z)-\sigma_\Omega(x)+\sigma_\Omega(x)\geq \frac{C_0\eps}{2}- |z|\|\nabla\sigma_\Omega\|_{L^\infty}>\frac{C_0\eps}{4},
    \end{align*}
    which gives $x-z\in\bbR^n\setminus \Omega^{\frac\eps4}$.
    \item Let $x-p\in \bbR^n\setminus \Omega^{\frac32\eps}$ for some $p\in\bbR^n$. Then for any $y\in \eps\mathbb{T}^n(x-p)$,
    \begin{align*}
        \sigma_\Omega(y)=\sigma_\Omega(y)-\sigma_\Omega(x-p)+\sigma_\Omega(x-p)\geq-(\|\nabla\sigma_\Omega\|_{L^\infty})|y-(x-p)|+\frac{3C_0\eps}{2}> C_0\eps,
    \end{align*}
    which implies $y\in \bbR^n\setminus \Omega^\eps$. Thus $\eps\mathbb{T}^n(x-p)\subset \bbR^n\setminus \Omega^\eps$.
    \item The fifth claim follows from the fact that for any $|\xi|\leq \eps\sqrt{n}$
    \begin{align*}
        \sigma_\Omega(x-p-\xi)=\sigma_\Omega(x-p-\xi)-\sigma_\Omega(x-p)+\sigma_\Omega(x-p)\leq 2C_0\eps.
    \end{align*}
    \item The last claim follows from that fact that for $\sigma_\Omega(x-p)\geq\frac32\eps$ and $|\xi|\leq \eps\sqrt{n}$, 
    \begin{align*}
        \sigma_\Omega(x-p-\xi)\geq \sigma_\Omega(x-p)-\frac{C_0}{16\sqrt{n}}|\xi|\geq C_0\eps.
    \end{align*}
\end{itemize}
This completes the proof.
\end{proof}

Next, we provide three auxiliary lemmas concerning distance functions. 
\begin{lemma}\label{lem.singdist}
Let $\alpha\in(0,1]$, $r_0,M,N>0$ and let $\lambda\in[0,1)$, $\beta\in[0,1)$ with $\beta+\lambda>0$. Let $\partial\Omega\in C^{1,\alpha}$ with $C^{1,\alpha}$ radius $r_0$, $C^{1,\alpha}$ norm bounded by $M$ and $\mathrm{diam}(\Omega)\leq N$. Fix $\eps\in(0,1]$ and recall the $\eps$-tubular neighborhood $\Omega^\eps$ defined in \eqref{defn.app.omegaeps}. Then we have the following:
    \begin{itemize}
        \item Let $\lambda\in(0,1)$. For any $x\in\Omega$, we get
    \begin{align}\label{singdist:1}
        \int_{\Omega^1}\frac{\,dy}{d_\Omega^\beta(y)|x-y|^{n-\lambda}}\leq c(1+d_\Omega^{\lambda-\beta}(x)+\mbox{\large$\chi$}_{\{\lambda=\beta\}}|\ln d_\Omega(x)|)
    \end{align}
    where $c=c(n,\alpha,r_0,M,N,\lambda,\beta)$.
    \item For any $x\in\Omega$, we have
    \begin{align}\label{singdist:2}
        \int_{\Omega^{\eps}\setminus \Omega}\frac{\,dy}{d_\Omega^\beta(y)|x-y|^{n-\lambda}}\leq c\eps^{1-\beta}(1+d_\Omega^{-1+\lambda}(x)),
    \end{align}
    where $c=c(n,\alpha,r_0,M,N,\lambda,\beta)$. 
    \item Let $\lambda<\beta$. For any $x\in\Omega$, we have 
 \begin{align}\label{singdist:3}
      \int_{\Omega^{\eps}\setminus \Omega}\frac{\,dy}{d_\Omega^\beta(y)|x-y|^{n-\lambda}}\leq c(\eps^{-\beta+\lambda}+d_\Omega^{-\beta+\lambda}(x)),
 \end{align}
 where $c=c(n,\alpha,r_0,M,N,\lambda,\beta)$
 \item Let $\lambda\in(0,1)$. For any $x\in\Omega$, we have
 \begin{align}\label{singdist:4}
     \int_{\Omega^\eps\setminus \Omega}\frac{\,dy}{|x-y|^{n-\lambda}}\leq\begin{cases}
        c \eps^{\lambda}(1+|\ln \eps|) &\quad\text{if }d_\Omega(x)\leq c\eps,\\
      c \eps(1+|\ln\eps|)d^{-1+\lambda}_\Omega(x) &\quad\text{if }d_\Omega(x)> c\eps,
     \end{cases} 
 \end{align}
 for some constant $c=c(n,\alpha,r_0,M,N,\lambda,\beta)$.
    \item Let $\lambda\in(0,1)$. For any $x\in\Omega$, we have
    \begin{equation}\label{singdist:5}
        \int_\Omega\frac{\,dy}{(d_\Omega(y)+\eps)|x-y|^{n-\lambda}}\leq \begin{cases}
            c(1+|\ln \eps|)d_\Omega^{-1+\lambda}(x)&\quad\text{if }d_\Omega(x)\geq c\eps,\\
            c(1+|\ln d_\Omega(x)|)\eps^{-1+\lambda}&\quad\text{if }d_\Omega(x)<c\eps
        \end{cases} 
    \end{equation}
    for some constant $c=c(n,\alpha,r_0,M,N,\lambda,\beta)$.
    \end{itemize} 
\end{lemma}
\begin{proof}
We omit the dependence of the constant $c$, as it always depends only on the given data $n,\alpha,r_0,M,N,\lambda,\beta$.

Note from \cite[Lemma B.2.4]{FRRO24} that for $\lambda>0$,
\begin{align}\label{singdist:ineq0}
     \int_{\Omega\cap (B_1(x)\setminus B_{\frac{d_\Omega(x)}2}(x))}\frac{1}{d_\Omega^\beta(y)|x-y|^{n-\lambda}}\,dy\leq c\left(1+d_\Omega^{\lambda-\beta}(x)+\mbox{\large$\chi$}_{\{\lambda=\beta\}}|\ln d_\Omega(x)|\right).
\end{align}
In addition, we observe
\begin{equation}\label{singdist:ineq01}
\begin{aligned}
      &\int_{\Omega\setminus B_1(x)}\frac{1}{d_\Omega^\beta(y)|x-y|^{n-\lambda}}\,dy\leq \int_{\Omega\setminus B_1(x)}\frac{1}{d_\Omega^\beta(y)}\,dy\leq c\\
      & \int_{\Omega\setminus B_1(x)}\frac{1}{(d_\Omega(y)+\eps)|x-y|^{n-\lambda}}\,dy\leq \int_{\Omega\setminus B_1(x)}\frac{1}{(d_\Omega(y)+\eps)}\,dy\leq c(|\ln\eps|+1).
\end{aligned}
\end{equation}
On the other hand, for $\lambda>0$ and $\beta\in(0,1]$, we have
\begin{align}\label{singdist:ineq011}
    &\int_{  \Omega\cap B_{\frac{d_\Omega(x)}2}}\frac{\,dy}{d_\Omega^\beta(y)|x-y|^{n-\lambda}}\leq  d_\Omega(x)^{\lambda-\beta}.
\end{align}
Combining this, \eqref{singdist:ineq0} and \eqref{singdist:ineq01} yields
\begin{align}\label{singdist:ineq1}
     \int_{\Omega}\frac{\,dy}{d_\Omega^\beta(y)|x-y|^{n-\lambda}}\leq c(1+d_\Omega^{\lambda-\beta}(x)).
\end{align}
Next, we estimate the term $J$ defined by 
    \begin{align}\label{singdist:defn.J}
        J\coloneqq \int_{\Omega^\eps\setminus \Omega}\frac{\,dy}{d^\beta_\Omega(y)|x-y|^{n-\lambda}}
    \end{align}
    for any $\eps\in(0,1]$ and $k\in\{0,1\}$. We may assume $r_0\leq \frac14$. Let $x_0\in\partial\Omega$ such that $d_\Omega(x)=|x-x_0|$. We split the estimate of $J$ into two cases. 
    \begin{itemize}
        \item Assume $x\notin B_{\frac{r_0}2}(x_0)$. Then  $|y-x|\geq d_\Omega(x)=|x-x_0|\geq\frac{r_0}{2}$ for any $y\in \Omega^\eps\setminus \Omega$. Therefore, we have 
        \begin{align}\label{singdist:est.Jeasy}
            J\leq c\int_{\Omega^\eps\setminus \Omega}\frac{\,dy}{d^\beta_\Omega(y)}\leq c\int_0^{c\eps}\frac{\,dt}{t^{\beta}} \leq c\eps^{1-\beta}.
        \end{align}
        \item Assume $x\in B_{\frac{r_0}2}(x_0)$. We observe
        \begin{align*}
            J=\int_{(\Omega^\eps\setminus \Omega)\setminus B_{r_0}(x_0)}(\cdot)+\int_{(\Omega^\eps\setminus \Omega)\cap B_{r_0}(x_0)}(\cdot)\eqqcolon J_1+J_2.
        \end{align*}
        For $J_1$, we observe for any $y\in (\Omega^\eps\setminus \Omega)\setminus B_{r_0}(x_0)$, $|y-x|\geq |y-x_0|-|x-x_0|\geq \frac{r_0}{2}$, which gives 
        \begin{align}\label{singdist:est.J_1}
            J_1\leq c\int_0^{c\eps}\frac{\,dt}{t^{\beta}} \leq c\eps^{1-\beta}.
        \end{align}
        Next, for $J_2$, we recall the function $\Psi_0$ satisfying \eqref{defn.hom}. By the change of variables $z=\Psi(y)$, we have 
        \begin{align*}
            J_2\leq c\int_{B_{Cr_0}(y_0)\cap \{-C\eps\leq z_n\leq0\}}\frac{\,dz}{|z_n|^{\beta}|X-z|^{n-\lambda}},
        \end{align*}
        for some constant $C=C(n,\alpha,r_0,M,N)$, where $y_0=\Psi(x_0)(=(y_0',0))$ and $X=\Psi(x)$. Specifically, we have used \eqref{defn.hom} to see that for any $y\in(\Omega^\eps\setminus \Omega)\cap B_{r_0/2}(x_0)$, $|\Psi(x)-\Psi(y)|\leq c|x-y|$, $z_n\leq0$ and $|(\Psi(y))_n|=|z_n|\eqsim d_\Omega(y)$.  We further estimate $J_2$ as
        \begin{align*}
            J_2 \leq c\int_{-C\eps}^0\frac{1}{|z_n|^{\beta}}\int_{B'_{Cr_0}(y_0')}\frac{\,dz'}{(|X'-z'|+|X_{n}|+|z_n|)^{n-\lambda}}\,dz_n\leq c\int_{-C\eps}^0\frac{\,dz_n}{|z_n|^{\beta}(|X_n|+|z_n|)^{1-\lambda}}.
        \end{align*}
        Now we will divide the proof into several cases to get \eqref{singdist:1}-\eqref{singdist:4}.
        \begin{itemize}
            \item Let $\eps=1$, $k=0$. If $0<\lambda\leq \beta$, we estimate $J_2$ as
            \begin{align*}
                J_2\leq c\int_{-|X_n|}^0\frac{\,dz_n}{|z_n|^{\beta}|X_n|^{1-\lambda}}+\int_{-C}^{-|X_n|}\frac{\,dz_n}{|z_n|^{\beta}|z_n|^{1-\lambda}}&\leq c(|X_n|^{\lambda-\beta}+\mbox{\large$\chi$}_{\lambda=\beta}|\ln|X_n|)\\
                &\leq c(d_\Omega^{\lambda-\beta}(x)+\mbox{\large$\chi$}_{\{\lambda=\beta\}}|\ln d_\Omega(x)|)
            \end{align*}
            where $c=c(n,r_0,M,\lambda,\beta)$. If $\lambda>\beta$, then
            \begin{align*}
                J_2\leq \int^{0}_{-C}\frac{\,dz_n}{|z_n|^{1-\lambda+\beta}}\leq c.
            \end{align*} Combining this, \eqref{singdist:est.Jeasy}, \eqref{singdist:est.J_1} and \eqref{singdist:ineq1} proves \eqref{singdist:1}.
            \item Let $k=0$. Then we have
        \begin{align*}
            J_2\leq c\int_{-C\eps}^0\frac{\,dz_n}{|z_n|^{\beta}(|X_n|+|z_n|)^{1-\lambda}}\leq  c\eps^{1-\beta}|X_n|^{-1+\lambda} \leq c\eps^{1-\beta}d_\Omega(x)^{-1+\lambda}.
        \end{align*}
        where $c=c(n,r_0,M,\lambda,\beta)$. Thus, a combination of this, \eqref{singdist:est.Jeasy} and \eqref{singdist:est.J_1} proves \eqref{singdist:2}. 
        \item Let $k=0$ and $\lambda<\beta$. When $\lambda<\beta$ and $|X_n|\leq C\eps$, we have 
        \begin{align*}
            J_2&\leq \int^{0}_{-|X_n|}\frac{\,dz_n}{|z_n|^\beta(|X_n|+|z_n|)^{1-\lambda}}+\int^{-|X_n|}_{-C\eps}\frac{\,dz_n}{|z_n|^\beta(|X_n|+|z_n|)^{1-\lambda}}\\
            &\leq c\left(\int^{0}_{-|X_n|}\frac{\,dz_n}{|z_n|^\beta|X_n|^{1-\lambda}}+\int^{-|X_n|}_{-C\eps}\frac{\,dz_n}{|z_n|^{1-\lambda+\beta}} \right)\\
            &\leq c(|X_n|^{\lambda-\beta}+\eps^{\lambda-\beta}).
        \end{align*}
        In addition, when $|X_n|>C\eps$, we have 
        \begin{align*}
            J_2\leq c\int_{0}^{-C\eps}\frac{\,dz_n}{|z_n|^\beta \eps^{1-\lambda}}\leq c\eps^{\lambda-\beta}.
        \end{align*}
        Combining the above two inequalities, \eqref{singdist:est.Jeasy} and \eqref{singdist:est.J_1}, we get \eqref{singdist:3}.
        \item Let $\lambda>0$ and $\beta=0$. Then
        if $|X_n|\leq C\eps$, then
        \begin{align*}
            J_2&\leq c\left(\int_{-|X_n|}^0\frac{\,dz_n}{(|X_n|+|z_n|)^{1-\lambda}}+\int^{-|X_n|}_{-C\eps}\frac{\,dz_n}{(|X_n|+|z_n|)^{1-\lambda}}\right)\leq c\eps^\lambda .
        \end{align*}
        Next, if $|X_n|>C\eps$, we have 
        \begin{align*}
            J_2\leq c\int_{-C\eps}^0\frac{\,dz_n}{(|X_n|+|z_n|)^{1-\lambda}}\leq c|X_n|^{-1+\lambda}\eps \leq cd_\Omega^{-1+\lambda}(x)\eps.
        \end{align*}
        Next, suppose 
        Thus, combining two cases yields \eqref{singdist:4}.
        \end{itemize}
    \end{itemize}
Now, it remains to prove \eqref{singdist:5}. For this, we are going to prove that for $\rho=\max\{\eps,d_\Omega(x)\}$ and $\lambda>0$,
\begin{align}\label{singdist:ineq02}
    I&\coloneqq\int_{\Omega\cap (B_1(x)\setminus B_{\frac{\rho}2}(x))}\frac{1}{(d_\Omega(y)+\eps)|x-y|^{n-\lambda}}\,dy\leq c(1+|\ln\eps|)\rho^{-1+\lambda}(x).
\end{align}
As in the proof of \cite[Lemma B.2.4]{FRRO24}, we deduce
\begin{align*}
    I\leq c\left(1+\sum_{i=1}^{l}2^{i(n-\lambda)}\int_{(B_{2^{-i+1}}(x)\setminus B_{2^{-i}}(x))\cap \{0<d_\Omega< c2^{-i}\}}(d_\Omega(y)+\eps)^{-1}|\nabla d_\Omega(y)|\,dy\right),
\end{align*}
where $l$ is the positive integer such that $2^{-l}\leq \frac{\rho}4\leq 2^{-l+1}$.
Rewriting $I$ and applying the Coarea formula, we have
\begin{align*}
    I&=c\left(1+\sum_{i=1}^{l}2^{i(n-\lambda)}\int_{(B_{2^{-i+1}}(x)\setminus B_{2^{-i}}(x))\cap \{0<d_\Omega< c2^{-i}\}}(d_\Omega(y)+\eps)^{-1}|\nabla d_\Omega(y)|\,dy\right)\\
    &\leq c\left(1+\sum_{i=1}^{l}2^{i(n-\lambda)}\int_{0}^{c2^{-i}}\frac{1}{t+\eps}\,dt\int_{(B_{2^{-i+1}}(x)\setminus B_{2^{-i}}(x))\cap \{d_\Omega=t\}}\,dH^{n-1}\right)\\
    &\leq c\left(1+\sum_{i=1}^{l}2^{i(1-\lambda)}(1+|\ln\eps|)\right)\\
    &\leq c(1+|\ln\eps|)\rho^{-1+\lambda},
\end{align*}
where we have also used the following estimate
\begin{align*}
    H^{n-1}(\{d_\Omega=t\}\cap (B_{2^{-i+1}}\setminus B_{2^{-i}}))\leq c2^{-i(n-1)},
\end{align*}
as $\{d_\Omega=t\}$ is Lipschitz for small $t$.

If $d_\Omega(x)\geq\eps$, then using \eqref{singdist:ineq02}, \eqref{singdist:ineq01} and \eqref{singdist:ineq011}, we have \eqref{singdist:5}.
Let $d_\Omega(x)<\eps$. Then there is a positive integer $l$ such that $2^{l-1}d_\Omega(x)<\eps\leq 2^ld_\Omega(x)$ we have 
\begin{align*}
    \int_{\Omega\cap(B_\eps(x)\setminus B_{\frac{d_\Omega(x)}2}(x))}\frac{\,dy}{(d_\Omega(y)+\eps)|x-y|^{n-\lambda}}&\leq \sum_{i=0}^{l}\int_{(B_{2^id_\Omega(x)}(x)\setminus B_{2^{i-1}d_\Omega(x)}(x))\cap \Omega}\frac{\,dy}{(d_\Omega(y)+\eps)|x-y|^{n-\lambda}}\\
    &\leq \sum_{i=0}^l\int_{B_{2^id_\Omega(x)}(x)\cap \Omega}\frac{\,dy}{\eps(2^id_\Omega(x))^{n-\lambda}}\\
    &\leq \frac1\eps\sum_{i=0}^l(2^id_\Omega(x))^{\lambda}i.
\end{align*}
Next, as $2^l\leq 2\eps/d_\Omega(x)$, we have 
\begin{equation*}
    l\leq c(|\ln \eps|+|\ln d_\Omega(x)|+1)\leq c(1+|\ln d_\Omega(x)|).
\end{equation*}
Using the above two inequalities, we have
\begin{align*}
     \int_{B_\eps(x)\setminus B_{\frac{d_\Omega(x)}2}(x)}\frac{\,dy}{(d_\Omega(y)+\eps)|x-y|^{n-\lambda}}\leq c\eps^{-1+\lambda}l\leq c\eps^{-1+\lambda}(1+|\ln d_\Omega(x)|).
\end{align*}
Therefore, combining this, \eqref{singdist:ineq02}, \eqref{singdist:ineq01} and the fact that 
\begin{equation*}
    \int_{\Omega\cap B_{\frac{d_\Omega(x)}2}(x)}\frac{\,dy}{(d_\Omega(y)+\eps)|x-y|^{n-\lambda}}\leq c\eps^{-1}\int_{ B_{\eps}(x)}\frac{\,dy}{|x-y|^{n-\lambda}}\leq c\eps^{-1+\lambda}
\end{equation*}
yields \eqref{singdist:5}. Thus, we have proved \eqref{singdist:5} for both cases $d_\Omega(x)>\eps$ and $d_\Omega(x)\leq \eps$. This completes the proof.

Therefore, combining all the cases, the first, second and third claim follow. 
\end{proof}
\begin{lemma}\label{lem.dist.aux2}
    Let $\eps\in(0,1]$ and let $x\in \Omega$. Recall the domain $\Omega^\eps$ given in \eqref{defn.omegaeps} and write $d_\eps(x)\coloneqq \mathrm{dist}(x,\partial\Omega^\eps)$. Then we have the following:
    \begin{itemize}
        \item Let $|z|\leq\frac{d_\eps(x)}{2}$. Then for any $t\in[0,1]$, we have 
        \begin{align}\label{dist.aux2:main1}
            \frac12d_\eps(x)\leq d_\eps(x-tz)\leq \frac32d_\eps(x).
        \end{align}
        \item Let $x-z\in \bbR^n\setminus \Omega^{\eps}$. Then we have 
        \begin{align}\label{dist.aux2:main2}
            |z|\geq\frac{d_\eps(x)}{2}
        \end{align}
        \item Let $x-z\in \Omega$ with $z\in \eps\mathbb{T}^n(p)$ for some $p\in\bbR^n$.
        Then for any $t\in[0,1]$
        \begin{align}\label{dist.aux2:main5}
            \frac{d_\eps(x-z)}2\leq d_\eps(x-(tz+(1-t)p))\leq \frac32 d_\eps(x-z)\quad\text{for any }t\in[0,1].
        \end{align}
    \end{itemize}
\end{lemma}
\begin{proof}
    The proof is based on elementary computations.
    \begin{itemize}
        \item Let $|z|<\frac{d_\eps(x)}{2}$ and $t\in[0,1]$. First, there is a point $x_0\in\partial\Omega^\eps$ such that $|x-x_0|=d_\eps(x)$. Then,
        \begin{align*}
            d_\eps(x-tz)\leq |x-tz-x_0|\leq |x-x_0|+|z|\leq \frac{3}2d_\eps(x).
        \end{align*}
        Next, there is a point $z_t\in \partial\Omega^\eps$ such that $|x-tz-z_t|=d_\eps(x-tz)$. Then
        \begin{align*}
            d_\eps(x)\leq |x-z_t|\leq |x-tz-z_t|+|tz|\leq d_\eps(x-tz)+\frac{d_\eps(x)}2.
        \end{align*}
 Combining the two above inequalities yields \eqref{dist.aux2:main1}.
        \item Let $x-z\in \bbR^n\setminus \Omega^{\eps}$. Assume $|z|<\frac{d_\eps(x)}{2}$. Then $x-z\in {B_{\frac{d_\eps(x)}2}}(x)$. However, $B_{\frac{d_\eps(x)}2}(x)\subset \Omega^{\eps}$, which contradicts $x-z\notin \bbR^n\setminus \Omega^{\eps}$.
        \item Let $x-z\in \Omega$ with $z\in \eps\mathbb{T}^n(p)$ for some $p\in\bbR^n$. Note that $d_\eps(x-z)\geq 8\sqrt{n}\eps\geq 8|z-p|$ by \autoref{lem.dist.app}. Thus, by applying \eqref{dist.aux2:main1} with $x$, $z$ and $t$ replaced by $x-z$, $z-p$ and $1-t$, respectively, we have the desired result.
    \end{itemize}
    This completes the proof.
\end{proof}

\begin{lemma}\label{lem.dist.tailint}
  Let $\alpha\in(0,1]$, $r_0,M,N>0$. Let $\partial\Omega\in C^{1,\alpha}$ with $C^{1,\alpha}$ radius $r_0$, $C^{1,\alpha}$ norm bounded by $M$ and $\mathrm{diam}(\Omega)\leq N$. There is a small constant $\eps_0=\eps_0(n,\alpha,r_0,M,N)$ such that for any $\eps\leq \eps_0$, the following holds: Let $x\in \Omega$, $\gamma\in(0,1)$, $\beta\geq0$, $\delta\in[0,1)$ and $d_\eps\equiv d_{\Omega^\eps}$.
    \begin{itemize}
    \item We have
     \begin{align}\label{dist.tailint:main3}
        \int_{|z|>\frac{d_\eps(x)}4}\frac{d_\eps^{-\delta}(x-z)}{|z|^{n+\beta}}\mbox{\large$\chi$}_{\{x-z\in \Omega_{\eps}\}}\,dz\leq cd^{-\delta-\beta}_\eps(x)
    \end{align}
    for some constant $c=c(n,\alpha,r_0,M,N,\beta,\delta)$.
    \item We have 
    \begin{align}\label{dist.tailint:main4}
        \int_{|z|>\frac{d_\eps(x)}{4}}\frac{d_\eps^{-\delta}(x-z)}{|z|^{n+\beta}}\mbox{\large$\chi$}_{\{x-z\in \Omega^{2\eps}\setminus \Omega_{2\eps}\}}\,dz\leq c\eps^{1-\delta} d_\eps^{-1-\beta}(x)
    \end{align}
    for some constant $c=c(n,\alpha,r_0,M,N,\beta,\delta)$.
    \end{itemize}
\end{lemma}
\begin{proof}
First, by \autoref{lem.c1alphadom}, there is a small constant $\eps_0=\eps_0(\Omega)$ such that for any $\eps\leq \eps_0$, $\Omega^{\eps}$ has $C^{1,\alpha}$ radius $r_1$ and $C^{1,\alpha}$ norm bounded by $M_1$, where $r_1$ and $M_1$ depend only on $n,\alpha,r_0,M,N$. Moreover, for any $x_0\in \partial\Omega^\eps$, there is a homeomorphism $\Psi_0$ having \eqref{defn.hom} with $r_0=r_1$, $M=M_1$.

Based on the flattening argument given in the proof of  \cite[Lemma 2.4]{CCKW24}, we prove for any $C_0>0$,
\begin{align}\label{dist.tailint:main2}
    I_k\coloneqq \int_{2^{k-1}d_\eps(x)<|z|\leq 2^kd_\eps(x),x-z\in\Omega^{2\eps}}d_\eps^{-\delta}(x-z)\mbox{\large$\chi$}_{\{d_\eps(x-z)\leq C_0\}}\,dz\leq c C_0^{1-\delta}(2^kd_\eps(x))^{n-1}
\end{align}
for some constant $c=c(n,\alpha,r_0,M,N,\delta)$. 

First, if $C_0\geq \frac{r_1}{4}$, then we have 
\begin{align*}
    I_k\leq \int_{2^{k-1}<|z|\leq 2^kd_\eps(x),x-z\in\Omega^{2\eps}}d_\eps^{-\delta}(x-z)\,dz\leq c\min\{(2^kd_\eps(x))^{n-1},1\}\leq c C_0^{1-\delta}(2^kd_\eps)^{n-1}
\end{align*}
for some constant $c=c(n,\alpha,r_0,M,N,\delta)$, as $\delta<1$ and $x-z\in\Omega^{2\eps}$. In addition, if $2^kd_\eps(x)\geq \frac{r_1}{4}$, then
\begin{align*}
    I_k\leq \int_{x-z\in\Omega^{2\eps}}d_\eps^{-\delta}(x-z)\mbox{\large$\chi$}_{\{d_\eps(x-z)\leq C_0\}}\,dz\leq cC_0^{1-\delta}\leq c C_0^{1-\delta}(2^kd_\eps)^{n-1}
\end{align*}
for some constant $c=c(n,s,\alpha,r_0,M,N,\delta)$.

Thus, we assume $C_0,2^kd_\eps(x)\leq \frac{r_1}{4}$. Using the change of variables $x-z=y$, we rewrite $I_k$ as
   \begin{align*}
       I_k=\int_{2^{k-1}d_\eps(x)<|y-x|<2^{k}d_\eps(x),y\in \Omega^{2\eps}}d_\eps^{-\delta}(y)\mbox{\large$\chi$}_{\{d_\eps(y)\leq C_0\}}\,dy
   \end{align*}
   Fix $y_1\in \Omega^{2\eps}$ such that $d_\eps(y_1)\leq C_0$ and $|y_1-x|<2^kd_\eps(x)$. Then there is a point $y_0\in \partial\Omega^\eps$ such that $|y_1-y_0|=d_\eps(y_1)$. Indeed, we have 
   \begin{align*}
       A_0=\{|y-x|<2^{k}d_\eps(x)\,:\, d_\eps(y)\leq C_0\}\subset B_{r_0}(y_0),
   \end{align*}
   as for $y\in A_0$,
   \begin{align*}
       |y-y_0|\leq |y-x|+|x-y_1|+|y_1-y_0|\leq 2^kd_\eps(x)+2^kd_\eps(x)+d_\eps(y_1)\leq 3r_1/4.
   \end{align*}
   As $\Omega^\eps$ has $C^{1,\alpha}$ radius $r_1$, there is a homeomorphism $\Psi:B_{r_1}(y_0)\to \Psi_0(B_{r_1}(y_0))\subset B_{C r_1}(y_0)$ that satisfies \eqref{defn.hom} with $r_0=r_1$, $M=M_1$ for some $C=C(n,\alpha,r_0,M,N)$. Now using the change of variables $z=\Psi(y)$ and \eqref{defn.hom}, we estimate $I_{k}$ as
   \begin{align*}
       I_k\leq c\int_{ B_{2^kCd_\eps(x)}(\Psi(x))\cap \{-C C_0\leq z_n\leq C C_0\}}\frac{\,dz}{|z_n|^{\delta}}\leq cC_0^{1-\delta}(2^kd_\eps(x))^{n-1},
   \end{align*}
   where $\Psi(x)=(\Psi(x)',0)$, $c=c(n,r_0,M,N,\beta,\delta)$ such that $\Psi(B_{r_1}(y_0) \cap\{d_\eps(y)\leq C_0\})\subset B_{C r_1}(y_0)\cap \{-C C_0\leq z_n\leq CC_0\}$ and $\Psi(B_{2^kd_\eps}(x))\subset B_{C2^kd_\eps(x)}(\Psi(x)) $. Therefore, we have \eqref{dist.tailint:main2}.

For \eqref{dist.tailint:main3}, using the fact that $d_\eps(x-z)\leq d_\eps(x)+|z|$ and \eqref{dist.tailint:main2}, we have
\begin{align*}
   &\sum_{k=0}^{\infty}\int_{2^{k-2}d_\eps(x)<|z|\leq 2^{k-1}{d_\eps(x)}, x-z\in\Omega}\frac{d_\eps^{-\delta}(x-z)}{|z|^{n+\beta}}\,dz\\
   &\leq c\sum_{k=0}^{\infty}\int_{2^{k-2}d_\eps(x)<|z|\leq 2^{k-1}{d_\eps(x)},x-z\in\Omega}\frac{d_\eps^{-\delta}(x-z)}{(2^{k}d_\eps(x))^{n+\beta}}\mbox{\large$\chi$}_{\{d_\eps(x-z)\leq 2^{k}d_\eps(x)\}}\,dz\\
   &\leq c\sum_{k=0}^\infty(2^kd_\eps(x))^{-\beta-\delta}
\end{align*}
for some constant $c=c(n,\alpha,r_0,M,N,\beta,\delta)$, which proves \eqref{dist.tailint:main3}.

For \eqref{dist.tailint:main4}, we observe $d_\eps(x-z)\leq C\eps$ for some $c=C(n,\alpha,r_0,M,N)$, as $x-z\in \Omega^{2\eps}\setminus \Omega_{2\eps}$. Thus, using \eqref{dist.tailint:main2}, we conclude
   \begin{align*}
        \int_{|z|>\frac{d_\eps(x)}{4}}\frac{d_\eps^{-\delta}(x-z)}{|z|^{n+\beta}}\mbox{\large$\chi$}_{\{x-z\in \Omega^{2\eps}\setminus \Omega_{2\eps}\}}\,dz\leq \sum_{k\geq0}I_k (2^kd(x))^{n+\beta}\leq c\eps^{1-\delta}d_\eps(x)^{-1-\beta}
   \end{align*}
   for some constant $c=c(n,\alpha,r_0,M,N,\beta,\delta)$,
   which completes the proof. 
\end{proof}

Lastly we give some basic properties of the functional given by for $R>0$,
\begin{align*}
    (u:\bbR^n\to \bbR,\phi:\bbR^n\to\bbR^n)\Rightarrow\inf\limits_{a\in\bbR,\,b\in\bbR^n}{E}(u-a-b\cdot(x+\phi);B_R),
\end{align*}
where $E(\cdot)$ is defined by \eqref{defn.excess}.

\begin{lemma}\label{lem.funct}
Let $s>\frac12$ and $R>0$. We have the following:
\begin{itemize}
    \item For any $u\in  L^1_{2s}(\bbR^n)$ and $\phi\in L^\infty(\bbR^n;\bbR^n)$ with $\|\phi\|_{L^\infty}\leq M$ for some constant $M\geq1$, there exist $a_0\in\bbR$ and $b_0\in\bbR^n$ such that 
    \begin{align}\label{funct:exist} 
         \inf_{a\in\bbR,\,b\in\bbR^n}{E}(u-a-b\cdot(x+\phi);B_R)= {E}(u-a_0-b_0\cdot(x+\phi);B_R).
    \end{align}
    In addition, we have 
    \begin{align}\label{funct:bdd.a_0b_0}
        |a_0|+|b_0|\leq c(n,s,R,M)E(u;B_R)
    \end{align}
    \item Assume that there are sequences $(\eps_l)_l$ $(\phi_l)_l$ and $(u_l)_l$ such that
    \begin{align*}
        \|\phi_l\|_{L^\infty(\bbR^n)}\leq \eps_l\quad\text{and}\quad u_l\in L^2(\bbR^n)\quad\text{with}\quad u_l\equiv 0\quad\text{on}\quad\bbR^n\setminus B_{M},
    \end{align*}
    where $\lim\limits_{l\to\infty}\eps_l=0$ and $M\geq 2R$. Let $a_l\in\bbR$ and $b_l\in\bbR^n$ be given by
    \begin{align*}
        \inf_{a\in\bbR,\,b\in\bbR^n}{E}(u_l-a-b\cdot(x+\phi_l);B_R)={E}(u_l-a_l-b_l\cdot(x+\phi_l);B_R).
    \end{align*}
    If $u_l\to u_\infty$ in $L^2(\bbR^n)$, $a_l\to a_\infty$, $b_l\to b_\infty$, then
    \begin{align}\label{funct:limit}
      \inf_{a\in\bbR,\,b\in\bbR^n}{E}(u_\infty-a-b\cdot x;B_R)=\lim_{l\to\infty}E(u_l-a_l-b_l\cdot(x+\phi_l);B_R)=E(u_\infty-a_\infty-b_\infty\cdot x;B_R).
    \end{align}
\end{itemize}
\end{lemma}

\begin{proof}
First, we observe
\begin{align*}
    \inf_{a\in\bbR,\,b\in\bbR^n}{E}(u-a-b\cdot(x+\phi);B_R)\leq E(u;B_R).
\end{align*}
Thus, there are sequences $(a_k)_k$ and $(b_k)_k$ such that
\begin{align*}
\inf_{a\in\bbR,\,b\in\bbR^n}{E}(u-a-b\cdot(x+\phi);B_R)=\lim_{k\to\infty}E(u-a_k-b_k\cdot(x+\phi);B_R).
\end{align*}
Moreover, by taking $k$ large enough, we may assume
\begin{align*}
    E(u-a_k-b_k\cdot(x+\phi);B_R)\leq E(u;B_R)+C_k
\end{align*}
for some constant $|C_k|\leq1$ and $C_k\to0$ as $k\to\infty$. We then get
\begin{align*}
    E(a_k-b_k\cdot(x+\phi);B_R)\leq 2E(u;B_R)+C_k.
\end{align*}
Thus, we deduce
\begin{align*}
    \int_{B_{4M}\cap \{\left|\frac{b_k}{|b_k|}\cdot x\right|>2M\}}|a_k+b_k\cdot(x+\phi)|\leq c(n,s,R,M)(E(u;B_R)+C_k),
\end{align*}
which implies 
\begin{align}\label{funct:akbk}
    |a_k|+|b_k|\leq c(n,s,R,M)(E(u;B_R)+C_k),
\end{align}
where the constant $c$ is independent of $k$. Up to subsequence, there are $a_0$ and $b_0\in\bbR^n$ such that
\begin{align}\label{funct:limit.akbk} 
    \lim_{k\to\infty}a_k= a_0\quad\text{and}\quad\lim_{k\to\infty}b_k= b_0.
\end{align}
Now using the fact that $||u|-|v||\leq |u-v|$, we have 
\begin{align*}
    |E(u-a_0-b_0\cdot(x+\phi);B_R)-E(u-a_k-b_k\cdot(x+\phi);B_R)|&\leq E(a_0-a_k-(b_0-b_k)\cdot(x+\phi);B_R)\\
    &\leq c(n,s,R,M)(|a_0-a_k|+|b_0-b_k|),
\end{align*}
which implies the existence of $a_0\in\bbR$ and $b_0\in\bbR^n$ such that
\begin{align*}
    \inf_{a\in\bbR,\,b\in\bbR^n}{E}(u-a-b\cdot(x+\phi);B_R)=E(u-a_0-b_0\cdot(x+\phi);B_R).
\end{align*}
Moreover, \eqref{funct:akbk} and \eqref{funct:limit.akbk} together with the fact that $C_k\to 0$ as $k\to\infty$ yields \eqref{funct:bdd.a_0b_0}.

We now prove the second claim. To this end, for any $\delta>0$, there is a large integer $l_0$ such that if $l\geq l_0$, then
    \begin{align}\label{funct:seq}
        \|u_l-u_\infty\|_{L^2(\bbR^n)}+|a_l-a_\infty|+|b_l-b_\infty|+\eps_l\leq \delta.
    \end{align}
    Let $a_0\in\bbR$ and $b_0\in\bbR^n$ satisfy
    \begin{align*}
        \inf_{a\in\bbR,\,b\in\bbR^n}{E}(u_\infty-a-b\cdot x;B_R)={E}(u_\infty-a_0-b_0\cdot x;B_R).
    \end{align*}
    Using \eqref{funct:seq} and $\|\phi_l\|_{L^\infty}\leq \eps_l$, and defining $M_l\coloneqq E(u_l-a_l-b_l\cdot (x+\phi_l);B_R)$, we have
    \begin{align*}
       &E(u_\infty-a_l-b_l\cdot x;B_R)\leq E(u_l-a_l-b_l\cdot (x+\phi_l);B_R)+E(u_l-u_\infty;B_R)+E(b_l\cdot \phi_l;B_R)\leq M_l+c\delta,\\
       &E(u_l-a_0-b_0\cdot x;B_R)\leq E(u_\infty-a_0-b_0\cdot x;B_R)+E(u_l-u_\infty;B_R)\leq E(u_\infty-a_0-b_0\cdot x;B_R)+c\delta
    \end{align*}
    for some constant $c=c(n,s,R)$ independent of $l$. Next by the choices of $a_l,b_l,a_0,b_0$, we derive
    \begin{align*}
       &{E}(u_\infty-a_0-b_0\cdot x;B_R)\leq  E(u_\infty-a_l-b_l\cdot x;B_R)\leq M_l+c\delta,\\
       &M_l\leq E(u_l-a_0-b_0\cdot (x+\phi_l);B_R)\leq E(u_\infty-a_0-b_0\cdot x;B_R)+c\delta,
    \end{align*}
    which implies
    \begin{align*}
 \inf_{a\in\bbR,\,b\in\bbR^n}{E}(u_\infty-a-b\cdot x;B_R)={E}(u_\infty-a_0-b_0\cdot x;B_R)=\lim_{l\to\infty}E(u_l-a_l-b_l\cdot (x+\phi_l);B_R).
    \end{align*}
    Lastly, we have 
    \begin{align*}
        |E(u_l-a_l-b_l\cdot (x+\phi_l);B_R)-E(u_\infty-a_\infty-b_\infty\cdot x;B_R)|&\leq E(u_l-u_\infty;B_R)+E(a_l-a_\infty;B_R)\\
        &\quad+ E((b_l-b_\infty)\cdot x;B_R)+E(b_l\phi_l;B_R).
    \end{align*}
    Using \eqref{funct:seq} and the fact that $\|\phi_l\|_{L^\infty}\leq \eps_l$, we have the desired result \eqref{funct:limit}. This completes the proof.
\end{proof}

\bibliographystyle{alpha}
\bibliography{literature}

\end{document}